\documentclass{article}
\usepackage{bm}
\usepackage[all]{xy}
\usepackage{latexsym,amssymb,amsfonts,amsthm}

\title{A predicative variant of Hyland's Effective Topos}
\author{Maria Emilia Maietti and Samuele Maschio\\Dipartimento di Matematica, Universit\`a di Padova, Italy\\maietti/maschio@math.unipd.it}
\date{}

\newtheorem{theorem}{Theorem}[section]
\newtheorem{lemma}[theorem]{Lemma}
\newtheorem{corollary}[theorem]{Corollary}
\theoremstyle{definition}
\newtheorem{definition}[theorem]{Definition}
\newtheorem{remark}[theorem]{Remark}

\def\noteps{\mathrel{\!\not\mathrel{\,\overline{\varepsilon}\!}\,}}
\newcommand{\mtt}{\mbox{{\bf mTT}}}

\newcommand{\tset}{{\mathbf{Set}^{r}}}
\newcommand{\thprops}{\overline{\mathbf{Prop}^{r}_{s}}}
\newcommand{\thprop}{\overline{\mathbf{Prop}^{r}}}

\newcommand{\tar}{\mbox{$\widehat{ID_1}$}}
\newcommand{\eff}{{\mathbf{Eff} }}
\newcommand{\peff}{\mathbf{pEff}}
\newcommand{\mset}{{\mathbf{pEff}_{set}}}
\newcommand{\mcol}{{\mathbf{pEff}_{col}}}
\newcommand{\mprop}{{\mathbf{pEff}_{prop}}}
\newcommand{\mprops}{\mathbf{pEff}_{prop_s}}
\newcommand{\mf}{{\bf MF}}

\begin{document}
\maketitle
\begin{abstract}
Here, we present a subcategory $\peff$ of Hyland's Effective Topos $\eff$
 which can be considered a predicative
variant  of  $\eff$ itself.

 The construction of $\peff$ is motivated
by the desire of providing
a ``predicative'' categorical universe of realizers to model the Minimalist Foundation  for constructive mathematics which was ideated by the first author with G. Sambin
in 2005 and completed into a two-level formal system by the first author in 2009.

$\peff$  is a  ``predicative'' categorical universe because its objects and morphisms
can be formalized in Feferman's predicative weak theory of inductive definitions $\tar$.

Moreover, it  is a predicative variant of the Effective Topos for the following reasons.

First,
 $\peff$ is a list-arithmetic locally cartesian closed pretopos of definable
objects in $\tar$  with a  fibred category of small objects over $\peff$
and a (non-small) classifier of small subobjects. 

Second, it happens to coincide with the exact completion on the lex
category   defined as a predicative rendering in $\tar$
of the subcategory of $\eff$ of recursive functions. As a consequence
it validates the Formal Church's thesis and it embeds  in $\eff$
 by preserving the list-arithmetic locally cartesian closed pretopos  structure.



Keywords: Realizability topos,  Lawvere's hyperdoctrines, Type theory \\
AMS classification: 03F50,18D30,18C99, 03D70
\end{abstract}
\section{Introduction}
As reported in \cite{story} Hyland's paper ``The Effective Topos'' \cite{eff}
 gave input to a whole new strand of research about realizability and its applications
to logic, mathematics and computer science.  Hyland applied the tripos-to-topos
construction in \cite{tripos}  by producing one of the first examples
of  elementary topos, denoted with $\eff$,  that is not a Grothendieck topos.
A characteristics of $\eff$ which attracted a lot of interest in logic and computer science relies on the fact that $\eff$ provides a realizability interpretation
of high-order logic that extends Kleene realizability semantics
of intuitionistic arithmetic and hence validates the formal Church's thesis (see \cite{vanOOsten}).

A predicative study of $\eff$, and more generally of realizability toposes,
had been already developed in the context of algebraic set theory by B. Van den Berg and I. Moerdijk,  in particular in \cite{VDB},
by taking Aczel's Constructive Zermelo-Fraenkel set theory (for short CZF) in \cite{czf} as the predicative constructive set theory to be realized in such toposes.

The authors of the present paper embarked into the project of making
a strictly predicative version of $\eff$, called $\peff$, which can be formalized
in Feferman's  weak theory of inductive definitions $\tar$ (see e.\,g.\,\cite{Fef})  whose proof-theoretic strenght is much lower than that of CZF.
 Our ultimate goal is to use $\peff$ as the effective universe
where to validate proofs done in the Minimalist Foundation in order to extract
their computational contents in terms of programs. 

The Minimalist Foundation, for short $\mf$,
is a predicative foundation for constructive mathematics ideated  in joint work of the first author with G. Sambin in \cite{mtt}
and completed to a two-level system in \cite{m09}.

$\mf$ is weaker than CZF (in terms of proof-theoretic strenght)  because $\mf$ can be interpreted
in Martin-L{\"o}f's type theory with one universe \cite{PMTT} as shown in \cite{m09}, and directly also  in Feferman's predicative
weak theory of inductive definitions \cite{Fef} as shown in \cite{MMM}.

 $\mf$ was called {\it minimalist} in \cite{mtt} because
it was intended to  constitute a common core among the most relevant constructive and classical foundations.

 One of the novelties of $\mf$ is that it consists of two levels with an interpretation of one into the other: 
an {\it intensional level} which should be a (type) theory with enough decidable properties to be a  base for a proof-assistant and to extract the computational contents from its proofs,  an {\it extensional level} formulated in a language as close as possible to that of ordinary mathematics, and an interpretation
of the extensional level in the intensional one by means of a quotient completion. Both the intensional level and the extensional level of  $\mf$
in \cite{m09} consists of dependent type systems based on versions of Martin-L\"of's type theory: the intensional one is based on \cite{PMTT}
and the extensional one on \cite{ML84}. 

A key difference between $\mf$ and the mentioned Martin-L\"of's type theories
is that in $\mf$ propositions are defined in such a way that choice principles, including the axiom of unique choice, are no longer necessarily valid in order to make
$\mf$ minimalist.

It is worth noting that the presence of two levels in $\mf$ is  relevant
to show its compatibility  both with intensional theories such as
 those formulated in type theory like Martin-L\"of's type theory~\cite{PMTT} or Coquand's Calculus of Constructions, or with extensional theories
such as those formulated in axiomatic set theory like Aczel's CZF axiomatic set theory, or those arising  in category theory like the internal theories of topoi or pretopoi.

Moreover, there is an analogy between  the two-level formal system of $\mf$ and 
the tripos-to-topos construction of a realizability topos: the role of the tripos is taken by the intensional level of $\mf$, the role of the realizability topos construction is taken by the quotient model
used in \cite{m09} to interpret the extensional level of $\mf$ in its intensional one, and
the internal language of a generic elementary topos corresponds to the extensional level of 
$\mf$.

The key difference between the tripos-to-topos construction
and the construction of $\mf$ is that the quotient completion employed in $\mf$
does  not  yield to an exact category.
The quotient completion of $\mf$ had been studied categorically in joint work of the first author with G. Rosolini in \cite{qu12}, \cite{elqu}
 under the name of ``elementary quotient completion of a Lawvere's elementary doctrine''. Such a completion turned out to be  a {\it generalization of the  well-known notion of
 exact completion on a lex category}. Instead,  the tripos-to-topos construction
is an instance of a  generalization
of the exact on regular completion related to an existential elementary
doctrine as shown in \cite{uxc}.

Here we build a predicative variant of $\eff$, called $\peff$,
by applying  the elementary quotient completion in \cite{qu12}
to 
 a Lawvere's hyperdoctrine
$$\thprop :\mathcal{C}_{r}^{op}\rightarrow \mathbf{Heyt}$$ 
whose logical structure extends the well-known Kleene realizability interpretation of intuitionistic connectives and supports
an interpretation (whilst with a pure combinatory non categorical interpretation of $\lambda$-abstraction) of the intensional level of \mf\
 extended
 with the  formal Church's thesis and the full axiom of choice (see \cite{IMMSt}). Then, from results in \cite{m09}, it  follows that  
$\peff$ validates the extensional level of \mf\ as desired.

$\peff$ can be seen as a predicative variant of $\eff$ for the following reasons.

First, $\peff$ is a predicative rendering of a topos  because
it is a list-arithmetic locally cartesian closed pretopos with
 a  fibred category of small objects over itself and a (non-small) classifier of small subobjects, somewhat  in the spirit of algebraic set theory
employed in \cite{VDB}.
The perculiarity of $\peff$ is that its structure of small objects 
is given in a fibred way via an indexed category of
 small objects in the starting realizability doctrine $\thprop$.
In turn small objects of $\thprop$ are defined via fixpoints of suitable monotone operators representing codes of sets and small propositions definable in the intensional level of the Minimalist Foundation.
 The fact that the classifier of small subobjects in $\peff$ is not small
is the key point to view $\peff$ as predicative.
Even more, the property that $\peff$ is formalizable in \tar\ makes the construction
{\it strictly predicative} as much as the theory \tar.


Second,
the elementary quotient completion construction used to build  $\peff$ happens to actually coincide with the exact completion on the lex base  category $\mathcal{C}_{r} $ of the doctrine $\thprop$  which is a predicative rendering  of the subcategory of $\eff$ of recursive functions
 within Feferman's predicative weak theory of inductive definitions $\tar$.
Therefore, recalling that  $\eff$ is the exact completion of partitioned assemblies,
since  the realizability doctrine $\thprop$ embeds
in the weak subobject doctrine of partitioned assemblies,
 it follows that $\peff$ embeds into $\eff$ by preserving the list-arithmetic locally cartesian closed  structure. As a consequence $\peff$ validates
the Formal Church's thesis.


As a further future work we  intend to
develop an abstract notion of predicative
topos that has $\peff$ and any elementary topos as examples
and the extensional level of \mf\ as its
internal language.

\section{How we build our predicative variant of $\eff$}
It is well known that in a predicative foundation the power-collection of subsets of a non-empty set can not be a set but only a proper collection. Therefore in order to build a predicative version of the topos $\eff$  we need to distinguish {\it sets and collections} in the universe. This is the key difference employed in the context of
algebraic set theory in \cite{swopos} to make a categorical model of a constructive predicative
set theory as Aczel's CZF~\cite{czf} with the use of ``small maps'' to denote
a family of sets indexed on a collection. This notion of categorical model
 had been
 in turn used to produce a predicative version
of realizability toposes, and hence also of $\eff$, in \cite{VDB}.

Here we want to make a predicative version of $\eff$ by working within a strictly predicative
theory as Feferman's weak theory of inductive definitions \tar. The ultimate goal
is to build a predicative effective universe where to validate
 (the extensional level of) the 
Minimalist Foundation \mf\ in order to extract programs from its constructive proofs. Hence, as in \mf,  in our predicative  we take into account also a further distinction besides that between sets and collections, namely 
 among generic propositions we distinguish ``small propositions'' as those
propositions closed under intuitionistic connectives and
 quantifiers restricted to sets only.
 Indeed, as in \mf,  we want to use the notion of small proposition
to define the notion of subset of a set and that of subsets classifier
in such a way that the {\it subobject classifier typical of a topos
becomes in our predicative universe a {\em collection} classifying subobjects
 defined as comprehensions of small propositions}.

In order to build our strictly predicative variant of $\eff$, that we name $\peff$, we proceed as follows:
\begin{list}{-}{ }
\item
We define 
a category $ \mathcal{C}_{r}$ of ``realized collections'' which
is a predicative rendering  in Feferman's weak theory of inductive definitions $\tar$ of the subcategory of recursive functions of $\eff$~\cite{MRC}.

\item
We build an indexed category of realized collections
$$ \mathbf{Col}^{r}:\mathcal{C}_{r}^{op}\rightarrow \mathbf{Cat}$$
 which provides a splitting 
of the codomain fibration associated to the category of realized collections
$\mathcal{C}_{r}$.

\item
We define an indexed category of ``realized propositions''
$$\mathbf{Prop}^r :\mathcal{C}_{r}^{op}\rightarrow \mathbf{Cat}$$
 as the  preordered reflection of $\mathbf{Col}^{r}$.
In particular, the fibres of $\mathbf{Prop}^r$
  extend
Kleene interpretation
of intuitionistic connectives so that $\mathbf{Prop}^r$ validates the Formal Church's thesis. 

\item
Within $\mathcal{C}_{r}$
we construct an object $\mathsf{U_{S}}$, which is intended as the universe of sets, and  which is defined as a class of realizers in $\tar$
representing
codes of sets definable in the intensional level
$\mtt$ of the Minimalist Foundation. We adapt here a similar
  construction made
in \cite{MMM} where we gave an extensional realizability interpretation of $\mtt$ in $\tar$.
Based on the internal notion of set we build an indexed category of sets
$$\tset:\mathcal{C}_{r}^{op}\rightarrow \mathbf{Cat}$$
which will be a sub-indexed category of $ \mathbf{Col}^{r}$.

\item
We build an indexed category
of small propositions
 $$\mathbf{Prop}^{r}_s:\mathcal{C}_{r}^{op}\rightarrow \mathbf{Heyt}$$
 as the preordered reflection of $\tset$.

\item
We then call {\it Effective Kleene \mf-tripos} the structure given
by the indexed category of realized collections, its sub-indexed category of realized sets, the hyperdoctrine  of realized propositions and its sub-doctrine
of small realized propositions.
\item
Finally, we define our effective universe $\peff$ as the elementary quotient
completion in \cite{qu12} with respect to the hyperdoctrine $\thprop$ whose fibres are the posetal reflection of those of $\mathbf{Prop}^{r}$. 
As a consequence, 
$\peff$ turns out to be closed under stable effective quotients with respect to an hyperdoctrine 
$$\mprop:\peff^{op}\longrightarrow{\bf Heyt}$$
obtained by lifting $\thprop$ to $\peff$ as described in \cite{qu12}.
\item
Within $\peff$ we single out an hyperdoctrine of small propositions:
$$\mprops:\peff^{op}\longrightarrow{\bf Heyt}$$
whose fibres are descent datas of small propositions in $\thprops$.
\item 
we show that in $\peff$ there is an object $\Omega$ classifying
subobjects obtained as comprehensions of small propositions, or in other terms
we show that  the hyperdoctrine
of small propositions in $\peff$ is representable in $\peff$ by $\Omega$, i.e. 
$$\mprops(-)\simeq \peff(-, \Omega)$$
\item 
Within $\peff$ we single out a fibration of families of realized sets whose fibres are quotients of indexed sets in the fibres of $\tset$ over small equivalence relations in the fibres  of $\thprops$.
\end{list}
We conclude by saying  that $\peff$ will turn out to be an exact on lex completion of the category of realized collections $\mathcal{C}_{r}$ and as a consequence
it embeds into $\eff$.

\section{Categorical preliminaries}
We just recall some categorical definitions we will use in the next.
\begin{definition}
 \begin{enumerate}
\item A \emph{weak exponential} (rel.\ \emph{exponential}) for $A$ and $B$ in a category  $\mathcal{C}$ with binary products is given by an object $B^{A}$ and an arrow $ev:B^{A}\times A\rightarrow B$ such that for every arrow $f:C\times A\rightarrow B$ in $\mathcal{C}$, there is an (rel.\ unique) arrow $f':C\rightarrow B^{A}$ in $\mathcal{C}$ for which the following diagram commutes:

$$\xymatrix@-1pc{
C\times A\ar[r]^{f}\ar[d]_-{f'\times id_{A}}			&B\\
B^{A}\times A\ar[ru]_-{ev}\\
}$$
\item A \emph{parameterized list object} for an object $A$ of $\mathcal{C}$ is given by an object $List(A)$ and two arrows $\epsilon:\mathsf{1}\rightarrow List(A)$ and $conc:List(A)\times A\rightarrow List(A)$ such that for every pair of arrows $f:P\rightarrow B$ and $g:B\times A\rightarrow B$ there is a unique arrow $listrec(f,g):P\times List(A)\rightarrow B$ for which the following diagram commute

$$\xymatrix@-1pc{
P\ar[rr]^-{\langle id_{P},\epsilon\circ !_{P}\rangle}\ar[rrdd]_-{f}	 &	&P\times List(A)\ar[dd]^-{listrec(f,g)}		&	&P\times (List(A)\times A)\ar[ll]_-{id_{P}\times conc }\ar[d]^-{ \langle \langle \pi_{1},\pi_{1}\circ \pi_{2}\rangle, \pi_{2}\circ \pi_{2}\rangle }\\
									&			&				&		&(P\times List(A))\times A\ar[d]^{listrec(f,g)\times id_{A}}\\
									&			&B				&		&B\times A\ar[ll]^{g}\\
}$$
where $\pi_{1}$ and $\pi_{2}$ denote left and right projections respectively.

\end{enumerate}
\end{definition}
 



 \begin{definition}
 A \emph{(weakly) cartesian closed} category is a finite product category with (weak) exponentials for all pairs of objects.

 A \emph{(weakly) locally cartesian closed} category is a category whose slices are all (weakly) cartesian closed categories.

\end{definition}

\begin{definition}\label{hyp}
A \emph{first-order hyperdoctrine} is a functor $\mathbf{P}:\mathcal{C}^{op}\rightarrow \mathbf{preHeyt}$ from a finite product category $\mathcal{C}$ to the category of Heyting prealgebras \footnote{A Heyting prealgebra is a preorder whose posetal reflection is a Heyting algebra. Morphisms between them are preorder morphisms which preserve the Heyting algebra structure on the posetal reflections (see e.\,g.\ \cite{VOO08}).} such that for every $f$ in $\mathcal{C}$, the morphism of Heyting prealgebras $\mathbf{P}_{f}$ has left and right adjoints $\exists_{f}$ and $\forall_{f}$ (in the category of preorders) satisfying the Beck-Chevalley condition i.\,e.\ if the following diagram is a pullback diagram in $\mathcal{C}$,

$$\xymatrix@-1pc{
P\ar[r]^{p_{2}}\ar[d]_{p_{1}}	&B\ar[d]^{g}	\\
A\ar[r]_{f}					&C\\
}$$

then $\mathbf{P}_{g}\circ\exists_{f}$ is equivalent to $\exists_{p_{2}}\circ \mathbf{P}_{p_{1}}$, i.\,e.\ for every $a\in \mathbf{P}(A)$, we have that both $\mathbf{P}_{g}(\exists_{f}(a))\sqsubseteq \exists_{p_{2}}(\mathbf{P}_{p_{1}}(a))$ and $\mathbf{P}_{g}(\exists_{f}(a))\sqsupseteq \exists_{p_{2}}(\mathbf{P}_{p_{1}}(a))$ hold in $\mathbf{P}(B)$. From this it also follows that $\mathbf{P}_{g}\circ\forall_{f}$ is equivalent to $\forall_{p_{2}}\circ\mathbf{P}_{p_{1}}$.
\end{definition}

\begin{definition}\label{weacomp}
A first-order hyperdoctrine $\mathbf{P}:\mathcal{C}^{op}\rightarrow \mathbf{preHeyt}$ has \emph{weak comprehensions} if for every object $A$ of $\mathcal{C}$ and for every $p\in \mathbf{P}(A)$, there exists an arrow $\mathsf{cmp}:\mathsf{Cmp}(p)\rightarrow A$ in $\mathcal{C}$ such that $\top\sqsubseteq \mathbf{P}_{\mathsf{cmp}}(p)$ in $\mathbf{P}(\mathsf{Cmp}(p))$ and for every arrow $f:B\rightarrow A$ such that $\top\sqsubseteq \mathbf{P}_{f}(p)$ in $\mathbf{P}(B)$ there exists an arrow $f':B\rightarrow \mathsf{Cmp}(p)$ such that $\mathsf{cmp}\circ f'=f$.
\end{definition}

\begin{definition}\label{preorder}
If $\mathbf{C}:\mathcal{C}^{op}\rightarrow \mathbf{Cat}$ is an indexed category, then  its  \emph{preorder reflection} is called $\mathbf{P}[\mathbf{C}]:\mathcal{C}^{op}\rightarrow \mathbf{Cat}$ and it is the functor whose fibres are the preordered  reflections  of the fibres of $\mathbf{C}$, while
its  \emph{posetal reflection} is called $\mathbf{PR}[\mathbf{C}]:\mathcal{C}^{op}\rightarrow \mathbf{Cat}$ and   it is the functor whose fibres are the posetal reflections  of those of $\mathbf{C}$.

\end{definition}

\begin{definition}\label{fcc}
Given a finitely complete category $\mathcal{C}$,  we denote by $\mathbf{wSub}_{\mathcal{C}}$ the posetal reflection of the slice pseudofunctor which we call
{\it the doctrine of weak subobject of $\mathcal{C}$}.
\end{definition}

\section{Feferman's weak theory of inductive definitions $\tar$}\label{tar}
Here we are going to give a brief description of Feferman's weak theory of inductive definitions (see e.\,g.\,\cite{Fef}).

Consider the language of second-order arithmetic given by a countable list of individual variables $x_{1},...,x_{n}...$, a countable list of set variables $X_{1},...,X_{n}...$, a constant $0$, a unary successor functional symbol $succ$, an $n$-ary functional symbol for every $n$-ary (definition of a) primitive recursive function, the equality predicate $=$ between individuals, the membership predicate $\epsilon$ between individuals and sets, connectives $\wedge, \vee, \rightarrow, \neg$ and individual and set quantifiers $\exists x$, $\forall x$, $\exists X$ and $\forall X$.

In particular atomic formulas of this language are $t=s$ and $t\,\epsilon\, X$ for individual terms $t$ and $s$ and set variables $X$.

Let $X$ be a set variable; its occurrence in the atomic formula $t\,\epsilon\, X$ is positive, while an occurrence of $X$ in a non-atomic formula $\varphi$ is \emph{positive} (\emph{negative} resp.\,) if one of the following conditions holds:

\begin{enumerate}
\item $\varphi$ is $\psi\wedge \rho$ or $\psi\vee \rho$ and the occurrence is positive (negative) in $\psi$ or in $\rho$;
\item $\varphi$ is $\psi\rightarrow \rho$ and the occurrence is positive (negative) in $\rho$ or it is negative (positive) in $\psi$;
\item $\varphi$ is $\neg \psi$ and the occurrence is negative (positive) in $\psi$;
\item $\varphi$ is $\exists x\, \psi$ or $\forall x\,\psi$ and the occurrence is positive (negative) in $\psi$;
\item $\varphi$ is $\exists X\, \psi$ or $\forall X\,\psi$ and the occurrence is positive (negative) in $\psi$.
\end{enumerate}

A second-order formula $\varphi(x,X)$ is \emph{admissible}  if it does not contain set quantifiers and it has at most one free individual variable $x$ and at most one free set variable $X$ and all the occurrences of the variable $X$ are positive. 

Let's now define the system $\tar$. 
It is a first-order classical theory whose language has a countable list of individual variables $x_{1},...,x_{n}...$, a constant $0$, a unary successor functional symbol $succ$, an $n$-ary functional symbol for every $n$-ary (definition of a) primitive recursive function, a unary predicate symbol $P_{\varphi}$ for every admissible second-order formula $\varphi(x,X)$, the equality predicate $=$, connectives $\wedge, \vee, \rightarrow, \neg$ and quantifiers $\exists$ and $\forall$. \footnote{We define $\bot$ and $\top$ as abbreviations for $0=succ(0)$ and $0=0$ respectively.}

The axioms of $\tar$ include the axioms of Peano arithmetic (including defining equations for primitive recursive functions) plus the following axiom schemas:
\begin{enumerate}
\item \emph{Induction principle} for every formula of the language of $\tar$;
\item \emph{Fixpoint schema}: for every admissible second-order formula $\varphi(x,X)$, $$P_{\varphi}(x)\leftrightarrow \varphi(x,P_{\varphi})$$ 
\end{enumerate}
where $\varphi(x,P_{\varphi})$ is the formula of $\tar$ obtained by substituting in $\varphi(x,X)$ all the subformulas $t\,\epsilon\, X$ for some $t$ with $P_{\varphi}(t)$.

\subsection{Notation of recursive functions}
In the weak theory of inductive definition $\tar$, one can encode Kleene's application $\{x\}(y)$ via a Kleene predicate $T(x,y,z)$ and a primitive recursive function $U(z)$: for every formula $P(z)$, $P(\{x\}(y))$ is an abbreviation for $\exists z\,(T(x,y,z)\wedge P(U(z)))$.

 We define the abbreviations $\{t\}(s_{1},...,s_{n})$ as follows:
\begin{enumerate}
\item $\{t\}(\,)$ is $t$;
\item $\{t\}(s_{1},...,s_{n+1})$ is $\{\{t\}(s_{1},...,s_{n})\}(s_{n+1})$.
\end{enumerate}
In $\tar$ one can encode a bijective product of natural numbers $p$ with its projections $p_{1}$, $p_{2}$ using primitive recursive functions which can be represented by numerals $\mathbf{p}_{1},\mathbf{p}_{2}, \mathbf{p}$ such that 
\begin{enumerate}
\item $\tar\vdash p(x,y)=\{\mathbf{p}\}(x,y)$, 
\item $\tar\vdash p_{i}(x)=\{\mathbf{p}_{i}\}(x)$ for $i=1,2$.
\end{enumerate}
There exists a numeral $\mathbf{ite}$ such that 
\begin{enumerate}
\item $\tar\vdash\{\mathbf{ite}\}(0,x,y)=x$
\item $\tar\vdash\{\mathbf{ite}\}(1,x,y)=y$
\end{enumerate}

Moreover one can surjectively encode finite lists of natural numbers $\overline{x}=[x_{0},...,x_{lh(\overline{x})-1}]$ in such a way that $0$ encode the empty list, the concatenation function is primitive recursive and it is represented by a numeral $\mathbf{cnc}$, the length function is primitive recursive and there is a numeral $\mathbf{listrec}$ representing the list recursor. Moreover the successor function can be represented by a numeral $\mathbf{succ}$ and the natural numbers recursor can be represented by a numeral $\mathbf{rec}$.

In $\tar$ one can also define $\lambda$-astraction $\Lambda x.t$ of terms built with Kleene brackets, variables and numerals as in any partial combinatory algebra. 

\section{The Effective Kleene \mf-tripos}
It is well known that Hyland's Effective Topos $\eff$ can be obtained
both as the tripos-to-topos completion~\cite{eff,tripos, vanOOsten}
of a tripos as well as
 the exact completion on the lex category of partioned assemblies \cite{MRC}.

In \cite{elqu,uxc}
it has been shown how the above completions can be seen as instances of   suitable
quotient completions with respect to a Lawvere's elementary doctrine (possibly with further structure).
In essence all {\it exact completions can be seen as a result of a  {\bf tripos-to-quotient completion} process}.

In particular the exact completion on a lex category $\mathcal{C}$ can be seen as
 an instance of the so called  
{\it elementary quotient completion} construction introduced in \cite{qu12}  and applied to the Lawvere's elementary  doctrine of weak subobjects on $\mathcal{C}$.

It is worth noting that
the notion of elementary quotient completion is more general than that of exact completion on a lex category since it does not necessarily
yield to an exact category unless a choice rule is satisfied by the 
starting doctrine (see \cite{MR16}).

Here, we are going to introduce a Lawvere's hyperdoctrine
$$\thprop :\mathcal{C}_{r}^{op}\rightarrow \mathbf{Heyt}$$ 
of {\it realized propositions} on a category of {\it realized collections $\mathcal{C}_{r}$ }
on which we perform an elementary quotient completion which happens
to be exact and which we take as our predicative variant of $\eff$
under the name of  $\peff$.

The  doctrine $\thprop$ will be part of a richer structure, which will be called  {\it Effective Kleene \mf-tripos}.

\subsection{The category $\mathcal{C}_r$ of realized collections}
Here we  define the category $\mathcal{C}_r$ of realized collections which is a rendering
of the subcategory of recursive functions of the Effective Topos (see for ex.\cite{MRC}) in  $\tar$. 
Then, we will describe its categorical structure.

 In particular to show that this category is weakly locally cartesian closed 
 we found easier to  introduce an indexed category 
 which gives a functorial presentation of the pullback pseudofunctor
associated to $\mathcal{C}_{r}$ and hence provides a splitting
to the codomain fibration of $\mathcal{C}_r$.
The existence of  such a splitting   is due to the syntactic nature
of our category $\mathcal{C}_r$ on whose objects we can  express a notion of family of realized collections  by using formulas of $\tar$
where the functorial action corresponds to a term substitution in the considered
formulas.


\begin{definition}\label{col}
A \emph{realized collection} (or simply a \emph{class}) $\mathsf{A}$ of $\tar$ is a formal expression $$\{x|\,\varphi_{\mathsf{A}}(x)\}$$ where $\varphi_{\mathsf{A}}(x)$ is a formula of $\tar$ with at most $x$ as free variable. 
We write $x\,\varepsilon\, \mathsf{A}$ as an abbreviation for $\varphi_{\mathsf{A}}(x)$. Classes with provably equivalent membership relations in $\tar$ are identified.

An \emph{operation} between classes of $\tar$ from $\mathsf{A}$ to $\mathsf{B}$ is an equivalence class $[\mathbf{n}]_{\approx_{\mathsf{A},\mathsf{B}}}$ of numerals with
$$x\,\varepsilon\,\mathsf{A}\vdash_{\tar}\{\mathbf{n}\}(x)\,\varepsilon\,\mathsf{B}$$
The equivalence relation $\approx_{\mathsf{A},\mathsf{B}}$ is defined as follows:
$$\mathbf{n}\approx_{\mathsf{A},\mathsf{B}} \mathbf{m}\textrm{ if and only if }x\,\varepsilon \,\mathsf{A}\vdash_{\tar} \{\mathbf{n}\}(x)=\{\mathbf{m}\}(x)$$
If $[\mathbf{n}]_{\approx_{\mathsf{A},\mathsf{B}}}:\mathsf{A}\rightarrow \mathsf{B}$ and $[\mathbf{m}]_{\approx_{\mathsf{B},\mathsf{C}}}:\mathsf{B}\rightarrow \mathsf{C}$ are operations between classes of $\tar$, then their \emph{composition} is defined by $$[\mathbf{m}]_{\approx_{\mathsf{B},\mathsf{C}}}\circ [\mathbf{n}]_{\approx_{\mathsf{A},\mathsf{B}}}:=[\Lambda x. \{\mathbf{m}\}(\{\mathbf{n}\}(x))]_{\approx_{\mathsf{A},\mathsf{C}}}:\mathsf{A}\rightarrow\mathsf{C}$$
The \emph{identity} operation for the class $\mathsf{A}$ of $\tar$ is defined as $$\mathsf{id}_{\mathsf{A}}:=[\Lambda x.x]_{\approx_{\mathsf{A},\mathsf{A}}}:\mathsf{A}\rightarrow \mathsf{A}$$
\end{definition}

\begin{lemma}
Realized collections of $\tar$ and operations between them with their composition and identity operations form a category which we will denote with $\mathcal{C}_{r}$ and we will call the category of realized collections of $\tar$. 
\end{lemma}

From now on, we will omit subscripts of $\approx$ when they will be clear from the context.

Now, we show that
$\mathcal{C}_{r}$ is finitely complete with finite coproducts, weak exponentials, 
and parameterized list objects. 

\begin{theorem}\label{t1} $\mathcal{C}_{r}$ is a finitely complete category with finite coproducts, parameterized list objects and weak exponentials. 
\end{theorem}
\begin{proof} 
In $\mathcal{C}_{r}$ the following hold:
\begin{enumerate}
\item $\mathsf{1}:=\{x|\,x=0\}$ is a terminal object;
\item $\mathsf{A}\times \mathsf{B}:=\{x|\,p_{1}(x)\,\varepsilon\, \mathsf{A}\,\wedge\,  p_{2}(x)\,\varepsilon\,\mathsf{B}\}$ defines a binary product for $\mathsf{A}$ and $\mathsf{B}$ together with the projections defined as $\mathsf{p}_{1}:=[\mathbf{p}_{1}]_{\approx}:\mathsf{A}\times \mathsf{B}\rightarrow \mathsf{A}$ and  $\mathsf{p}_{2}:=[\mathbf{p}_{2}]_{\approx}:\mathsf{A}\times \mathsf{B}\rightarrow \mathsf{B}$;

\item if $\mathsf{f}=[\mathbf{n}]_{\approx}$ and $\mathsf{g}=[\mathbf{m}]_{\approx}$ are arrows in $\mathcal{C}_{r}$ from $\mathsf{A}$ to $\mathsf{B}$, then their equalizer is given by the class  $\mathsf{Eq}(\mathsf{f},\mathsf{g}):=\{x|\,x\,\varepsilon\,\mathsf{A}\,\wedge\, \{\mathbf{n}\}(x)=\{\mathbf{m}\}(x)\}$ together with the arrow $[\Lambda x.x]_{\approx}:\mathsf{Eq}(\mathsf{f},\mathsf{g})\rightarrow \mathsf{A}$;

\item $\mathsf{0}:=\{x|\,\bot\}$ is an initial object;

\item $\mathsf{A}+\mathsf{B}:=\{x|\;( p_{1}(x)=0\,\wedge\,  p_{2}(x)\,\varepsilon\,\mathsf{A})\vee ( p_{1}(x)=1\,\wedge \, p_{2}(x)\,\varepsilon\,\mathsf{B})\}$ gives a binary coproduct for $\mathsf{A}$ and $\mathsf{B}$ together with the injections $\mathsf{j}_{1}:=[\Lambda x.\{\mathbf{p}\}(0,x)]_{\approx}$ from $\mathsf{A}$ to $\mathsf{A}+\mathsf{B}$ and $\mathsf{j}_{2}:=[\Lambda x.\{\mathbf{p}\}(1,x)]_{\approx}$ from $\mathsf{B}$ to $\mathsf{A}+\mathsf{B}$;
\item $\mathsf{List}(\mathsf{A}):=\{x|\,\forall j(\,j<lh(x)\rightarrow (x)_{j}\,\varepsilon\,\mathsf{A})\}$ defines a parameterized list object for $\mathsf{A}$ together with the empty list arrow $\epsilon:=[\Lambda x.0]_{\approx}:\mathsf{1}\rightarrow \mathsf{List}(\mathsf{A})$ and the append arrow $\mathsf{cons}$ defined as $$[\Lambda x. \{\mathbf{cnc}\}(\{\mathbf{p}_{1}\}(x),\{\mathbf{p}_{2}\}(x))]_{\approx}:\mathsf{List}(\mathsf{A})\times \mathsf{A}\rightarrow\mathsf{List}(\mathsf{A})$$
\item $\mathsf{A}\Rightarrow\mathsf{B}:=\{x|\,\forall u\, (u\,\varepsilon\,\mathsf{A}\, \rightarrow\, \{x\}(u)\,\varepsilon\,\mathsf{B})\}$ defines a weak exponential 
for $\mathsf{A}$ and $\mathsf{B}$ together with the evaluation arrow defined by $$\mathsf{ev}:=[\Lambda x.\{\{\mathbf{p}_{1}\}(x)\}(\{\mathbf{p}_{2}\}(x))]_{\approx}:(\mathsf{A}\Rightarrow \mathsf{B})\times \mathsf{A}\rightarrow\mathsf{B}$$
\end{enumerate}
\end{proof}
\begin{remark}
A parameterized natural numbers object can be defined by considering the object $\mathsf{List}(\mathsf{1})$, and the arrows zero and successor defined as $\epsilon$ and $\mathsf{cons}\circ \langle \mathsf{id}_{\mathsf{List}(\mathsf{1})},!_{\mathsf{List}(\mathsf{1})}\rangle$ respectively where $!_{\mathsf{List}(\mathsf{1})}$ is the unique arrow in $\mathcal{C}_{r}$ from $\mathsf{List}(\mathsf{1})$ to $\mathsf{1}$. However one can directly define a parameterized natural numbers object $\mathsf{N}$ as $\{x|\,x=x\}$ together with the arrows $[\Lambda x.0]_{\approx}:\mathsf{1}\rightarrow \mathsf{N}$ and $[\mathbf{succ}]_{\approx}:\mathsf{N}\rightarrow \mathsf{N}$. This presentation of  natural numbers object will help to avoid
annoyed encodings when showing the validity the Formal Church's thesis in the
doctrine based on $\mathcal{C}_r$.
\end{remark}

\begin{lemma}\label{mono} An arrow $\mathsf{j}:=[\mathbf{j}]_{\approx}:\mathsf{A}\rightarrow \mathsf{B}$ in $\mathcal{C}_{r}$ is a monomorphism if and only if 
$$x\,\varepsilon\, \mathsf{A}\wedge y\,\varepsilon\, \mathsf{A}\wedge \{\mathbf{j}\}(x)=  \{\mathbf{j}\}(y)\vdash_{\tar} x=y.$$
\end{lemma}
\begin{proof}
Just consider the kernel pair of the monomorphism.
\end{proof}
\begin{theorem}
Finite coproducts in $\mathcal{C}_{r}$ are disjoint and stable.
\end{theorem}
\begin{proof}
By lemma~\ref{mono} coproduct injections are monic.
The rest
can be proven via a direct straightforward verification.
\end{proof}

\begin{remark}
Notice that the category $\mathcal{C}_{r}$ is not well pointed. In fact if $I$ is an undecidable sentence of $\tar$ (which exists as soon as $\tar$ is consistent, by G\"odel's incompleteness theorem) the arrows $[\Lambda x.0]_{\approx}$ and $[\Lambda x.1]_{\approx}$ are well defined distinct arrows from $\{x|\,(x=0\wedge I)\vee (x=1\wedge \neg I)\}$ to $\{x|x=0\vee x=1\}$. However the object  $\{x|\,(x=0\,\wedge\, I)\,\vee\, (x=1\,\wedge\, \neg I)\}$ has no points, i.\,e.\ there are no arrows from $\mathsf{1}$ to $\{x|\,(x=0\,\wedge \,I)\,\vee\, (x=1\,\wedge\, \neg I)\}$ in $\mathcal{C}_{r}$.
\end{remark}

\subsubsection{The indexed category  of realized collections}
\label{indexcol}
Here we are going to describe the indexed category of realized collections which
we will use  to show that
$\mathcal{C}_{r}$ is weakly locally cartesian closed in a straightforward way.

\begin{definition}\label{colcon}
Suppose $\mathsf{A}$ is an object of $\mathcal{C}_{r}$. A \emph{family of realized collections} on $\mathsf{A}$ is a formal expression $$\{x'|\,\varphi_{\mathsf{C}}(x,x')\}$$ where $\varphi_{\mathsf{C}}(x,x')$ is a formula of $\tar$ with at most $x$ and $x'$ as free variables (we write $x'\,\varepsilon\, \mathsf{C}(x)$ as an abbreviation for $\varphi_{\mathsf{C}}(x,x')$) for which $$x'\,\varepsilon\,\mathsf{C}(x)\,\vdash_{\tar}\,x\,\varepsilon\, \mathsf{A}$$ 

Families of realized collections on $\mathsf{A}$ with provably (in $\tar$) equivalent membership relations are identified.

An \emph{operation} from a family of realized collections $\mathsf{C}(x)$ on $\mathsf{A}$ to another $\mathsf{D}(x)$ is given by an equivalence class $[\mathbf{n}]_{\approx_{\mathsf{C}(x),\mathsf{D}(x)}}$ of numerals such that 
$$x'\,\varepsilon\,\mathsf{C}(x)\,\vdash_{\tar}\,\{\mathbf{n}\}(x,x')\,\varepsilon\,\mathsf{D}(x)$$
with respect to the equivalence relation defined as follows: 
$$\mathbf{n}\approx_{\mathsf{C}(x),\mathsf{D}(x)} \mathbf{m}\textnormal{ if and only if }x'\,\varepsilon\, \mathsf{C}(x)\,\vdash_{\tar}\,\{\mathbf{n}\}(x,x')=\{\mathbf{m}\}(x,x')$$

 If $\mathsf{f}=[\mathbf{n}]_{\approx_{\mathsf{C}(x),\mathsf{D}(x)}}:\mathsf{C}(x)\rightarrow \mathsf{D}(x)$ and $\mathsf{g}=[\mathbf{m}]_{\approx_{\mathsf{D}(x),\mathsf{E}(x)}}:\mathsf{D}(x)\rightarrow \mathsf{E}(x)$ are operations between families of realized collection on $\mathsf{A}$, then their \emph{composition} is defined as 
 $$\mathsf{g}\circ \mathsf{f}:=[\Lambda x.\Lambda x'.\{\mathbf{m}\}(x,\{\mathbf{n}\}(x,x'))]_{\approx_{\mathsf{C}(x),\mathsf{E}(x)}}:\mathsf{C}(x)\rightarrow\mathsf{E}(x)$$
  
If $\mathsf{C}(x)$ is a family of realized collections on $\mathsf{A}$, then its \emph{identity} operation is defined by
 $$\mathsf{id}_{\mathsf{C}(x)}:=[\Lambda x.\Lambda x'.x']_{\approx_{\mathsf{C}(x),\mathsf{C}(x)}}:\mathsf{C}(x)\rightarrow\mathsf{C}(x)$$
\end{definition}
The proof of the following lemma is an immediate verification.

\begin{lemma}
For every object $\mathsf{A}$ of $\mathcal{C}_{r}$, families of realized collections on $\mathsf{A}$ and operations between them together with their composition and identity operations define a category. We denote this category with $\mathbf{Col}^{r}(\mathsf{A})$.
\end{lemma} 
 From now on, we will omit the subscripts of $\approx$ when they will be clear from the context.

The proof of the following lemma consists of an immediate verification, too.

 \begin{lemma}
If $\mathsf{f}:=[\mathbf{n}]_{\approx}:\mathsf{A}\rightarrow \mathsf{B}$, then the following assignments give rise to a functor $\mathbf{Col}^{r}_{\mathsf{f}}$ from $\mathbf{Col}^{r}(\mathsf{B})$ to $\mathbf{Col}^{r}(\mathsf{A})$:
\begin{enumerate}
\item for every object $\mathsf{C}(x)$ of $\mathbf{Col}^{r}(\mathsf{B})$
$$\mathbf{Col}^{r}_{\mathsf{f}}(\mathsf{C}(x)):=\{x'|\,x\,\varepsilon\, \mathsf{A}\wedge x'\,\varepsilon\, \mathsf{C}(\{\mathbf{n}\}(x))\}$$
\item for every arrow $\mathsf{g}:=[\mathbf{k}]_{\approx}:\mathsf{C}(x)\rightarrow\mathsf{D}(x)$ of $\mathbf{Col}^{r}(\mathsf{B})$ 
$$\mathbf{Col}^{r}_{\mathsf{f}}(\mathsf{g}):=[\Lambda x.\Lambda x'. \{\mathbf{k}\}(\{\mathbf{n}\}(x),x')]_{\approx}:{\mathbf{Col}^{r}_{\mathsf{f}}(\mathsf{C}(x))\rightarrow \mathbf{Col}^{r}_{\mathsf{f}}(\mathsf{D}(x))}.$$\end{enumerate}

\end{lemma}
We can now prove the following theorem enumerating the properties of $\mathbf{Col}^{r}$.
\begin{theorem}\label{tc1} For every object $\mathsf{A}$ of $\mathcal{C}_{r}$, $\mathbf{Col}^{r}(\mathsf{A})$ is a finitely complete category with list objects, finite coproducts and weak exponentials. Moreover for every arrow $\mathsf{f}:\mathsf{A}\rightarrow\mathsf{B}$ in $\mathcal{C}_{r}$, the functor $\mathbf{Col}^{r}_{\mathsf{f}}$ preserves this structure. 
\end{theorem}
\begin{proof} The proof of this theorem is similar to the proof of theorem \ref{t1} and the fact that the functors $\mathbf{Col}^{r}_{\mathsf{f}}$ preserve the structure is an easy verification. 
In particular
\begin{enumerate}
\item $\mathsf{1}_{\mathsf{A}}:=\{x'|\,x'=0\wedge x\,\varepsilon\, \mathsf{A}\}$ is a terminal object;
\item $\mathsf{B}(x)\times_{\mathsf{A}} \mathsf{C}(x):=\{x'|\,p_{1}(x')\,\varepsilon\, \mathsf{B}(x)\,\wedge\,  p_{2}(x')\,\varepsilon\,\mathsf{C}(x)\}$ defines a binary product for $\mathsf{B}(x)$ and $\mathsf{C}(x)$ together with the left  projection arrow defined as $[\Lambda x.\Lambda x'.\{\mathbf{p}_{1}\}(x')]_{\approx}:\mathsf{B}(x)\times_{\mathsf{A}} \mathsf{C}(x)\rightarrow \mathsf{B}(x)$ and the right projection arrow defined as $[\Lambda x.\Lambda x'.\{\mathbf{p}_{2}\}(x')]_{\approx}:\mathsf{B}(x)\times_{\mathsf{A}} \mathsf{C}(x)\rightarrow \mathsf{C}(x)$; 
\item if $\mathsf{f}=[\mathbf{n}]_{\approx}$ and $\mathsf{g}=[\mathbf{m}]_{\approx}$ are arrows in $\mathbf{Col}^{r}(\mathsf{A})$ from $\mathsf{B}(x)$ to $\mathsf{C}(x)$, then their equalizer is given by the realized collection  
$$\mathsf{Eq}(\mathsf{f},\mathsf{g}):=\{x'|\,x'\,\varepsilon\,\mathsf{B}(x)\,\wedge\, \{\mathbf{n}\}(x,x')=\{\mathbf{m}\}(x,x')\}$$ and the arrow $\mathsf{eq}(\mathsf{f},\mathsf{g}):=[\Lambda x.\Lambda x'.x']_{\approx}:\mathsf{Eq}(\mathsf{f},\mathsf{g})\rightarrow \mathsf{B}(x)$;
\item $\mathsf{0}_{\mathsf{A}}:=\{x'|\,\bot\}$ is an initial object;

\item the realized collection $\mathsf{B}(x)+_{\mathsf{A}}\mathsf{C}(x)$ defined as 
$$\{x'|\;( p_{1}(x')=0\wedge  p_{2}(x')\,\varepsilon\, \mathsf{B}(x))\vee ( p_{1}(x')=1\wedge  p_{2}(x')\,\varepsilon\,\mathsf{C}(x))\}$$ gives a binary coproduct in $\mathbf{Col}^{r}(\mathsf{A})$ for $\mathsf{B}(x)$ and $\mathsf{C}(x)$ together with the injection arrows defined as $[\Lambda x.\Lambda x'.\{\mathbf{p}\}(0,x')]_{\approx}:\mathsf{B}(x)\rightarrow\mathsf{B}(x)+_{\mathsf{A}}\mathsf{C}(x)$ and $[\Lambda x.\Lambda x'.\{\mathbf{p}\}(1,x')]_{\approx}:\mathsf{C}(x)\rightarrow\mathsf{B}(x)+_{\mathsf{A}}\mathsf{C}(x)$;

\item $\mathsf{List}_{\mathsf{A}}(\mathsf{B}(x)):=\{x'|\,x\,\varepsilon\, A\wedge\forall j(j<lh(x')\rightarrow (x')_{j}\,\varepsilon\,\mathsf{B}(x))\}$ defines a parameterized list object for $\mathsf{B}(x)$ together with the empty list arrow defined as $[\Lambda x.\Lambda x'.0]_{\approx}:\mathsf{1}_{\mathsf{A}}\rightarrow \mathsf{List}_{\mathsf{A}}(\mathsf{B}(x))$ and the append arrow defined as $[\Lambda x.\Lambda x'.\{\mathbf{cnc}\}(\{\mathbf{p}_{1}\}(x'),\{\mathbf{p}_{2}\}(x'))]_{\approx}:\mathsf{List}_{\mathsf{A}}(\mathsf{B}(x))\times_{\mathsf{A}} \mathsf{B}(x)\rightarrow \mathsf{List}_{\mathsf{A}}(\mathsf{B}(x))$;

\item $\mathsf{B}(x)\Rightarrow_{\mathsf{A}}\mathsf{C}(x):=\{x'|\,x\,\varepsilon\, \mathsf{A}\wedge \forall t\, (t\,\varepsilon\,\mathsf{B}(x)\, \rightarrow\, \{x'\}(t)\,\varepsilon\,\mathsf{C}(x))\}$ defines a weak exponential for $\mathsf{B}(x)$ and $\mathsf{C}(x)$ together with the evaluation arrow $$[\Lambda x.\Lambda x'.\{\{\mathbf{p}_{1}\}(x')\}(\{\mathbf{p}_{2}\}(x'))]_{\approx}:(\mathsf{B}(x)\Rightarrow_{\mathsf{A}}\mathsf{C}(x))\times_{\mathsf{A}}\mathsf{B}(x)\rightarrow \mathsf{C}(x)$$
\end{enumerate}
\end{proof}
The following theorem follows from a simple verification.
\begin{theorem} The assignments $\mathsf{A}\mapsto \mathbf{Col}^{r}(\mathsf{A})$ and $\mathsf{f}\mapsto \mathbf{Col}^{r}_{\mathsf{f}}$ define an indexed category
$$\mathbf{Col}^{r}:\mathcal{C}_{r}^{op}\rightarrow \mathbf{Cat}.$$
\end{theorem}
The functors $\mathbf{Col}^{r}_{\mathsf{f}}$ are also called \emph{substitution functors}.

In the following theorem we prove that substitution functors have left adjoints.
\begin{theorem}\label{tc2} For every $\mathsf{f}:\mathsf{A}\rightarrow \mathsf{B}$ in $\mathcal{C}_{r}$, the functor $\mathbf{Col}^{r}_{\mathsf{f}}$ has a left adjoint $$\Sigma_{\mathsf{f}}:\mathbf{Col}^{r}(\mathsf{A})\rightarrow \mathbf{Col}^{r}(\mathsf{B}).$$
\end{theorem}
\begin{proof}
If $\mathsf{f}=[\mathbf{n}]_{\approx_{\mathsf{A},\mathsf{B}}}:\mathsf{A}\rightarrow \mathsf{B}$ is an arrow in $\mathcal{C}_{r}$, then a left adjoint $\Sigma_{\mathsf{f}}$ to $\mathbf{Col}^{r}_{\mathsf{f}}$ is defined by the following conditions:
\begin{enumerate}
\item if $\mathsf{C}(x)$ is an object of $\mathbf{Col}^{r}(\mathsf{A})$, then 
$$\Sigma_{\mathsf{f}}(\mathsf{C}(x)):=\{x'|\,x=\{\mathbf{n}\}(p_{1}(x'))\,\wedge\, 
p_{2}(x')\,\varepsilon\,\mathsf{C}(p_{1}(x'))\}$$
\item if $\mathsf{g}:=[\mathbf{m}]_{\approx}:\mathsf{C}(x)\rightarrow \mathsf{D}(x)$ in $\mathbf{Col}^{r}(\mathsf{A})$, then $$\Sigma_{\mathsf{f}}(\mathsf{g}):\Sigma_{\mathsf{f}}(\mathsf{C}(x))\rightarrow \Sigma_{\mathsf{f}}(\mathsf{D}(x)).$$
is defined as $[\Lambda x.\Lambda x'.\{\mathbf{p}\}\,(\{\mathbf{p}_{1}\}(x'),\,\{\mathbf{m}\}(\{\mathbf{p}_{1}\}(x'),\,\{\mathbf{p}_{2}\}(x')))]_{\approx}$.
\end{enumerate}

\end{proof}

In the following theorem we prove that substitution functors have weak versions of right adjoints.

\begin{theorem}\label{tc3} For every $\mathsf{f}:=[\mathbf{n}]_{\approx}:\mathsf{A}\rightarrow \mathsf{B}$ in $\mathcal{C}_{r}$ and every object $\mathsf{C}(x)$ in $\mathbf{Col}^{r}(\mathsf{A})$, there exists 
\begin{enumerate}
\item an object $\Pi_{\mathsf{f}}(\mathsf{C}(x))$ in $\mathbf{Col}^{r}(\mathsf{B})$ and 
\item an arrow $\mathsf{ev}^{\Pi,\mathsf{f}}_{\mathsf{C}(x)}:\mathbf{Col}^{r}_{\mathsf{f}}(\Pi_{\mathsf{f}}(\mathsf{C}(x)))\rightarrow \mathsf{C}(x)$ in $\mathbf{Col}^{r}(\mathsf{B})$
\end{enumerate}
such that for every $\mathsf{D}(x)\in \mathbf{Col}^{r}(\mathsf{B})$ and every arrow $\mathsf{g}:\mathbf{Col}_{f}^{r}(\mathsf{D}(x))\rightarrow \mathsf{C}(x)$ in $\mathbf{Col}^{r}(\mathsf{A})$, there exists an arrow $\mathsf{g}':\mathsf{D}(x)\rightarrow \Pi_{\mathsf{f}}(\mathsf{C}(x))$ in $\mathbf{Col}^{r}(\mathsf{B})$ such that the following diagram commutes in $\mathbf{Col}^{r}(\mathsf{A})$:
{\small\def\objectstyle{\scriptstyle}
\def\labelstyle{\scriptstyle}
{
$$\xymatrix@-1pc{ 
\mathbf{Col}_{\mathsf{f}}^{r}(\mathsf{D}(x))\ar[r]^-{\mathsf{g}}\ar[d]_-{\mathbf{Col}^{r}_{\mathsf{f}}(\mathsf{g}')}			&\mathsf{C}(x)\\
\mathbf{Col}_{\mathsf{f}}^{r}(\Pi_{\mathsf{f}}(\mathsf{C}(x)))\ar[ru]_-{\mathsf{ev}^{\Pi,\mathsf{f}}_{\mathsf{C}(x)}}\\
}$$}}
\end{theorem}
\begin{proof}
It is sufficient to define 
$$\Pi_{\mathsf{f}}(\mathsf{C}(x)):=\{x'|\,x\,\varepsilon\,\mathsf{B}\wedge \forall y\,(\, x=\{\mathbf{n}\}(y)\wedge y\,\varepsilon\,\mathsf{A}\rightarrow \{x'\}(y)\,\varepsilon\,\mathsf{C}(y)\,)\}$$
and 
$$\mathsf{ev}^{\Pi,\mathsf{f}}_{\mathsf{C}(x)}:=[\Lambda x.\Lambda x'. \{x'\}(x)]_{\approx}$$
 \end{proof} 
\begin{remark}
Suppose $\mathsf{f}:\mathsf{A}\rightarrow \mathsf{B}$ is an arrow in $\mathcal{C}_{r}$.
Notice that for every arrow $\mathsf{g}:\mathsf{C}(x)\rightarrow \mathsf{D}(x)$ in $\mathsf{Col}^{r}(\mathsf{A})$, one can obtain an arrow from $\Pi_{\mathsf{f}}(\mathsf{C}(x))$ to $ \Pi_{\mathsf{f}}(\mathsf{D}(x))$ since there exists an arrow $\mathsf{h}$ making the following diagram commute.
{\small
\def\objectstyle{\scriptstyle}
\def\labelstyle{\scriptstyle}
{
$$\xymatrix@-1pc{
\mathbf{Col}_{\mathsf{f}}^{r}(\Pi_{\mathsf{f}}(\mathsf{C}(x)))\ar[r]^-{\mathsf{g}\circ \mathsf{ev}^{\Pi,\mathsf{f}}_{\mathsf{C}(x)} }\ar[d]_-{\mathbf{Col}^{r}_{\mathsf{f}}(\mathsf{h})}			&\mathsf{D}(x)\\
\mathbf{Col}_{\mathsf{f}}^{r}(\Pi_{\mathsf{f}}(\mathsf{D}(x)))\ar[ru]_-{\mathsf{ev}^{\Pi,\mathsf{f}}_{\mathsf{D}(x)}}\\
}$$}}
However, while the existence of such an arrow is guaranteed by the previous theorem, it is not the case for its uniqueness.
 In particular, we are not able to find a funtorial interpretation of $\lambda$-abstraction  to obtain a functor  $\Pi_{\mathsf{f}}:\mathbf{Col}^{r}(\mathsf{A})\rightarrow \mathbf{Col}^{r}(\mathsf{B})$.
\end{remark}
\begin{remark}\label{remark} If we consider the arrow $\mathsf{p}_{1}:\mathsf{A}\times \mathsf{B}\rightarrow \mathsf{A}$ in $\mathcal{C}_{r}$ and an object $\mathsf{R}(x)$ in $\mathbf{Col}^{r}(\mathsf{A}\times \mathsf{B})$, then $\Sigma_{\mathsf{p}_{1}}(\mathsf{R}(x))$ is isomorphic in $\mathbf{Col}^{r}(\mathsf{A})$ to $$\Sigma'_{\mathsf{p}_{1}}(\mathsf{R}(x)):=\{x'|\,x\,\varepsilon\, \mathsf{A} \wedge p_{1}(x')\,\varepsilon \, \mathsf{B}\wedge p_{2}(x')\,\varepsilon\,\mathsf{R}(p(x,p_{1}(x')))\}$$
and $\Pi_{\mathsf{p}_{1}}(\mathsf{R})$ is equivalent in $\mathbf{Col}^{r}(\mathsf{A})$ to $$\Pi'_{\mathsf{p}_{1}}(\mathsf{R}(x)):=\{x'|\,x\,\varepsilon\, \mathsf{A} \wedge \forall t( t\,\varepsilon \, \mathsf{B}\rightarrow \{x'\}(t)\,\varepsilon\,\mathsf{R}(p(x,t)))\}$$
\end{remark}

\begin{definition}\label{sigma} If $\mathsf{A}$ is an object of $\mathcal{C}_{r}$ and $\mathsf{B}(x)$ is a an object of $\mathbf{Col}^{r}(\mathsf{A})$, then $\Sigma(\mathsf{A},\mathsf{B}(x))$ is the object of $\mathcal{C}_{r}$ defined as $\{x|\,p_{1}(x)\,\varepsilon\, \mathsf{A}\wedge p_{2}(x)\,\varepsilon\, \mathsf{B}(p_{1}(x))\}$.
\end{definition}
The proof of the following lemma consists of an easy verification.
\begin{lemma}
For every $\mathsf{f}:=[\mathbf{n}]_{\approx}:\mathsf{A}\rightarrow \mathsf{B}$ in $\mathcal{C}_{r}$ and every $\mathsf{C}$ in $\mathbf{Col}^{r}(\mathsf{B})$, the following diagram is a pullback
{\small \def\objectstyle{\scriptstyle}
\def\labelstyle{\scriptstyle}
{
$$\xymatrix@-1pc{
\Sigma(\mathsf{A},\mathbf{Col}_{\mathsf{f}}^{r}(\mathsf{C}))\ar[rrr]^-{\Sigma(\mathsf{f},\mathsf{C})}\ar[d]_-{\mathsf{p}_{1}^{\Sigma}}		&& &\Sigma(\mathsf{B},\mathsf{C})\ar[d]^{\mathsf{p}_{1}^{\Sigma}}\\
\mathsf{A}\ar[rrr]_-{\mathsf{f}}	& &	&\mathsf{B}\\
}$$}}
if $\Sigma(\mathsf{f},\mathsf{C})$ is $[\Lambda x. \{\mathbf{p}\}(\{\mathbf{n}\}(\{\mathbf{p}_{1}\}(x)),\{\mathbf{p}_{2}\}(x))]_{\approx}$. 
\end{lemma}

\begin{theorem}\label{pdue} For every $\mathsf{A}$ in $\mathcal{C}_{r}$, the categories $\mathcal{C}_{r}/\mathsf{A}$ and $\mathbf{Col}^{r}(\mathsf{A})$ are equivalent.
\end{theorem}
\begin{proof} Suppose $\mathsf{A}$ is an object of $\mathcal{C}_{r}$.

It is sufficient to consider the functor $\mathbf{J}_{\mathsf{A}}:\mathcal{C}_{r}/\mathsf{A}\rightarrow \mathbf{Col}^{r}(\mathsf{A})$ defined by  the assignments
$$[\mathbf{b}]_{\approx}:\mathsf{B}\rightarrow \mathsf{A}\mapsto \{x'|\,x'\,\varepsilon\, \mathsf{B}\wedge x=\{\mathbf{b}\}(x')\}$$
$$[\mathbf{n}]_{\approx}\mapsto [\Lambda x.\Lambda x'.\{\mathbf{n}\}(x')]_{\approx} $$

and the functor $\mathbf{I}_{\mathsf{A}}:\mathbf{Col}^{r}(\mathsf{A})\rightarrow\mathcal{C}_{r}/\mathsf{A}$ defined by the assignments (see def.~\ref{sigma})
$$\mathsf{B}(x)\mapsto \mathsf{p}_{1}^{\Sigma}:=[\mathbf{p}_{1}]_{\approx}:\Sigma(\mathsf{A},\mathsf{B}(x))\rightarrow \mathsf{A}$$
$$[\mathbf{n}]_{\approx}\mapsto [\Lambda x.\{\mathbf{p}\}(\{\mathbf{p}_{1}\}(x),\{\mathbf{n}\}(\{\mathbf{p}_{1}\}(x),\{\mathbf{p}_{2}\}(x)))]_{\approx}$$

\end{proof}
\begin{corollary}\label{weakly}
$\mathcal{C}_{r}$ is weakly locally cartesian closed.
\end{corollary}

$\mathbf{Col}^{r}$ is a functorial account of the slice pseudofunctor $\mathcal{C}_{r}/-$. In fact, whilst  the $\mathbf{J}_{\mathsf{A}}$'s and the $\mathbf{I}_{\mathsf{A}}$'s don't give rise to natural transformations, we have the following result whose proof follows easily. 
\begin{theorem}\label{t3}
For every $\mathsf{f}:\mathsf{A}\rightarrow \mathsf{B}$ in $\mathcal{C}^{r}$
the functor
$$\mathbf{I}_{\mathsf{A}}\circ \mathbf{Col}^{r}_{\mathsf{f}}\circ \mathbf{J}_{\mathsf{B}}$$ coincides with the functor
$$\mathcal{C}_{r}/\mathsf{f}:\mathcal{C}_{r}/\mathsf{B}\rightarrow \mathcal{C}_{r}/\mathsf{A}$$
defined as follows:
\begin{enumerate}
\item $\mathcal{C}_{r}/\mathsf{f}$ send an object $\mathsf{e}:\mathsf{E}\rightarrow \mathsf{B}$ in $\mathcal{C}_{r}/\mathsf{B}$ to its pullback $(\mathcal{C}_{r}/\mathsf{f})(\mathsf{e})$ along $\mathsf{f}$ determined by the following commutative diagram
{\small \def\objectstyle{\scriptstyle}
\def\labelstyle{\scriptstyle}
{
$$\xymatrix@-1pc{ 
\mathsf{Eq}(\mathsf{f}\circ \mathsf{p}_{1},\mathsf{e}\circ\mathsf{p}_{2})\ar[rd]^{\qquad\mathsf{eq}(\mathsf{f}\circ \mathsf{p}_{1},\mathsf{e}\circ\mathsf{p}_{2})}\ar[rdd]_-{(\mathcal{C}_{r}/\mathsf{f})(\mathsf{e})}\ar[rd]^{}		&	&\\
								&\mathsf{A}\times \mathsf{E}\ar[r]^-{\mathsf{p}_{2}}\ar[d]^-{\mathsf{p}_{1}}		&\mathsf{E}\ar[d]^-{\mathsf{e}}\\
								&\mathsf{A}\ar[r]_-{\mathsf{f}}										&\mathsf{B}\\
}$$}}
\item if $\mathsf{g}$ is an arrow in $\mathcal{C}_{r}/\mathsf{f}$ from $\mathsf{e}:\mathsf{E}\rightarrow \mathsf{B}$ to $\mathsf{e}':\mathsf{E}'\rightarrow \mathsf{B}$, then $(\mathcal{C}_{r}/\mathsf{f})(\mathsf{g})$ is the unique arrow making the following diagram commute
{\small \def\objectstyle{\scriptstyle}
\def\labelstyle{\scriptstyle}
{
$$\xymatrix@-1pc{
														&\mathsf{Eq}(\mathsf{f}\circ \mathsf{p}_{1},\mathsf{e}'\circ\mathsf{p}_{2})\ar[ldd]^-{(\mathcal{C}_{r}/\mathsf{f})(\mathsf{e}')\qquad}	\ar[r]^-{\mathsf{p}_{2}\circ \mathsf{eq}}&\mathsf{E}'\ar[ldd]^-{\mathsf{e}'}\\
\mathsf{Eq}(\mathsf{f}\circ \mathsf{p}_{1},\mathsf{e}\circ\mathsf{p}_{2})\ar[d]_-{(\mathcal{C}_{r}/\mathsf{f})(\mathsf{e})}\ar[ru]^-{(\mathcal{C}_{r}/\mathsf{f})(\mathsf{g})\qquad}\ar[r]^-{\mathsf{p}_{2}\circ \mathsf{eq}}	&\mathsf{E}\ar[d]_-{\mathsf{e}}\ar[ru]^-{\mathsf{g}}				&\\
\mathsf{A}	\ar[r]_-{\mathsf{f}}			&\mathsf{B}			&\\									
}$$}}

\end{enumerate}
\end{theorem}

The following two lemmas will be helpful in the following sections. Their proofs are immediate.

\begin{lemma}\label{univprod}Let $\mathsf{A}$ and $\mathsf{B}$ be objects of $\mathcal{C}_{r}$ and let $\mathsf{C}$ and $\mathsf{D}$ be object of $\mathbf{Col}^{r}(\mathsf{B})$. Suppose that $\mathsf{f}:\mathsf{A}\rightarrow \Sigma(\mathsf{B},\mathsf{C})$ and $\mathsf{g}:\mathsf{A}\rightarrow \Sigma(\mathsf{B},\mathsf{D})$ be arrows in $\mathcal{C}_{r}$ such that $\mathsf{p}_{1}^{\Sigma}\circ \mathsf{f}=\mathsf{p}_{1}^{\Sigma}\circ \mathsf{g}:\mathsf{A}\rightarrow \mathsf{B}$. Then there exists a unique arrow $$\mathsf{h}:\mathsf{A}\rightarrow \Sigma(\mathsf{B},\mathsf{C}\times_{\mathsf{B}} \mathsf{D})$$ such that the following diagram commutes.
{\small\def\objectstyle{\scriptstyle}
\def\labelstyle{\scriptstyle}
{
$$\xymatrix@-1pc{
&\Sigma(\mathsf{B},\mathsf{C})\\
\mathsf{A}\ar[r]^-{\mathsf{h}}\ar[ru]^{\mathsf{f}}\ar[rd]_{\mathsf{g}}		&\Sigma(\mathsf{B}, \mathsf{C}\times_{\mathsf{B}} \mathsf{D})\ar[u]_{\mathbf{I}(\mathsf{p}_{1})}\ar[d]^{\mathbf{I}(\mathsf{p}_{2})}\\
										&\Sigma(\mathsf{B},\mathsf{D})\\
}$$}}
\end{lemma}

\begin{lemma}\label{univsig}Let $\mathsf{A}$,$\mathsf{B}$ and $\mathsf{C}$ be objects of $\mathcal{C}_{r}$, let $\mathsf{D}$ be an object of $\mathbf{Col}^{r}(\mathsf{C})$ and let $\mathsf{f}:\mathsf{B}\rightarrow \mathsf{C}$ be an arrow in $\mathcal{C}_{r}$. Suppose that $\mathsf{g}:\mathsf{A}\rightarrow \Sigma(\mathsf{C},\mathsf{D})$ and $\mathsf{h}:\mathsf{A}\rightarrow \mathsf{B}$ be arrows in $\mathcal{C}_{r}$ such that $\mathsf{f}\circ \mathsf{h}=\mathsf{p}_{1}^{\Sigma}\circ \mathsf{g}:\mathsf{A}\rightarrow \mathsf{C}$. Then there exists a unique arrow $$\mathsf{j}:\mathsf{A}\rightarrow \Sigma(\mathsf{B},\mathbf{Col}_{\mathsf{f}}\mathsf{D})$$ such that the following diagram commutes.
{\small \def\objectstyle{\scriptstyle}
\def\labelstyle{\scriptstyle}
{
$$\xymatrix@-1pc{										&\Sigma(\mathsf{C},\mathsf{D})\\
\mathsf{A}\ar[r]^-{\mathsf{j}}\ar[ru]^{\mathsf{g}}\ar[rd]_{\mathsf{h}}		&\Sigma(\mathsf{B}, \mathbf{Col}_{\mathsf{f}}(\mathsf{D}))\ar[u]_{\Sigma(\mathsf{f},\mathsf{D})}\ar[d]^{\mathsf{p}_{1}^{\Sigma}}\\
										&\mathsf{B}\\
}$$}}
\end{lemma}

\subsection{The hyperdoctrine $\thprop$ of realized propositions}
We first define the indexed category of realized propositions on $\mathcal{C}_{r}$ whose posetal reflection $\thprop$ will be a Lawvere's hyperdoctrine.
\begin{definition}
The indexed category $\mathbf{Prop}^{r}:\mathcal{C}_{r}^{op}\rightarrow \mathbf{Cat}$ is defined as the preorder reflection $\mathbf{P}[\mathbf{Col}^{r}]$ of $\mathbf{Col}^{r}$ (see definition \ref{preorder}).
In this case we write $x'\Vdash \mathsf{P}(x)$ instead of $x'\,\varepsilon\, \mathsf{P}(x)$ for an object $\mathsf{P}(x)$ in $\mathbf{Col}^{r}(\mathsf{A})$ when we look at it as an object of $\mathbf{Prop}^{r}(\mathsf{A})$.
\end{definition}

\begin{remark}\label{unif}
Notice that if $\mathsf{A}$ is an object of $\mathcal{C}_{r}$ and $\mathsf{P}(x)$ and $\mathsf{Q}(x)$ are objects of $\mathsf{Prop}(\mathsf{A})$, then $\mathsf{P}(x)\sqsubseteq_{\tar}\mathsf{Q}(x)$ if and only if there is a numeral $\mathbf{r}$ such that 
$$x'\Vdash \mathsf{P}(x)\vdash_{\tar}\{\mathbf{r}\}(x,x')\Vdash \mathsf{Q}(x).$$
So $\mathsf{P}(x)$ entails $\mathsf{Q}(x)$ if there is a recursive way (recursively depending on $x$ in $\mathsf{A}$) to send realizers of $\mathsf{P}(x)$ to realizers of $\mathsf{Q}(x)$. The fact that $\mathsf{A}$ is a realized collection of natural numbers allows us to exploit the numerical data from the underlying domain $\mathsf{A}$ and include them in the notion of entailment. This does not happen in the effective tripos~\cite{eff} where the base category is $\mathsf{Set}$ and its entailment is uniformly defined with respect to the points in the underlying domain $A$. Namely, if $P$ and $Q$ are functions from $A$ to $\mathcal{P}(\mathbb{N})$, then in the effective tripos $$P\sqsubseteq_{A,\eff}Q\textnormal{ if and only if }\bigcap_{a\in A}\{n\in \mathbb{N}|\,\forall x\in P(a) (\{n\}(x)\in Q(b))\}\neq \emptyset.$$
However (see remark \ref{rmk}) these two notions can be compared.

\end{remark}

Now we are read to prove the following theorem:
\begin{theorem} 
$\mathbf{Prop}^{r}$ is a first-order hyperdoctrine (see definition \ref{hyp}).
\end{theorem}
\begin{proof}
For every object $\mathsf{A}$ of $\mathcal{C}_{r}$, $\mathbf{Prop}^{r}(\mathsf{A})$ is a Heyting prealgebra. In fact it is sufficient to consider bottom, top, binary infimums, binary supremums and Heyting implication given by $\bot_{\mathsf{A}}:=\mathsf{0}_{\mathsf{A}}$, $\top_{\mathsf{A}}:=\mathsf{1}_{\mathsf{A}}$, $\mathsf{C}(x)\sqcap_{\mathsf{A}}\mathsf{D}(x):=\mathsf{C}(x)\times_{\mathsf{A}}\mathsf{D}(x)$, $\mathsf{C}(x)\sqcup_{\mathsf{A}}\mathsf{D}(x):=\mathsf{C}(x)+_{\mathsf{A}}\mathsf{D}(x)$ and $\mathsf{C}(x)\rightarrow_{\mathsf{A}}\mathsf{D}(x):=\mathsf{C}(x)\Rightarrow_{\mathsf{A}}\mathsf{D}(x)$ respectively for all objects $\mathsf{C}(x),\mathsf{D}(x)$ in $\mathbf{Prop}^{r}(\mathsf{A})$.
Moreover from $\ref{tc1}$ it immediately follows that $\mathbf{Prop}^{r}_{\mathsf{f}}$ is a morphism of Heyting prealgebras from $\mathbf{Prop}^{r}(\mathsf{B})$ to $\mathbf{Prop}^{r}(\mathsf{A})$ for every arrow $\mathsf{f}:\mathsf{A}\rightarrow \mathsf{B}$ in $\mathcal{C}_{r}$.
From $\ref{tc2}$ and $\ref{tc3}$ one can easily obtain that for every such an arrow, $\exists_{\mathbf{f}}(\mathsf{C}(x)):=\Sigma_{\mathbf{f}}(\mathsf{C}(x))$ and $\forall_{\mathsf{f}}(\mathsf{C}(x)):=\Pi_{\mathbf{f}}(\mathsf{C}(x))$ define left and right adjoints to $\mathbf{Prop}^{r}_{\mathsf{f}}$ respectively in the category of preorders.
One can easily check that these adjoints satisfy Beck-Chevalley condition.
\end{proof}

We can also prove the following important property of $\mathbf{Prop}^{r}$:
\begin{lemma}\label{weco} The hyperdoctrine $\mathbf{Prop}^{r}$ has weak comprehensions (see definition \ref{weacomp}).
\end{lemma}
\begin{proof}
If $\mathsf{A}$ is an object of $\mathcal{C}_{r}$ and $\mathsf{P}(x)$ is in $\mathbf{Prop}^{r}(\mathsf{A})$, then we can consider the object $\Sigma(\mathsf{A},\mathsf{P}(x))$ of $\mathcal{C}_{r}$ and the arrow $\mathsf{p}_{1}^{\Sigma}:\Sigma(\mathsf{A},\mathsf{P}(x)) \rightarrow\mathsf{A}$.
This arrow determines a weak comprehension for $\mathsf{P}(x)$. \end{proof}
\begin{remark}
Notice that an equalizer for $\mathsf{f}:=[\mathbf{n}]_{\approx},\mathsf{g}:=[\mathbf{m}]_{\approx}:\mathsf{A}\rightarrow \mathsf{B}$ in $\mathcal{C}_{r}$ can also be defined as $$[\mathbf{p}_{1}]_{\approx}:\Sigma(\mathsf{A},\mathsf{f}(x)=_{\mathsf{B}}\mathsf{g}(x))\rightarrow \mathsf{A}$$ where $\mathsf{f}(x)=_{\mathsf{B}}\mathsf{g}(x)$ is the object of $\mathbf{Prop}^{r}(\mathsf{A})$ defined as $$\{x'|\,x'=0\wedge x\,\varepsilon\, \mathsf{A}\wedge\{\mathbf{n}\}(x)=\{\mathbf{m}\}(x)\}$$
\end{remark}
\begin{remark}\label{rmk}The hyperdoctrine of realized propositions enjoys also another interesting property. To this purpose
we give the definition of ``separated realized proposition'':  if $\mathsf{A}$ is an object of $\mathcal{C}_{r}$, a realized proposition $\mathsf{P}(x)$ in $\mathbf{Prop}^{r}(\mathsf{A})$ is called \emph{separated} if $$x'\Vdash \mathsf{P}(x)\wedge x'\Vdash \mathsf{P}(y)\vdash_{\tar}x=y $$
It is very easy to show that if $\mathsf{A}$ is an object of $\mathcal{C}_{r}$, then every object of $\mathbf{Prop}^{r}(\mathsf{A})$ is equivalent in $\mathbf{Prop}^{r}(\mathsf{A})$ to a separated one. In fact if $\mathsf{P}(x)$ is an object of $\mathbf{Prop}^{r}(\mathsf{A})$, then we can consider the separated object 
$$\mathsf{P}_{sep}(x):=\{x'|\,p_{1}(x')=x\wedge p_{2}(x')\Vdash \mathsf{P}(x)\}$$
and observe that
 $\mathsf{P}(x)\sim_{\mathsf{A}}\mathsf{P}_{sep}(x)$ follows in \tar.

Notice that this property does not hold in the subobject doctrine of the Effective Topos.

\end{remark}

\begin{definition}We denote with $\overline{\mathbf{Prop}^{r}}$ the posetal reflection (see definition \ref{preorder}) of $\mathbf{Prop}^{r}$ (which coincides with the posetal reflection of $\mathbf{Col}^{r}$).
\end{definition}

Note that
\begin{theorem}\label{ws}
The first-order hyperdoctrine $\overline{\mathbf{Prop}^{r}}$ is naturally isomorphic to the doctrine of weak subobjects $\mathbf{wSub}_{\mathcal{C}^{r}}$ (see definition~\ref{fcc}). 
\end{theorem}
\begin{proof}
This is an immediate consequence of theorems \ref{pdue} and \ref{t3}.
\end{proof}

\subsection{Family of sets and small propositions in $\mathcal{C}_{r}$}

Here we give the notion of {\it families of realized sets}
by adopting the realizability interpretation 
given in \cite{IMMSt} of ``{\it  dependent sets}''
belonging to the type theory $\mtt$ of 
the intensional level of the Minimalist Foundation   in \cite{m09}.

More in detail, following \cite{IMMSt}  we define a universe of sets internally in $\mathcal{C}_{r}$ as a fixpoint of a suitable admissible formula of $\tar$. 
This admissible formula will describe the elements of such a universe as codes of realized sets (which are defined in turn with their elements) in an \emph{inductive} way. However, since  we will not need to use induction on our universe of sets, we do not need to work in the proper theory of inductive definitions with least fixpoints and we can work just in  \tar\ with fixpoints that are not necessarily the least ones.




Now, we proceed by defining in $\tar$ the formulas $set(A)$ expressing that $A$ is a {\it realized set} 
and formulas $x\,\varepsilon \,A$ expressing the notion of {\it realizers of a set $A$}, or {\it elements of the realized set $A$}, inductively {\it by using the notation of sets in the type theory
\  \mtt\ }~\cite{m09} as follows.
 
\noindent
For the emptyset $\mathsf{N}_{0}$ we have the following clauses:
\begin{enumerate}
\item[$0s)$] $set(\mathsf{N}_{0})\equiv^{def} \top$ 
\item[$0\varepsilon)$] $n\,\varepsilon\, \mathsf{N}_{0}\equiv^{def}\bot$
\end{enumerate}
For the singleton set $\mathsf{N}_{1}$ we have the following clauses:
\begin{enumerate}
\item[$1s)$] $set(\mathsf{N}_{1})\equiv^{def} \top$
\item[$1\varepsilon)$] $n\,\varepsilon\, \mathsf{N}_{1}\equiv^{def}n=0$
\end{enumerate}
For the dependent sum $(\Sigma x\in A)B$ of a family of sets $B$ indexed on a set $A$ we have the following clauses:
\begin{enumerate}
\item[$\Sigma s)$] $set((\Sigma x\in A)B)\equiv^{def} set(A)\,\wedge\,\forall x\,(x\,\varepsilon\,A\rightarrow set(B))$
\item[$\Sigma\varepsilon)$] $n\,\varepsilon\,(\Sigma x\in A)B\equiv^{def}set(A)\,\wedge\,\forall x\,(x\,\varepsilon\,A\rightarrow set(B))\,\wedge\, p_{1}(n)\,\varepsilon\,A\,\wedge\,(p_{2}(n)\,\varepsilon\,B)[p_{1}(n)/x]$
\end{enumerate}
For the dependent product $(\Pi x\in A)B$ of a family of sets $B$ indexed on a set $A$ we have the following clauses:
\begin{enumerate}
\item[$\Pi s)$] $set((\Pi x\in A)B)\equiv^{def} set(A)\,\wedge\,\forall x\,(x\,\varepsilon\,A\rightarrow set(B))$
\item[$\Pi\varepsilon)$] $n\,\varepsilon\,(\Pi x\in A)B\equiv^{def} set(A)\,\wedge\,\forall x\,(x\,\varepsilon\,A\rightarrow set(B))\,\wedge\,\forall x\,(x\,\varepsilon \,A\rightarrow \{n\}(x)\,\varepsilon\,B)$
\end{enumerate}
For the disjoint sum $A+B$ of two sets $A$ and $B$ we have the following clauses:
\begin{enumerate}
\item[$+ s)$] $set(A+B)\equiv^{def} set(A)\,\wedge\,set(B)$
\item[$+\varepsilon)$] $n\,\varepsilon\,A+B\equiv^{def}$\\
$set(A)\,\wedge\,set(B)\,\wedge\, ((p_{1}(n)=0\,\wedge\,p_{2}(n)\,\varepsilon\,A)\,\vee\,(p_{1}(n)=1\,\wedge\,p_{2}(n)\,\varepsilon\,B))$
\end{enumerate}
For the set $\mathsf{List}(A)$ of lists of elements of a set $A$ we have the following clauses:
\begin{enumerate}
\item[$\mathsf{L} s)$] $set(\mathsf{List}(A))\equiv^{def} set(A)$
\item[$\mathsf{L}\varepsilon)$] $n\,\varepsilon\,\mathsf{List}(A)\equiv^{def}set(A)\,\wedge\,\forall i\,(i<lh(n)\rightarrow (n)_{i}\,\varepsilon\, A)$
\end{enumerate}
For the propositional identity $\mathsf{Id}(A,x,y)$ of two elements $x,y$ of a set $A$ of \mtt\  we have the following clauses:
\begin{enumerate}
\item[$\mathsf{Id}s)$] $set(\mathsf{Id}(A,x,y))\equiv^{def} set(A)\,\wedge\, x\,\varepsilon A\,\wedge\,y\,\varepsilon \,A$
\item[$\mathsf{Id}\varepsilon)$] $n\,\varepsilon\,\mathsf{Id}(A,x,y)\equiv^{def}set(A)\,\wedge\, x\,\varepsilon A\,\wedge\,y\,\varepsilon \,A\,\wedge\,x=y$
\end{enumerate}

We can make these clauses positive by adding a predicate of \emph{non-realizability} $\noteps$ and using classical logic; clauses $0s,0\varepsilon,1s,1\varepsilon,+s,+\varepsilon, \mathsf{L}s, \mathsf{Id}s,\mathsf{Id}\varepsilon$ remain equal, $\Sigma s, \Sigma \varepsilon, \Pi s, \Pi \varepsilon, \mathsf{L}\varepsilon$ are transformed into the following clauses $\Sigma s', \Sigma \varepsilon', \Pi s', \Pi \varepsilon', \mathsf{L}\varepsilon'$ and moreover we add clauses $0\noteps, 1\noteps, \Sigma\noteps, \Pi\noteps, +\noteps, \mathsf{L}\noteps, \mathsf{Id}\noteps$ for $\noteps$ which are obtained using the fact that $\noteps$ is intended to positively mimic the negation of $\overline{\varepsilon}$ and that in classical logic $\varphi\rightarrow \psi$ and $\neg \varphi \vee \psi$ are equivalent:
\begin{enumerate}
\item[$\Sigma s')$] $set((\Sigma x\in A)B)\equiv^{def} set(A)\,\wedge\,\forall x\,(x\noteps A\,\vee\, set(B))$
\item[$\Sigma \varepsilon')$] $n\,\varepsilon\,(\Sigma x\in A)B\equiv^{def}$
\item[]	$\qquad \qquad set(A)\wedge\forall x\,(x\noteps A\vee set(B))\wedge p_{1}(n)\,\varepsilon\,A\wedge(p_{2}(n)\,\varepsilon\,B)[p_{1}(n)/x]$
\item[$\Pi s')$] $set((\Pi x\in A)B)\equiv^{def} set(A)\,\wedge\,\forall x\,(x\noteps A\,\vee\, set(B))$
\item[$\Pi \varepsilon')$] $n\,\varepsilon\,(\Pi x\in A)B\equiv^{def} set(A)\,\wedge\,\forall x\,(x\noteps A\,\vee\, set(B))\,\wedge\,\forall x\,(x\noteps A\vee \{n\}(x)\,\varepsilon\,B)$
\item[$\mathsf{L}\varepsilon')$] $n\,\varepsilon\,\mathsf{List}(A)\equiv^{def}set(A)\,\wedge\,\forall i\,(i\geq lh(n)\,\vee\, (n)_{i}\,\varepsilon\, A)$
\item[$0\noteps)$] $n\noteps \mathsf{N}_{0}\equiv^{def}\top$
\item[$1\noteps)$] $n\noteps \mathsf{N}_{1}\equiv^{def}n>0$
\item[$\Sigma\noteps)$] $n\noteps(\Sigma x\in A)B\equiv^{def}$\\
$set(A)\,\wedge\,\forall x\,(x\noteps A\vee set(B))\,\wedge\,( p_{1}(n)\noteps A\,\vee\,(p_{2}(n)\noteps B)[p_{1}(n)/x])$
\item[$\Pi\noteps)$] $n\noteps(\Pi x\in A)B\equiv^{def} set(A)\,\wedge\,\forall x\,(x\noteps A\vee set(B))\,\wedge\,\exists x\,(x\,\varepsilon\, A\,\wedge\, \{n\}(x)\noteps B)$
\item[$+\noteps)$] $n\noteps A+B\equiv^{def}$\\
$set(A)\,\wedge\,set(B)\,\wedge\,(\neg (p_{1}(x)=0)\,\vee\,p_{2}(x)\noteps A)\,\wedge\,(\neg (p_{1}(x)=1)\,\vee\,p_{2}(x)\noteps B)$
\item[$\mathsf{L}\noteps)$] $n\noteps\mathsf{List}(A)\equiv^{def}set(A)\,\wedge\,\exists i\,(i< lh(n)\,\wedge\, (n)_{i}\noteps A)$
\item[$\mathsf{Id}\noteps)$] $n\noteps\mathsf{Id}(A,x,y)\equiv^{def}set(A)\,\wedge\, x\,\varepsilon \,A\,\wedge\,y\,\varepsilon \,A\,\wedge\,\neg(x=y)$
\end{enumerate}
The next step consists in rendering such a realizability notion internally in $\tar$ using a fixpoint formula for an admissible formula $\varphi(x,X)$.
In order to do this:
\begin{enumerate}
\item we will encode the above list of \mtt-sets  in $\tar$ as a number;
\item we will use also the \mtt-notion of {\it dependent set $B$ (or family of $B$ of sets) on a set $A$}~\cite{m09}  
which will be interpreted as a code $b$ of a recursive function sending realizers of $A$ to codes of sets; in this case $B[a/x]$ will become $\{b\}(a)$;
\item set constructors will be encoded by numerals $1,...,6$ in such a way that we have the following encodings:
\begin{enumerate}
\item $\mathsf{n}_{0}:=p(1,0)$ and $\mathsf{n}_{1}=p(1,1)$ will encode basic sets $\mathsf{N}_{0}$ and $\mathsf{N}_{1}$ respectively;
\item $\sigma(a,b):=p(2,p(a,b))$, $\pi(a,b):=p(3,p(a,b))$ and $a\oplus b:=p(4,p(a,b))$ will encode sets of the form $(\Sigma x\in A)B$, $(\Pi x\in A)B$ and $A+B$ respectively;
\item $\mathsf{list}(a):=p(5,a)$ and $\mathsf{id}(a,b,c):=p(6,p(a,p(b,c)))$ will encode sets of the form $\mathsf{List}(A)$ and $\mathsf{Id}(A,b,c)$ respectively;
\end{enumerate}
\item we have numerals $\mathbf{n}_{0}$, $\mathbf{n}_{1}$, $\boldsymbol{\sigma}$, $\boldsymbol{\pi}$, $\boldsymbol{\oplus}$, $\mathbf{list}$ and $\mathbf{id}$  which represent these constructors in $\tar$: 
\begin{enumerate}
\item $\tar\vdash \mathsf{n}_{0}=\mathbf{n}_{0}$ and $\tar\vdash \mathsf{n}_{1}=\mathbf{n}_{1}$
\item $\tar\vdash \{\boldsymbol{\sigma}\}(x,y)=\sigma(x,y)$ and $\tar\vdash \{\boldsymbol{\pi}\}(x,y)=\pi(x,y)$
\item $\tar\vdash \{\boldsymbol{\oplus}\}(x,y)=x\oplus y$, $\tar\vdash \{\mathbf{list}\}(x)=\mathsf{list}(x)$ and $\tar\vdash \{\mathbf{id}\}(x,y,z)=\mathsf{id}(x,y,z)$
\end{enumerate}
\item we encode formulas $set(A)$, $x\,\varepsilon\,A$ and $x\noteps A$ as $\mathsf{set}(a)\,\epsilon\,X$, $\mathsf{E}(x,a)\,\epsilon\,X$ and $\overline{\mathsf{E}}(x,a)\,\epsilon\,X$ respectively where $\mathsf{set}(a)$, $\mathsf{E}(x,a)$ and $\overline{\mathsf{E}}(x,a)$ are $p(20,a)$, $p(21,p(x,a))$ and $p(22,p(x,a))$ respectively.
\end{enumerate}
 
A formulation of the admissible formula $\varphi(x,X)$ is the following one:

 {\tiny
 \begin{enumerate}
 \item[]
 \item[$(basic)$] $x=\mathsf{set}(\mathsf{n}_{0})\,\qquad \vee$
 \item[] $\exists y\,(x=\overline{\mathsf{E}}(y,\mathsf{n}_{0}))\,\qquad \vee\,$
 \item[] $x=\mathsf{set}(\mathsf{n}_{1})\,\qquad \vee$
 \item[] $x=\mathsf{E}(0,\mathsf{n}_{1})\,\qquad \vee$
 \item[] $\exists y\,(\, y>0\,\wedge\,x=\overline{\mathsf{E}}(y,\mathsf{n}_{1})\, )$
 \\
 \item[] $\vee$
 \\
 \item[$(\Sigma)$] $\exists y\,\exists z\,[\;\;\mathsf{set}(y)\,\epsilon\,X\ \wedge\ \forall t\,(\, \overline{\mathsf{E}}(t,y)\,\epsilon\,X\ \vee\
 \mathsf{set}(\{z\}(t))\,\epsilon\,X\, )\  \wedge $ 
 \item[] $\qquad\qquad (\;x=\mathsf{set}(\sigma(y,z))\,\qquad \vee$
 \item[] $\qquad\qquad\;\exists u\,(\mathsf{E}(\, p_{1}(u),y)\,\epsilon\,X\,\wedge\,\mathsf{E}(p_{2}(u),\{z\}(p_{1}(u))\,)\,\epsilon\,X \,\wedge\, x=\mathsf{E}(u,\sigma(y,z))\, )\,\qquad \vee\,$
 \item[] $\qquad\qquad\;\exists u\,(\, (\overline{\mathsf{E}}(p_{1}(u),y)\,\epsilon\,X\,\vee\,\overline{\mathsf{E}}(p_{2}(u),\{z\}(p_{1}(u))\, )\,\epsilon\,X\,  )\,\wedge\, x=\overline{\mathsf{E}}(u,\sigma(y,z))\, )\;	\;)\;\;]$
 \\
 \item[]$\vee$
 \\
 \item[$(\Pi)$] $\exists y\,\exists z\,[\;\;\mathsf{set}(y)\,\epsilon\,X\ \wedge\ \forall t\,(\,  \overline{\mathsf{E}}(t,y)\,\epsilon\,X\ \vee\ \mathsf{set}(\{z\}(t))\,\epsilon\,X\, )\, \wedge\, $
 \item[] $\qquad\qquad(\;x=\mathsf{set}(\pi(y,z))\,\vee$
 \item[] $\qquad\qquad\;\exists u\,(\forall t\,(\, \overline{\mathsf{E}}(t,y)\,\epsilon\,X\,\vee\,\mathsf{E}(\{u\}(t),\{z\}(t))\,\epsilon\,X\, ) \,\wedge\, x=\mathsf{E}(u,\pi(y,z))  \, )\,\qquad \vee\,$
 \item[] $\qquad\qquad\;\exists u\,(\, \exists t\,(\, \mathsf{E}(t,y)\,\epsilon\,X\,\wedge\,\overline{\mathsf{E}}(\{u\}(t),\{z\}(t))\,\epsilon\,X\, )\,\wedge\, x=\overline{\mathsf{E}}(u,\pi(y,z)) \, )\;)\;\;]$
 \\
 \item[]$\vee$\\
 \item[$(+)$] $\exists y\,\exists z\,[\;\;\mathsf{set}(y)\,\epsilon\,X\,\wedge\,\mathsf{set}(z)\,\epsilon\,X\wedge $
 \item[] $(\;x=\mathsf{set}(y\oplus z)\,\qquad \vee$
 \item[] $\exists u(\, (\, (\, p_{1}(u)=0\, \wedge\, \mathsf{E}(p_{2}(u),y)\,\epsilon\,X\, )\ \vee\ (\ p_{1}(u)=1\, \wedge\, \mathsf{E}(p_{2}(u),z)\,\epsilon\,X\, )\ )\, \wedge\,  x=\mathsf{E}(u,y\oplus z)\ )\,\qquad \vee\,$
 \item[] $\exists u\, (\ (\, p_{1}(u)\neq 0\ \vee\ \overline{\mathsf{E}}(p_{2}(u),y)\epsilon X\,)\, \wedge\, (\, p_{1}(u)\neq 1\, \vee\, \overline{\mathsf{E}}(p_{2}(u),z)\epsilon X\, )\, \wedge\, x=\overline{\mathsf{E}}(u,y\oplus z))\;) \;\;]\,$
 \\
 \item[]$\vee$\\
 \item[$(\mathsf{List})$] $\exists y\,[\;\;\mathsf{set}(y)\,\epsilon\,X\,\wedge\,$
 \item[] $\qquad\qquad(\;x=\mathsf{set}(\mathsf{list}(y))\,\qquad\vee$
 \item[] $\qquad\qquad\;\exists u\,(\, \forall i(\ i\geq lh(u)\ \vee \ \mathsf{E}((u)_{i},y)\,\epsilon\,X\ ) \ \wedge\ x=\mathsf{E}(u,\mathsf{list}(y))\ )\,\qquad\vee\,$
 \item[] $\qquad\qquad\;\exists u\,(\exists i(i< lh(u)\,\wedge \,\overline{\mathsf{E}}((u)_{i},y)\,\epsilon\,X)\,\wedge\, x=\overline{\mathsf{E}}(u,\mathsf{list}(y)))\;\;\;) \;\;]$\\
 \item[]$\vee$\\
 \item[$(\mathsf{Id})$] $\exists y\,\exists z\,\exists z'\,[\;\;\mathsf{set}(y)\,\epsilon\,X\,\wedge\,\mathsf{E}(z,y)\,\epsilon\,X\,\wedge\,\mathsf{E}(z',y)\,\epsilon\,X\ \wedge\ $
 \item[] $\qquad\qquad\qquad(\;x=\mathsf{set}(\mathsf{id}(y,z,z'))\,\qquad \vee$
 \item[] $\qquad\qquad\qquad\;\exists u\,(z=z' \,\wedge\, x=\mathsf{E}(u,\mathsf{id}(y,z,z')))\,\qquad \vee\,$
 \item[] $\qquad\qquad\qquad\;\exists u\,(\neg(z=z')\,\wedge\, x=\overline{\mathsf{E}}(u,\mathsf{id}(y,z,z')))\;\;\;) \;\;]$
 \end{enumerate}
 }

Now we employ the  admissible formula $\varphi(x,X)$ to  obtain its fixpoint $P_{\varphi}(x)$ in \tar\ by which we define
$$\mathsf{Set}(x)\equiv^{def}P_{\varphi}(\mathsf{set}(x))\qquad x\,\overline{\varepsilon}\,y\equiv^{def}P_{\varphi}(\mathsf{E}(x,y))\qquad x\noteps y\equiv^{def}P_{\varphi}(\overline{\mathsf{E}}(x,y))$$
In order to define a universe made of these encodings in the category $\mathcal{C}_{r}$ we have to pay attention to a side effect of the construction of fixpoints in $\tar$. Indeed, we must add to the formula $\mathsf{Set}(x)$ a coherence requirement which selects only those codes for which $\noteps$ acts as the negation of $\overline{\varepsilon}$; in fact this is not guaranteed by the fixpoint construction in $\tar$, as fixpoints in $\tar$ are not necessarily \emph{least} fixpoints.

So we define the universe of sets in $\mathcal{C}_{r}$ as follows:
$$\mathsf{U}_{\mathsf{S}}:=\{x|\;\mathsf{Set}(x)\,\wedge\,\forall t\,(t\,\overline{\varepsilon}\,x\leftrightarrow \neg\, t\noteps x)\}$$
Then, we employ such a universe
to define a full indexed subcategory of $\mathbf{Col}^{r}$ 
with  families of realized sets.

\begin{definition}
Let $\mathsf{A}$ be an object of $\mathcal{C}_{r}$. The category $\mathbf{Set}^{r}(\mathsf{A})$ is the full subcategory of $\mathbf{Col}^{r}(\mathsf{A})$ whose objects are families of realized collections on $\mathsf{A}$ of the form $$\tau_{\mathsf{A}}(\mathbf{n}):=\{x'|\,x'\,\overline{\varepsilon}\, \{\mathbf{n}\}(x)\,\wedge\, x\,\varepsilon \,\mathsf{A}\}$$ for a numeral $\mathbf{n}$ defining an operation 
 $[\mathbf{n}]_{\approx}:\mathsf{A}\rightarrow \mathsf{U_{S}}$ in $\mathcal{C}_{r}$. 
 The objects of $\mathbf{Set}^{r}(\mathsf{A})$ are called \emph{families of realized sets} on $\mathsf{A}$.
\end{definition}
The following lemma is an immediate consequence of the previous definition.
\begin{lemma}
Suppose $\mathsf{f}:\mathsf{A}\rightarrow \mathsf{B}$ and $\mathsf{g}:[\mathsf{n}]_{\sim}:\mathsf{B}\rightarrow \mathsf{U}_{S}$ are arrows in $\mathcal{C}_{r}$ and suppose  $\mathbf{m}$ is a numeral such that $\mathsf{g}\circ \mathsf{f}=[\mathbf{m}]_{\sim}:\mathsf{A}\rightarrow \mathsf{U}_{S}$.
Then $\mathbf{Col}^{r}_{\mathsf{f}}(\tau_{\mathsf{B}}(\mathbf{n}))=\tau_{\mathsf{A}}(\mathbf{m})$.
\end{lemma}
As a consequence of the previous lemma we can give the following definition.
 
\begin{definition}
 $\mathbf{Set}^{r}:\mathcal{C}_{r}^{op}\rightarrow \mathbf{Cat}$ is the indexed category whose fibre over an object $\mathsf{A}$ of $\mathcal{C}_{r}$ is $\mathbf{Set}^{r}(\mathsf{A})$ and which send an arrow $f:\mathsf{A}\rightarrow \mathsf{B}$ in $\mathcal{C}_{r}$ to the restriction of $\mathbf{Col}^{r}_{f}$ to $\mathbf{Set}^{r}(\mathsf{B})$,  since its image is included in $\mathbf{Set}^{r}(\mathsf{A})$.\end{definition}
 
One can prove as in the case of $\mathbf{Col}^{r}$ the following theorem.
\begin{theorem}\label{setset}
For every $\mathsf{A}$ in $\mathcal{C}_{r}$, $\mathbf{Set}^{r}(\mathsf{A})$ is a finitely complete category with finite coproducts, list objects and weak exponentials. 
\end{theorem}     
\begin{proof}
As a consequence of the clauses determining the fixpoint formula giving rise to $\mathsf{U_{S}}$, in $\mathbf{Col}^{r}(\mathsf{A})$ we have that:
\begin{enumerate} 
\item $\mathsf{1}_{\mathsf{A}}\cong \tau_{\mathsf{A}}(\Lambda x.\mathbf{n}_{1})$
\item $\tau_{\mathsf{A}}(\mathbf{n})\times_{\mathsf{A}}\tau_{\mathsf{A}}(\mathbf{m})\cong \tau_{\mathsf{A}}(\,\Lambda x. \{\boldsymbol{\sigma}\}(\{\mathbf{n}\}(x),\,\Lambda y.\{\mathbf{m}\}(x)))$
\item if $[\mathbf{j}]_{\approx},[\mathbf{k}]_{\approx}:\tau_{\mathsf{A}}(\mathbf{n})\rightarrow \tau_{\mathsf{A}}(\,\mathbf{m})$, then 
$$\mathsf{Eq}([\mathbf{j}]_{\approx},[\mathbf{k}]_{\approx})\cong \tau_{\mathsf{A}}(\Lambda x. \{\boldsymbol{\sigma}\}(\{\mathbf{n}\}(x),\Lambda y. \{\mathbf{id}\}(\,\{\mathbf{m}\}(x),\{\mathbf{j}\}(x,y),\{\mathbf{k}\}(x,y)))) $$
\item $\mathbf{0}_{\mathsf{A}}\cong \tau_{\mathsf{A}}(\,\Lambda x.\mathbf{n}_{0})$
\item $\tau_{\mathsf{A}}(\mathbf{n})+_{\mathsf{A}}\tau_{\mathsf{A}}(\mathbf{m})\cong \tau_{\mathsf{A}}(\,\Lambda x. \{\boldsymbol{\oplus}\}(\,\{\mathbf{n}\}(x),\, \{\mathbf{m}\}(x)))$
\item $\mathsf{List}(\tau_{\mathsf{A}}(\mathbf{n}))\cong \tau_{\mathsf{A}}(\,\Lambda x. \{\mathbf{list}\}(\,\{\mathbf{n}\}(x)))$
\item $\tau_{\mathsf{A}}(\mathbf{n})\Rightarrow \tau_{\mathsf{A}}(\,\mathbf{m})\cong \tau_{\mathsf{A}}(\,\Lambda x. \{\boldsymbol{\pi}\}(\,\{\mathbf{n}\}(x),\, \Lambda y.\{\mathbf{m}\}(x)))$
\end{enumerate}
\end{proof}
\begin{remark}
Notice also that the embedding of $\mathbf{Set}^{r}$ in $\mathbf{Col}^{r}$ preserves finite limits, finite coproducts, list objects and exponentials.  
\end{remark}
Now we need to define a notion of map with set-sized fibers in $\mathcal{C}_{r}$ in order to prove other meaningful properties of $\mathbf{Set}^{r}$.
\begin{definition}\label{repress} An arrow $f:\mathsf{A}\rightarrow \mathsf{B}$ in $\mathcal{C}_{r}$ is \emph{representable} if there exists an object $\mathsf{C}(x)$ in $\mathbf{Set}^{r}(\mathsf{B})$ for which $f$ is isomorphic in $\mathcal{C}_{r}/\mathsf{B}$ to $\mathbf{I}_{\mathsf{B}}(\mathsf{C}(x))$ (see notation in theorem \ref{pdue}).
\end{definition}
Using the encoding of $\Sigma$ and $\Pi$ sets we can prove the following two properties about substitution along representable maps.
\begin{theorem}\label{ts} If $\mathsf{f}:\mathsf{A}\rightarrow \mathsf{B}$ in $\mathcal{C}_{r}$ is representable, then $\mathbf{Set}^{r}_{\mathsf{f}}$ has a left adjoint.\end{theorem}
\begin{proof}
Consider the arrow $\mathsf{p}_{1}^{\Sigma}=[\mathbf{p}_{1}]_{\approx}:\Sigma(\mathsf{B},\tau_{\mathsf{B}}(\mathbf{n}))\rightarrow \mathsf{B}$. Then 
$$\Sigma_{\mathsf{p}_{1}^{\Sigma}}(\tau_{\Sigma(\mathsf{B},\tau_{\mathsf{B}}(\mathbf{n}))}(\mathbf{m}))\simeq \tau_{\mathsf{B}}(\Lambda x. \{\boldsymbol{\sigma}\}( \{\mathbf{n}\}(x),\Lambda y. \{\mathbf{m}\}(\{\mathbf{p}\}(x,y))))$$ \end{proof}

\begin{theorem}\label{ts2} If $\mathsf{f}:\mathsf{A}\rightarrow \mathsf{B}$ in $\mathcal{C}_{r}$ is representable, then for every object $\mathsf{C}(x)$ in $\mathbf{Set}^{r}(\mathsf{A})$, $\Pi^{r}_{\mathsf{f}}(\mathsf{C}(x))$ is isomorphic in $\mathbf{Col}^{r}(\mathsf{B})$ to an object of $\mathbf{Set}^{r}(\mathsf{B})$.
\end{theorem}
\begin{proof}
Consider the arrow $\mathsf{p}_{1}^{\Sigma}=[\mathbf{p}_{1}]_{\approx}:\Sigma(\mathsf{B},\tau_{\mathsf{B}}(\mathbf{n}))\rightarrow \mathsf{B}$. Then 
$$\Pi_{\mathsf{p}_{1}^{\Sigma}}(\tau_{\Sigma(\mathsf{B},\tau_{\mathsf{B}}(\mathbf{n}))}(\mathbf{m}))\simeq \tau_{\mathsf{B}}((\Lambda x. \{\boldsymbol{\pi}\}( \{\mathbf{n}\}(x),\Lambda y. \{\mathbf{m}\}(\{\mathbf{p}\}(x,y))))$$\end{proof}

Now we are going to define the indexed category of small realized propositions
as the preorder reflection of the indexed category of realized sets.
Therefore small realized propositions coincide with  those
realized propositions that are also fibre objects of $\mathbf{Set}^{r}$:
\begin{definition} The functor $\mathbf{Prop}^{r}_{s}:\mathcal{C}_{r}^{op}\rightarrow \mathbf{Pord}$ is defined as the preorder reflection $\mathbf{P}[\mathbf{Set}^{r}]$ (see definition \ref{preorder}) of the indexed category $\mathbf{Set}^{r}$. Then,  fibre objects of $\mathbf{Prop}^{r}_{s}$ are called {\it small realized propositions}.
\end{definition}
Similarly to the case of $\mathbf{Prop}^{r} $ one can deduce from \ref{setset} the following theorem.
\begin{theorem}For every object $\mathsf{A}$ in $\mathcal{C}_{r}$, $\mathbf{Prop}^{r}_{s}(\mathsf{A})$ is a Heyting prealgebra and for every arrow $\mathsf{f}$ in $\mathcal{C}_{r}$, $(\mathbf{Prop}^{r}_{s})_{\mathsf{f}}$ preserves the structure of Heyting prealgebra.\end{theorem}
The following theorem follows directly from theorems \ref{ts} and \ref{ts2}:
\begin{theorem}\label{tsp} 
If $\mathsf{f}:\mathsf{A}\rightarrow \mathsf{B}$ in $\mathcal{C}_{r}$ is representable, then $(\mathbf{Prop}^{r}_{s})_{\mathsf{f}}$ has left and right adjoints satisfying Beck-Chevalley conditions.
\end{theorem}

\begin{definition}
We define $\overline{\mathbf{Prop}^{r}_{s}}$ as the posetal reflections (see definition \ref{preorder}) of the doctrine $\mathbf{Prop}^{r}_{s}$ and we use the same notations for order, bottom, top, binary infima and suprema, Heyting implication and left and right adjoints $\exists$ and $\forall$.
\end{definition}

\subsection{The definition of the Effective Kleene \mf-tripos}
Now we are ready to give the definition of the Effective Kleene \mf-tripos.
\begin{definition} The $5$-uple $(\mathcal{C}_{r}, \mathbf{Col}^{r}, \mathbf{Set}^{r}, \mathbf{Prop}^{r},\mathbf{Prop}^{r}_{s})$ is called the \emph{Effective  Kleene \mf-tripos} in \tar.
\end{definition}
Note that  the mentioned indexed categories in the definition of the
 Effective Kleene \mf-tripos
 make the following
diagram commute in $\mathbf{Cat}^{\mathcal{C}_{r}^{op}}$
{\small \def\objectstyle{\scriptstyle}
\def\labelstyle{\scriptstyle}
{
$$\xymatrix@-1pc{
\mathbf{Set}^{r}\ar @{^{(}->}[r]			&\mathbf{Col}^{r}\\
\mathbf{Prop}^{r}_{s}\ar @{^{(}->}[r]\ar @{^{(}->}[u]			&\mathbf{Prop}^{r}\ar @{^{(}->}[u]\\
}$$}}

The Effective Kleene \mf-tripos is so called because its logical part   extends Kleene realizability of intuitionistic connectives and it has enough
structure to support an interpretation (via combinators) of the intensional level $\mtt$ of the Minimalist Foundation in \cite{m09} (as shown in \cite{IMMSt}) and to produce  
a predicative variant of $\eff$ once the elementary quotient completion is applied to it.

\subsection{The internal language of the doctrine $\thprop$}
Here we define a {\it fragment} of the internal language of the doctrine $\thprop$ to easily prove that  it extends Kleene realizability interpretation of Heyting arithmetics and hence it validates the Formal Church's thesis.

\begin{enumerate}
\item For every object $\mathsf{A}$ of $\mathcal{C}_{r}$ there is a list of variables $x_{1}^{\mathsf{A}},...,x_{n}^{\mathsf{A}}...$;
\item a \emph{context} is a (possibly empty) finite list of objects and distinct variables $[\xi_{1}\in\mathsf{A}_{1},...,\xi_{n}\in\mathsf{A}_{n}]$, where for every $i=1,...,n$, the variable $\xi_{i}$ is $x_{k}^{\mathsf{A}_{i}}$ for some $k$;
\item if $\Gamma=[\xi_{1}\in\mathsf{A}_{1},...,\xi_{n}\in\mathsf{A}_{n}]$ is a context, then $\xi_{i}[\Gamma]$ is a term in context of type $\mathsf{A}_{i}$;
\item if $\mathsf{e}:\mathsf{1}\rightarrow \mathsf{B}$, then $\mathsf{e}[\Gamma]$ is a term in context of type $\mathsf{B}$;
\item if $t_{1}[\Gamma],..,t_{n}[\Gamma]$ are terms in context of type $\mathsf{A}_{1},...,\mathsf{A}_{n}$ respectively and \\ $\mathsf{f}:\mathsf{A}_{1}\times...\times\mathsf{A}_{n}\rightarrow \mathsf{B}$ is an arrow of $\mathcal{C}_{r}$, then $\mathsf{f}(t_{1},...,t_{n})[\Gamma]$ is a term in context of type $\mathsf{B}$;
\item if $\mathsf{P}$ is an element of $\overline{\mathbf{Prop}^{r}}(\mathsf{1})$, then $\mathsf{P}[\Gamma]$ is a formula in context;
\item if $t_{1}[\Gamma],..,t_{n}[\Gamma]$ are terms in context of type $\mathsf{A}_{1},...,\mathsf{A}_{n}$ respectively and $\mathsf{R}$ is an element of \\$\overline{\mathbf{Prop}^{r}}(\mathsf{A}_{1}\times...\times\mathsf{A}_{n})$, then $\mathsf{R}(t_{1},...,t_{n})[\Gamma]$ is a formula in context;
\item if $t[\Gamma]$ and $s[\Gamma]$ are terms in context of type $\mathsf{A}$, then $(t=_{\mathsf{A}}s)[\Gamma]$ is a formula in context;
\item $\bot[\Gamma]$ is a formula in context;
\item if $\phi[\Gamma]$ and $\phi'[\Gamma]$ are formulas in context, then $(\phi\wedge \phi')[\Gamma]$, $(\phi\vee \phi')[\Gamma]$ and $(\phi\rightarrow \phi')[\Gamma]$ are formulas in context;
\item if $\phi[\Gamma,\xi\in \mathsf{A}]$ is a formula in context, then $(\exists \xi\in \mathsf{A})\phi[\Gamma]$ and $(\forall \xi\in \mathsf{A})\phi[\Gamma]$ are formulas in context. 
\end{enumerate}
Every term and formula in context of the internal language of $\overline{\mathbf{Prop}^{r}}$ is interpreted in the Predicative Effective Kleene \mf-tripos as follows.
Suppose $\Gamma$ is the context $[\xi_{1}\in\mathsf{A}_{1},...,\xi_{n}\in\mathsf{A}_{n}]$, then 
\begin{enumerate}
\item $\left\|[\;]\right\|:=\mathsf{1}$ and $\left\|\Gamma\right\|:=\mathsf{A}_{1}\times....\times \mathsf{A}_{n}$;
\item $\left\|\xi_{i}[\Gamma]\right\|:=\mathsf{p}_{i}:\left\|\Gamma\right\|\rightarrow \mathsf{A}_{i}$;
\item $\left\|\mathsf{e}[\Gamma]\right\|:=\mathsf{e}\,\circ\, !_{\left\|\Gamma\right\|,\mathsf{1}}:\left\|\Gamma\right\|\rightarrow cod(\mathsf{e})\footnote{Here and later we use $cod$ to denote the codomain of an arrow and $dom$ to denote its domain.}$
\item $\left\|\mathsf{f}(t_{1},..,t_{n})[\Gamma]\right\|:=\mathsf{f}\circ \langle \left\|t_{1}[\Gamma]\right\|,...,\left\|t_{n}[\Gamma]\right\| \rangle:\left\|\Gamma\right\|\rightarrow cod(\mathsf{f})$;
\item $\left\|\mathsf{P}[\Gamma]\right\|:=\overline{\mathbf{Prop}^{r}}_{!_{\left\|\Gamma\right\|,\mathsf{1}}}(\mathsf{P})\in \overline{\mathbf{Prop}^{r}}(\left\|\Gamma\right\|)$
\item $\left\|\mathsf{R}(t_{1},..,t_{n})[\Gamma]\right\|:=\overline{\mathbf{Prop}^{r}}_{\langle \left\|t_{1}[\Gamma]\right\|,...,\left\|t_{n}[\Gamma]\right\|\rangle}(\mathsf{R})\in \overline{\mathbf{Prop}^{r}}(\left\|\Gamma\right\|)$;
\item $\left\|(t=_{\mathsf{A}}s)[\Gamma]\right\|:=\overline{\mathbf{Prop}^{r}}_{\langle \left\|t[\Gamma]\right\|, \left\|s[\Gamma]\right\|\rangle}(\exists_{\Delta_{\mathsf{A}}}(\top_{\mathsf{A}}))$;
\item $\left\|(\phi\wedge\phi')[\Gamma]\right\|:=\left\|\phi[\Gamma]\right\|\sqcap_{\left\|\Gamma\right\|} \left\|\phi'[\Gamma]\right\|\in \overline{\mathbf{Prop}^{r}}(\left\|\Gamma\right\|)$, 
\item $\left\|(\phi\vee\phi')[\Gamma]\right\|:=\left\|\phi[\Gamma]\right\|\sqcup_{\left\|\Gamma\right\|} \left\|\phi'[\Gamma]\right\|\in \overline{\mathbf{Prop}^{r}}(\left\|\Gamma\right\|)$,
\item $\left\|(\phi\rightarrow\phi')[\Gamma]\right\|:=\left\|\phi[\Gamma]\right\|\Rightarrow_{\left\|\Gamma\right\|} \left\|\phi'[\Gamma]\right\|\in \overline{\mathbf{Prop}^{r}}(\left\|\Gamma\right\|)$;
\item $\left\|(\exists \xi\in \mathsf{A})\phi[\Gamma]\right\|:=\exists_{\langle\mathsf{p}_{1},..,\mathsf{p}_{n}\rangle}(\left\|\phi[\Gamma,\xi\in\mathsf{A}]\right\|)\in \overline{\mathbf{Prop}^{r}}(\left\|\Gamma\right\|)$,
\item $\left\|(\forall \xi\in \mathsf{A})\phi[\Gamma]\right\|:=\forall_{\langle\mathsf{p}_{1},..,\mathsf{p}_{n}\rangle}(\left\|\phi[\Gamma,\xi\in\mathsf{A}]\right\|)\in \overline{\mathbf{Prop}^{r}}(\left\|\Gamma\right\|)$.
\end{enumerate}
If $\phi[\Gamma]$ is a formula in context of the internal language of $\overline{\mathbf{Prop}^{r}}$, we define its validity as follows:
$$\overline{\mathbf{Prop}^{r}}\vdash \phi[\Gamma]\textnormal{ if and only if }\top_{\left\|\Gamma\right\|}\sqsubseteq_{\left\|\Gamma\right\|} \left\|\phi[\Gamma]\right\|.$$
The language of Heyting arithmetic $\mathsf{HA}$ can be translated into the internal language of $\overline{\mathbf{Prop}^{r}}$, as every primitive $n$-ary recursive function $f$ can be represented in $\mathcal{C}_{r}$ by an arrow $$\overline{f}:\underbrace{\mathsf{N}\times....\times\mathsf{N}}_{n\textnormal{ times}}\rightarrow \mathsf{N}.$$ The translation which assigns  a term $\overline{t}$ of the language of $\overline{\mathbf{Prop}^{r}}$ to every term $t$ of $\mathsf{HA}$ 
 and  a formula $\overline{\phi}$ in the internal language of $\overline{\mathbf{Prop}^{r}}$ to every formula $\phi$ of $\mathsf{HA}$ can be defined as follows:\\
\begin{enumerate}
\item $\overline{x_{i}}$ is $x_{i}^{\mathsf{N}}$ for every $i=1,...,n,...$ and $\overline{f(t_{1},...,t_{n})}$ is $\overline{f}(\overline{t_{1}},...,\overline{t_{n}})$,
\item $\overline{t=s}$ is $\overline{t}=_{\mathsf{N}}\overline{s}$ and $\overline{\bot}$ is $\bot$, 
\item $\overline{\phi\wedge \psi}$ is $\overline{\phi}\wedge \overline{\psi}$, $\overline{\phi\vee \psi}$ is $\overline{\phi}\vee \overline{\psi}$ and $\overline{\phi\rightarrow \psi}$ is $\overline{\phi}\rightarrow \overline{\psi}$,
\item $\overline{\forall \xi\, \phi}$ is $(\forall \overline{\xi}\in \mathsf{N}) \overline{\phi}$ and $\overline{\exists \xi\, \phi}$ is $(\exists \overline{\xi}\in \mathsf{N}) \overline{\phi}$.\\
\end{enumerate}
We can prove that the formulations of the principle  of  Axiom of choice and of the Formal Church's thesis
with weak exponentials are validated by $\overline{\mathbf{Prop}^{r}}$.
\begin{theorem}[axiom of choice with weak exponentials]
For every objects $\mathsf{A}$ and $\mathsf{B}$ in $\mathcal{C}_{r}$ and $\mathsf{R}\in \overline{\mathbf{Prop}^{r}}(A\times B)$ 
$$\overline{\mathbf{Prop}^{r}}\vdash (\forall x\in \mathsf{A})(\exists y\in \mathsf{B})\mathsf{R}(x,y)\rightarrow (\exists f\in \mathsf{A}\Rightarrow \mathsf{B})(\forall x\in \mathsf{A})\mathsf{R}(x,\mathsf{ev}(f,x)).\footnote{\label{nb}Here and later it is intended that variables have the right type, i.\,e.\ the formula is well defined.}$$
\end{theorem}
\begin{proof}
This fact is immediate once one uses remark \ref{remark}.
\end{proof}

The next theorem  follows easily from the fact that the interpretation of the internal language of $\overline{\mathbf{Prop}^{r}}$ extends  Kleene realizability interpretation of Heyting arithmetics:
\begin{lemma}
Suppose $\phi$ is a formula of Heyting arithmetic $\mathsf{HA}$ and $\mathbf{n}$ is a numeral.  If $\mathsf{HA}\vdash \mathbf{n}\Vdash \phi$, then $\overline{\mathbf{Prop}^{r}}\vdash \overline{\phi}$.
\end{lemma}

\begin{theorem}[Formal Church's Thesis]
\label{wect}
{\small
$$\overline{\mathbf{Prop}^{r}}\vdash (\forall \overline{x}\in \mathsf{N})(\exists \overline{z}\in \mathsf{N})\mathsf{R}(\overline{x},\overline{z})\rightarrow(\exists \overline{e}\in \mathsf{N})(\forall \overline{x}\in \mathsf{N})(\exists \overline{y}\in \mathsf{N})(\overline{T(e,x,y)}\wedge \mathsf{R}(\overline{y},\overline{U(y)}))$$}
where $T(e,x,y)$ and $U(y)$ are Kleene's predicate and primitive recursive function respectively.
\end{theorem}
Finally we prove the following important result.
\begin{theorem}[Choice rule]\label{choice rule}
If $\mathsf{A}$ and $\mathsf{B}$ are objects in $\mathcal{C}_{r}$, $\Gamma$ is a context of the internal language of $\overline{\mathbf{Prop}^{r}}$ with list of variables $\overline{x}$, $\mathsf{R}\in \overline{\mathbf{Prop}^{r}}(\left\|\Gamma\right\|\times A\times B)$ and $$\overline{\mathbf{Prop}^{r}}\vdash (\forall x\in \mathsf{A})(\exists y\in \mathsf{B})\mathsf{R}(\overline{x},x,y)\,[\Gamma]$$
then there exists an arrow $\mathsf{f}:\left\|\Gamma\right\|\times\mathsf{A}\rightarrow \mathsf{B}$ such that $$\overline{\mathbf{Prop}^{r}}\vdash (\forall x\in A)\mathsf{R}(\overline{x},x,\mathsf{f}(\overline{x},x))\,[\Gamma]$$
\end{theorem}
\begin{proof}
Let's prove the case in which $\Gamma$ is the empty context, as the general case is just a simple variation of this one.

Suppose $\mathsf{A}$ and $\mathsf{B}$ are object in $\mathcal{C}_{r}$ and $\mathsf{R}=[\mathsf{R}'(x)]_{\sim_{\mathsf{A}\times \mathsf{B}}}\in \overline{\mathbf{Prop}^{r}}(\mathsf{A}\times \mathsf{B})$. 

If $\overline{\mathbf{Prop}^{r}}\vdash (\forall x\in \mathsf{A})(\exists y\in \mathsf{B})\mathsf{R}(x,y)$, then $\top_{\mathsf{A}}\sqsubseteq_{\mathsf{A}}\Sigma'_{\mathsf{p}_{1}}(\mathsf{R}'(x))$.

In particular this means that there exists a numeral $\mathbf{r}$ such that 
$$x\,\varepsilon\, \mathsf{A}\wedge x'=0\vdash_{\tar}x\,\varepsilon\, \mathsf{A} \wedge p_{1}(\{\mathbf{r}\}(x,x'))\,\varepsilon\,\mathsf{B}\wedge p_{2}(\{\mathbf{r}\}(x,x'))\,\varepsilon\,\mathsf{R}(p(x,p_{1}(\{\mathbf{r}\}(x,x'))))$$
The arrow $\mathsf{f}:\mathsf{A}\rightarrow \mathsf{B}$ we are looking for can be defined as $[\Lambda x.\{\mathbf{p}_{1}\}(\{\mathbf{r}\}(x,0))]_{\approx}$. \end{proof}

\section{A strictly predicative version of Hyland's Effective Topos}
Here we are going to build the category $\peff$ which we propose
as a strictly predicative variant of the Effective Topos $\eff$ in $\tar$.

We construct
$\peff$  by applying the elementary quotient completion construction introduced in \cite{qu12}  to the hyperdoctrine $\thprop$ of the  Effective Kleene \mf-tripos. We do this for two reasons.
First, from \cite{m09} we know that the elementary quotient completion of $\thprop$
has enough structure to support an interpretation of the extensional level
of the Minimalist Foundation given that the doctrine $\thprop$  supports a (non-categorical)
interpretation of its intensional level.

Second, we know that for $\thprop$ the elementary quotient completion happens to coincide with the exact on lex completion
of the category of realized collections $\mathcal{C}_{r}$.
This follows from the fact that in
 \cite{MR16} it was shown that  the elementary quotient
completion of an elementary doctrine $\mathcal{D}$ on a lex base category $\mathcal{B}$ is equivalent to the exact completion on a lex category $\mathcal{B}$ precisely
when a choice rule holds in $\mathcal{D}$ and this rule holds in $\thprop$
as theorem~\ref{choice rule} shows.

From results in \cite{qu12}, by construction $\peff$ comes equipped with 
a Lawvere's hyperdoctrine that we call $\mprop$ and that is closed under stable effective
quotient of equivalence relations defined by $\mprop$.

\subsection{The category $\peff$}
\begin{definition}\label{efftop} The base category $\peff $ of the \emph{Predicative Effective p-Topos} is the category $\mathcal{Q}_{\overline{\mathbf{Prop}^{r}}}$ obtained by applying the elementary quotient completion
in \cite{qu12} to the doctrine $\thprop$, i.\,e.\ it is the category whose objects are the pairs $(\mathsf{A},\mathsf{R})$ where $\mathsf{A}$ is an object of the category $\mathcal{C}_{r}$ and $\mathsf{R}$ is an object of $\overline{\mathbf{Prop}^{r}}(\mathsf{A}\times \mathsf{A})$ for which the following hold:\\
\begin{enumerate}
\item $\overline{\mathbf{Prop}^{r}}\vdash(\forall x\in \mathsf{A})\mathsf{R}(x,x)$;
\item $\overline{\mathbf{Prop}^{r}}\vdash(\forall x\in \mathsf{A})(\forall y\in \mathsf{A})(\mathsf{R}(x,y)\rightarrow \mathsf{R}(y,x))$;
\item $\overline{\mathbf{Prop}^{r}}\vdash(\forall x\in \mathsf{A})(\forall y\in \mathsf{A})(\forall z\in \mathsf{A})(\mathsf{R}(x,y)\wedge \mathsf{R}(y,z)\rightarrow \mathsf{R}(x,z))$.\\
\end{enumerate}
An arrow from $(\mathsf{A},\mathsf{R})$ to $(\mathsf{B},\mathsf{S})$ is an equivalence class $[\mathsf{f}]_{\simeq}$ of arrows $\mathsf{f}:\mathsf{A}\rightarrow \mathsf{B}$ in $\mathcal{C}_{r}$ such that 
$$\overline{\mathbf{Prop}^{r}}\vdash(\forall x\in \mathsf{A})(\forall y\in \mathsf{A})(\mathsf{R}(x,y)\rightarrow \mathsf{S}(\mathsf{f}(x),\mathsf{f}(y)))$$
with respect to the equivalence relation defined by 
$$\mathsf{f}\simeq \mathsf{g} \textnormal{ if and only if }\overline{\mathbf{Prop}^{r}}\vdash (\forall x\in \mathsf{A})(\mathsf{R}(x,x)\rightarrow \mathsf{S}(\mathsf{f}(x),\mathsf{g}(x))).$$
\end{definition}

As pointed out in $\cite{MR16}$ the elementary completion coincides with the ex/lex completion by Carboni and Celia Magno \cite{excom} when we consider the doctrine of weak subobjects $\mathbf{wSub}$ of a finitely complete category.  As we proved theorem \ref{ws}, this applies to our case.
\begin{theorem}\label{compl}
$$\peff\cong (\mathcal{C}_{r})_{ex/lex}$$
In particular $\peff$ is an exact category.
\end{theorem}

From theorems \ref{compl} and \ref{weakly}, using the main result in \cite{Ros}, we get the following theorem. 
\begin{theorem}\label{mairos} $\peff$ is a locally cartesian closed category.
\end{theorem}

Even more we can show that $\peff$ is actually a list-arithmetic pretopos
by proving:

\begin{lemma}
$\peff$ has disjoint and stable finite coproducts and list objects.
\end{lemma}
\begin{proof}
The existence of disjoint stable coproducts follows essentially from results in \cite{Car}.
\begin{enumerate}
\item an initial object is defined by $(\mathsf{0},\top_{\mathsf{0}\times \mathsf{0}})$.
\item a binary coproduct for $(\mathsf{A},\mathsf{R})$ and $(\mathsf{B},\mathsf{S})$ is given by the object
{\small
$$\begin{array}{lc}(\mathsf{A}+\mathsf{B}\, ,\, & \  ||\ 
 (\exists t\in \mathsf{A})\,(\exists s\in \mathsf{A})\,(\ \mathsf{R}(t,s)\,\wedge\, \mathsf{p}_{1}(x)=_{\mathsf{A}+\mathsf{B}}\mathsf{j}_{1}(t)\,\wedge \,\mathsf{p}_{2}(x)=_{\mathsf{A}+\mathsf{B}}\mathsf{j}_{1}(s)\ )\\
&\vee\\
&(\exists t\in \mathsf{B})\,(\exists s\in \mathsf{B})\,(\ \mathsf{S}(t,s)\,\wedge\, \mathsf{p}_{1}(x)=_{\mathsf{A}+\mathsf{B}}\mathsf{j}_{2}(t)\,\wedge\, \mathsf{p}_{2}(x)=_{\mathsf{A}+\mathsf{B}}\mathsf{j}_{2}(s)\ )\\
& \,[x\in (\mathsf{A}+\mathsf{B})\times (\mathsf{A}+\mathsf{B})]\ \  ||\ )\end{array}$$
}
together with the injections $[\mathsf{j}_{1}]_{\simeq}$ and $[\mathsf{j}_{2}]_{\simeq}$.
\item a list object for $(\mathsf{A},\mathsf{R})$ has as underlying object $\mathsf{List}(\mathsf{A})$ and equivalence relation given by
{\small
$$\begin{array}{c}
||\ \ \mathsf{lh}(\mathsf{p}_{1}(x))=_{\mathsf{N}}\mathsf{lh}(\mathsf{p}_{2}(x))\\
 \wedge\\
\begin{array}{c}
 (\forall n\in \mathsf{N})\ \ \\
((\exists a\in \mathsf{A})(\exists b\in \mathsf{A})(\ \mathsf{R}(a,b)
\wedge\mathsf{j}_{1}(a)=_{\mathsf{A}+\mathsf{1}}\mathsf{comp}(\mathsf{p}_{1}(x),n)\wedge\mathsf{j}_{1}(b)=_{\mathsf{A}+\mathsf{1}}\mathsf{comp}(\mathsf{p}_{2}(x),n) )\\
\vee
\\
 (\mathsf{comp}(\mathsf{p}_{1}(x),n)=_{\mathsf{A}+\mathsf{1}}\mathsf{j}_{2}(0)\,\wedge\,\mathsf{comp}(\mathsf{p}_{2}(x),n)=_{\mathsf{A}+\mathsf{1}}\mathsf{j}_{2}(0))\ )\end{array}\\
\qquad \qquad \qquad \qquad \,[x\in \mathsf{List}(\mathsf{A})\times \mathsf{List}(\mathsf{A})]\ \ ||
\end{array}$$
}
where $\mathsf{lh}:\mathsf{List}(\mathsf{A})\rightarrow \mathsf{N}$ is the length arrow and $\mathsf{comp}:\mathsf{List}(\mathsf{A})\times \mathsf{N}\rightarrow \mathsf{A}+\mathsf{1}$ is the component arrow\footnote{The component arrow sends a pair $((l_{0},...,l_{n}),j)$ to $(0,l_{j})$ if $j\leq n$ and to $(1,0)$ otherwise.} in the category $\mathcal{C}_{r}$. The empty list arrow and the append arrow are given by $[\epsilon]_{\simeq}$ and $[\mathsf{cons}]_{\simeq}$ respectively.
\item Stability and disjointness of coproducts follow from a direct verification. 
\end{enumerate}
\end{proof}

Now we give the definition of the hyperdoctrine $\mprop$ of propositions  associated by construction to $\peff$ as an elementary quotient completion in \cite{qu12} (this doctrine defines the equivalence relations with respect to which the elementary quotient completion is closed under effective quotients):
\begin{definition}
The functor $\mprop:\peff^{op}\rightarrow \mathbf{Pord}$
is defined as follows (see \ref{qu12}):\\
\begin{enumerate}
\item $\mathsf{P}\in \mprop(\mathsf{A},\mathsf{R})\textnormal{ iff }\mathsf{P}\in \overline{\mathbf{Prop}^{r}}(\mathsf{A})\textnormal{ and }$
$$\overline{\mathbf{Prop}^{r}}\vdash (\forall x\in \mathsf{A})(\forall y\in \mathsf{A})(\mathsf{P}(x)\wedge \mathsf{R}(x,y)\rightarrow \mathsf{P}(y))$$
\item the fibre preorder of $\mprop(\, ((\mathsf{A},\mathsf{R}))$ for objects
$P,Q$ is defined as follows: $\mathsf{P}\sqsubseteq_{(\mathsf{A},\mathsf{R})}\mathsf{Q}\textnormal{ iff }\mathsf{P}\sqsubseteq_{\mathsf{A}}\mathsf{Q};$
\item $\mprop_{[\mathsf{f}]_{\simeq}}(\mathsf{P}):=\overline{\mathbf{Prop}^{r}}_{\mathsf{f}}(\mathsf{P})$. \\
\end{enumerate}
\end{definition}

From results in $\cite{MR16}$ and theorem \ref{ws} it follows that:
\begin{theorem}
The hyperdoctrine  $\mprop$ of $\peff$ is equivalent to its subobject doctrine, i.e.
$$\mprop \cong \mathbf{Sub}_{\peff}$$
\end{theorem}

Moreover from results in \cite{qu12}, theorem \ref{choice rule} and results in \cite{MR16} we get
\begin{theorem}
$\mprop $ is a first-order hyperdoctrine. In particular $\peff$ is a Heyting category.
Moreover, the rule of unique choice holds for $\mprop $,  i.\,e.\ if 
$$\mprop\vdash (\forall x\in (\mathsf{A},\mathsf{R}))\,(\exists !\, y\in (\mathsf{B},\mathsf{S}))\,\mathsf{P}(x,y)$$
then there exists an arrow $\mathsf{f}:(\mathsf{A},\mathsf{R})\rightarrow(\mathsf{B},\mathsf{S})$ such that 
$$\mprop\vdash (\forall x\in (\mathsf{A},\mathsf{R}))\,\mathsf{P}(x,\mathsf{f}(x))\footnote{The internal language of $\mprop$ is defined analogously to that of $\overline{\mathbf{Prop}^{r}}$.}$$
\end{theorem}

An important property of the doctrine $\mprop$ is that 
it validates the Formal Church's Thesis as a consequence of the fact
that the underlying doctrine $\thprop$ validates its weak form in lemma~\ref{wect}:
\begin{theorem}[Formal Church's thesis in $\mprop$]
{\small
$$\mprop\vdash (\forall \overline{x}\in \mathsf{N}_{\peff})(\exists \overline{z}\in \mathsf{N}_{\peff})\mathsf{R}(\overline{x},\overline{z})\rightarrow$$
$$\qquad\qquad(\exists \overline{e}\in \mathsf{N}_{\peff})(\forall \overline{x}\in \mathsf{N}_{\peff})(\exists \overline{y}\in \mathsf{N})(\overline{T(e,x,y)}\wedge \mathsf{R}(\overline{y},\overline{U(y)}))$$}
\end{theorem}

We can now define an hyperdoctrine of small propositions on $\peff$ as follows:

\begin{definition}
The functor $$\mprops:\peff^{op}\rightarrow \mathbf{Pord}$$
is the subfunctor of $\mprop$ whose value on $(\mathsf{A},\mathsf{R})$ in $\peff$ is the full subcategory $\mprops(\mathsf{A},\mathsf{R})$ of $\mprop(\mathsf{A},\mathsf{R})$ whose objects are those $\mathsf{P}$ which are both in $\overline{\mathbf{Prop}^{r}_{s}}(\mathsf{A})$ and in $\mprop(\mathsf{A},\mathsf{R})$.
\end{definition}

\begin{theorem}\label{class}
There exists an object $\Omega$ in $\peff$ which represents $\mprops$, i.\,e.\ there is a natural isomorphism between $\mprops(-)$ and $\peff(-,\Omega)$.
\end{theorem}
\begin{proof}
It is sufficient to define the object $\Omega$ of $\peff$ as follows:
$$\Omega:=(\mathsf{U}_{\mathsf{S}}, [\mathsf{EQ}(x)]_{\sim_{\mathsf{U}_{\mathsf{S}}\times\mathsf{U}_{\mathsf{S}} }})$$
where $x'\Vdash\mathsf{EQ}(x)$ is defined as 
$$\forall t\,(t\,\overline{\varepsilon}\,p_{1}(x)\rightarrow \{p_{1}(x')\}(t)\,\overline{\varepsilon}\,p_{2}(x))\wedge \forall s(s\,\overline{\varepsilon}\,p_{2}(x)\rightarrow \{p_{2}(x')\}(s)\,\overline{\varepsilon}\,p_{1}(x))$$
\end{proof}

The previous results about $\peff$ can be summarized in the following theorem.
\begin{theorem}
The category $\peff$ is a locally cartesian closed list-arithmetic pretopos with a classifier for $\mprops$.
\end{theorem}

\subsection{The fibration of sets over $\peff$}
The aim of this section is to introduce a notion of {\it a family of  sets depending on an object $(\mathsf{A},\mathsf{R})$ of $\peff$} that will give rise to a full subcategory of
the slice category $\peff/(\mathsf{A},\mathsf{R})$ whose subobjects 
are comprehensions of {\it small propositions}. 

It is well known that in a finitely complete category $\mathcal{C}$
we can represent a family of $\mathcal{C}$-objects depending on an object $A$ of
$\mathcal{C}$ as an arrow $B\rightarrow A$ of $\mathcal{C}$ and that we can
organize such families into the so called codomain fibration $\mathcal{C}^\rightarrow\rightarrow \mathcal{C}$ (see \cite{jacobbook}).

In section~\ref{indexcol} we have seen that for the category $C_r$
the codomain fibration admits a splitting via the indexed category of realized collections. 

In the case of $\peff$, a splitting of its associated codomain fibration seems out of reach for the lack not only of a functorial  choice of pullbacks but even of a choice of generic pullbacks due to the fact that $\peff$ is an exact completion.

So we need to proceed by simply adopting the fibrational approach.
Then, the idea, which we borrow from algebraic set theory \cite{MJ}, is to define a family of sets depending on an object 
 $(\mathsf{A},\mathsf{R})$ of $\peff$ as an arrow  
$$[\mathsf{f}]_{\simeq}:(\mathsf{B},\mathsf{S})\longrightarrow (\mathsf{A},\mathsf{R})$$
whose antimages are defined via realized sets in  $\mathcal{C}_{r}$ in some way.

To this purpose, we observe that in our setting,
as in the elementary quotient completion in \cite{m09},
the slice category of $\peff$ over an object $(\mathsf{A},\mathsf{R})$ 
is {\it equivalent} to the {\it category of internal categorical diagrams on $(\mathsf{A},\mathsf{R})$ 
seen as an internal posetal groupoid }, i.e. on the poset internal category whose objects are the elements of $\mathsf{A}$
and for which, assuming to work in the internal language of the doctrine
$\thprop$, an arrow from $a$ to $a'$ in $\mathsf{A}$ exists if and only if
$\mathsf{R}(a,a')$ holds. In \cite{m09} such internal diagrams are shown
to be equivalently described
in terms of {\it dependent extensional collections}.

Since our base category  $\mathcal{C}_{r}$ is of syntactic nature
as that in \cite{m09}, instead of working with internal diagrams
we  also prefer to adopt the notion of  {\it dependent extensional collection} which we mimick in our setting
under the name of  {\it family of collection over $(\mathsf{A},\mathsf{R})$}.
Then, we specialize such a notion to that of {\it family of sets depending over  $(\mathsf{A},\mathsf{R})$}, always following the corresponding one in \cite{m09} .

\subsubsection{The notion of family of collections in $\peff$}

The intuitive idea behind the definition of a family of collections 
over an object $(\mathsf{A},\mathsf{R})$ in $\peff$ suggested by \cite{m09}
 is the following: 
\begin{itemize}
\item first take  a family of realized collections $\mathsf{B}(a)$ over $\mathsf{A}$  in $\mathcal{C}_{r}$;
\item then,  take an equivalence relation $\mathsf{S}_{a}(b,b')$ on each $B(a)$ (depending on $\mathsf{A}$);
\item furthermore, take  a family of isomorphisms $\sigma_{a,a'}(b)$ for any $a,a'\,\varepsilon\, \mathsf{A}$ such that $\mathsf{R}(a,a')$ holds, i.e. there exists a realizer for $\mathsf{R}(a,a')$, without making their definition depending
on the  specific realizer of $\mathsf{R}(a,a')$;
\item finally make the various isomorphisms  represent
an action of the relation $\mathsf{R}$ on the family $\mathsf{B}(a)$ which
also preserves  $\mathsf{S}$ by requiring that the isomorphisms satisfy the following conditions:
\begin{enumerate}
\item if $S_{a}(b,b')$ and $\mathsf{R}(a,a')$, then $S_{a'}(\sigma_{a,a'}(b),\sigma_{a,a'}(b'))$
\item if $\mathsf{R}(a,a)$ and $b\,\varepsilon\, \mathsf{B}(a)$, then $S_{a}(b,\sigma_{a,a}(b)))$
\item if $\mathsf{R}(a,a')$, $\mathsf{R}(a',a'')$ and $b\,\varepsilon\, \mathsf{B}(a)$, then $S_{a}(\sigma_{a,a''}(b),\sigma_{a',a''}(\sigma_{a,a'}(b))))$
\end{enumerate}
All these conditions imply that the family $\sigma_{-,-}$ is well-defined (it preserves relations $S_{a}$'s) and \emph{functorial} up to the relation $\mathsf{S}$. 
\end{itemize}

This informal idea is made precise by the following list of definitions.
\begin{definition}
Let $\mathsf{A}$ be an object of $\mathcal{C}_{r}$ and let $\mathsf{B}$ be an object of $\mathbf{Col}^{r}(\mathsf{A})$. A {\em  $\overline{\mathbf{Prop}^r}$-equivalence relation depending on $\mathsf{B}$}, or simply a {\em dependent
equivalence relation},  is an object $\mathsf{S}$ in $\overline{\mathbf{Prop}^{r}}(\Sigma(\mathsf{A},\mathsf{B}\times\mathsf{B}))$ such that 
\begin{enumerate}
\item $\top_{\Sigma(\mathsf{A},\mathsf{B})}\sqsubseteq_{\Sigma(\mathsf{A},\mathsf{B})}\overline{\mathbf{Prop}^{r}}_{\mathbf{I}_{\mathsf{A}}(\Delta_{\mathsf{B}})}(\mathsf{S})$
\item  $\mathsf{S}\sqsubseteq_{\Sigma(\mathsf{A},\mathsf{B}\times\mathsf{B})}\overline{\mathbf{Prop}^{r}}_{\mathbf{I}_{\mathsf{A}}(\mathsf{tw}_{\mathsf{B}})}(\mathsf{S})$

\item $\overline{\mathbf{Prop}^{r}}_{\mathbf{I}_{\mathsf{A}}(\langle \mathsf{p}_{1},\mathsf{p}_{2}\rangle)}(\mathsf{S})\sqcap \overline{\mathbf{Prop}^{r}}_{\mathbf{I}_{\mathsf{A}}(\langle \mathsf{p}_{2},\mathsf{p}_{3}\rangle)}(\mathsf{S})\sqsubseteq_{\Sigma(\mathsf{A},\mathsf{B}\times\mathsf{B}\times\mathsf{B})}\overline{\mathbf{Prop}^{r}}_{\mathbf{I}_{\mathsf{A}}(\langle \mathsf{p}_{1},\mathsf{p}_{3}\rangle)}(\mathsf{S})$
\end{enumerate}

where 
\begin{enumerate}
\item $\mathbf{I}_{\mathsf{A}}$ is the functor defined in the proof of theorem \ref{pdue}, 
\item $\Delta_{\mathsf{B}}:\mathsf{B}\rightarrow \mathsf{B}\times \mathsf{B}$ in $\mathbf{Col}^{r}\mathsf{A})$,
\item $\mathsf{tw}_{\mathsf{B}}:=\langle \mathsf{p}_{2},\mathsf{p}_{1}\rangle:\mathsf{B}\times \mathsf{B}\rightarrow \mathsf{B}\times \mathsf{B}$ in $\mathbf{Col}^{r}\mathsf{A})$,
\item $\langle \mathsf{p}_{i},\mathsf{p}_{j}\rangle:\mathsf{B}\times \mathsf{B}\times \mathsf{B}\rightarrow \mathsf{B}\times \mathsf{B}$ for $i<j$, $i,j\in \{1,2,3\}$ in $\mathbf{Col}^{r}\mathsf{A})$.
\end{enumerate}
\end{definition}

\begin{definition}\label{labb}
Let $(\mathsf{A},\mathsf{R})$ be an object of $\peff$, $\mathsf{B}$ an object of $\mathbf{Col}^{r}\mathsf{A})$ and $\mathsf{S}$ an $\overline{\mathbf{Prop}^r}$-equivalence relation depending  on $\mathsf{B}$. Suppose that $\mathsf{R}'$ is a
representative of $\mathsf{R}$. Then {\em an action of  $\mathsf{R}'$ over $\mathsf{B}$ with respect to $\mathsf{S}$} is an arrow 
$$\sigma:\mathbf{Col}^{r}_{\mathsf{p}_{1}\circ \mathsf{p}_{1}^{\Sigma}}(\mathsf{B})\rightarrow \mathbf{Col}^{r}_{\mathsf{p}_{2}\circ \mathsf{p}_{1}^{\Sigma}}(\mathsf{B})$$ in $\mathbf{Col}^{r}\Sigma(\mathsf{A}\times \mathsf{A},\mathsf{R}'))$ such that

\begin{enumerate}
\item $\sigma$ preserves $\mathsf{S}$, i.\,e.\,
$$\overline{\mathbf{Prop}}^{r}_{\Sigma(\mathsf{p}_{1}\circ \mathsf{p}_{1}^{\Sigma},\mathsf{B}\times \mathsf{B})}(\mathsf{S})\sqsubseteq \overline{\mathbf{Prop}}^{r}_{\Sigma(\mathsf{p}_{2}\circ \mathsf{p}_{1}^{\Sigma},\mathsf{B}\times \mathsf{B})\circ \mathbf{I}(\sigma\times \sigma)}(\mathsf{S}) $$
where for clarity recall that the considered arrows make  the following diagram
commutative
{\small \def\objectstyle{\scriptstyle}
\def\labelstyle{\scriptstyle}
{
$$\xymatrix@-1pc{
\Sigma(\mathsf{A},\mathsf{B}\times \mathsf{B})\\
\Sigma(\Sigma(\mathsf{A}\times \mathsf{A},\mathsf{R}'),\mathbf{Col}_{\mathsf{p}_{1}\circ \mathsf{p}_{1}^{\Sigma}}(\mathsf{B}\times \mathsf{B}))\ar[u]^-{\Sigma(\mathsf{p}_{1}\circ \mathsf{p}_{1}^{\Sigma},\mathsf{B}\times \mathsf{B})}\ar[r]_-{\mathbf{I}(\sigma\times \sigma)}	&\Sigma(\Sigma(\mathsf{A}\times \mathsf{A},\mathsf{R}'),\mathbf{Col}_{\mathsf{p}_{2}\circ \mathsf{p}_{1}^{\Sigma}}(\mathsf{B}\times \mathsf{B}))\ar[ul]_-{\qquad\Sigma(\mathsf{p}_{2}\circ \mathsf{p}_{1}^{\Sigma},\mathsf{B}\times \mathsf{B})}\\
}$$}}

\item $\sigma $ does not depend on realizers in $\mathsf{R}'$, i.\,e.\ if $\overline{\sigma}$ is the unique arrow in $\mathcal{C}_{r}$ making the following diagram commutatitve (which exists thanks to lemma \ref{univprod}), 
{\small \def\objectstyle{\scriptstyle}
\def\labelstyle{\scriptstyle}
{
$$\xymatrix@-1pc{
		&\Sigma(\Sigma(\mathsf{A}\times \mathsf{A},\mathsf{R'}\times \mathsf{R}'),\mathbf{Col}_{\mathsf{p}_{2}\circ \mathsf{p}_{1}^{\Sigma}}(\mathsf{B}))\\
\Sigma(\Sigma(\mathsf{A}\times \mathsf{A},\mathsf{R'}\times \mathsf{R}'),\mathbf{Col}_{\mathsf{p}_{1}\circ \mathsf{p}_{1}^{\Sigma}}(\mathsf{B}))\ar[r]^{\overline{\sigma}}\ar[ru]^-{\mathbf{I}(\mathbf{Col}_{\mathbf{I}(\mathsf{p}_{1})}(\sigma))}	\ar[rd]_{\mathbf{I}(\mathbf{Col}_{\mathbf{I}(\mathsf{p}_{2})}(\sigma))}	&\Sigma(\Sigma(\mathsf{A}\times \mathsf{A},\mathsf{R'}\times \mathsf{R}'),\mathbf{Col}_{\mathsf{p}_{2}\circ \mathsf{p}_{1}^{\Sigma}}(\mathsf{B}\times \mathsf{B}))\ar[u]^{\mathbf{I}(\mathsf{p}_{1})}	\ar[d]_{\mathbf{I}(\mathsf{p}_{2})}\\
&\Sigma(\Sigma(\mathsf{A}\times \mathsf{A},\mathsf{R'}\times \mathsf{R}'),\mathbf{Col}_{\mathsf{p}_{2}\circ \mathsf{p}_{1}^{\Sigma}}(\mathsf{B}))\\
}$$}}
then 
$\top\sqsubseteq \mathbf{Col}_{\,\Sigma(\mathsf{p}_{2}\circ \mathsf{p}_{1}^{\Sigma},\mathsf{B}\times \mathsf{B})\circ\, \overline{\sigma}}(\mathsf{S}).$

\item $\sigma$ preserves the identities:
let $\rho$ be the unique arrow making the following diagram commutative in $\mathcal{C}_{r}$ (which exists thanks to lemma \ref{univprod}), 
{\small \def\objectstyle{\scriptstyle}
\def\labelstyle{\scriptstyle}
{
$$\xymatrix@-1pc{
						&\Sigma(\mathsf{A},\mathsf{B})\\
\Sigma(\Sigma(\mathsf{A},\mathbf{Col}_{\Delta_{\mathsf{A}}}(\mathsf{R}')),\mathbf{Col}_{\mathsf{p}_{1}^{\Sigma}}(\mathsf{B}))\ar[r]^-{\rho}\ar[ru]^-{\Sigma(\mathsf{p}_{1}\circ \mathsf{p}_{1}^{\Sigma},\mathsf{B})}\ar[d]^{\Sigma(\Sigma(\Delta_{\mathsf{A}},\mathsf{R}'),\mathbf{Col}_{\mathsf{p}_{1}\circ \mathsf{p}_{1}^{\Sigma}}(\mathsf{B}))} &\Sigma(\mathsf{A},\mathsf{B}\times \mathsf{B})\ar[u]_{\mathbf{I}(\mathsf{p}_{1})}\ar[dd]^{\mathbf{I}(\mathsf{p}_{2})}\\
\Sigma(\Sigma(\mathsf{A}\times \mathsf{A},\mathsf{R}'),\mathbf{Col}_{\mathsf{p}_{1}\circ\mathsf{p}_{1}^{\Sigma}}(\mathsf{B}))\ar[d]^-{\mathbf{I}(\sigma)}	&\\
\Sigma(\Sigma(\mathsf{A}\times \mathsf{A},\mathsf{R}'),\mathbf{Col}_{\mathsf{p}_{2}\circ\mathsf{p}_{1}^{\Sigma}}(\mathsf{B}))\ar[r]^-{\Sigma(\mathsf{p}_{2}\circ \mathsf{p}_{1}^{\Sigma},\mathsf{B})}	&\Sigma(\mathsf{A},\mathsf{B})\\
}$$}}
then 
$\top\sqsubseteq\overline{\mathbf{Prop}^{r}}_{\rho}(\mathsf{S})$.

\item if $\tau$ (which exists thanks to lemma \ref{univprod}) and $\tau'$ (which exists thanks to lemma \ref{univsig}) are the unique arrows in $\mathcal{C}_{r}$ making the following diagrams commute
{\scriptsize
\def\objectstyle{\scriptstyle}
\def\labelstyle{\scriptstyle}
{
$$\xymatrix@-1pc{
\Sigma(\Sigma(\mathsf{A}\times \mathsf{A},\mathsf{R}'),\mathbf{Col}_{\mathsf{p}_{1}\circ\mathsf{p}_{1}^{\Sigma}}(\mathsf{B}))\ar[r]^{\mathbf{I}(\sigma)}		&\Sigma(\Sigma(\mathsf{A}\times \mathsf{A},\mathsf{R}'),\mathbf{Col}_{\mathsf{p}_{2}\circ\mathsf{p}_{1}^{\Sigma}}(\mathsf{B}))\ar[d]^-{\Sigma(\mathsf{p}_{2}\circ \mathsf{p}_{1}^{\Sigma},\mathsf{B})}	\\	
\Sigma(\Sigma(\mathsf{A}\times \mathsf{A}\times \mathsf{A},\mathbf{Col}_{\langle \mathsf{p}_{1},\mathsf{p}_{2}\rangle}(\mathsf{R}'),\mathbf{Col}_{\mathsf{p}_{1}\circ\mathsf{p}_{1}^{\Sigma}}(\mathsf{B}))\ar[u]_-{\Sigma(\Sigma(\mathsf{p}_{1},\mathsf{R}'),\mathbf{Col}_{\mathsf{p}_{1}\circ \mathsf{p}_{1}^{\Sigma}}(\mathsf{B}))}				&\Sigma(\mathsf{A},\mathsf{B})\\
\Sigma(\Sigma(\mathsf{A}\times \mathsf{A}\times \mathsf{A},\mathbf{Col}_{\langle \mathsf{p}_{1},\mathsf{p}_{2}\rangle}(\mathsf{R}')\times\mathbf{Col}_{\langle \mathsf{p}_{2},\mathsf{p}_{3}\rangle}(\mathsf{R}')\times\mathbf{Col}_{\langle \mathsf{p}_{1},\mathsf{p}_{3}\rangle}(\mathsf{R}') ),\mathbf{Col}_{\mathsf{p}_{1}\circ\mathsf{p}_{1}^{\Sigma}}(\mathsf{B}))\ar[d]^-{\mathsf{p}_{1}^{\Sigma}}\ar[r]^-{\tau'}\ar[u]_-{\Sigma(\mathbf{I}(\mathsf{p}_{1}),\mathbf{Col}_{\mathsf{p}_{1}\circ \mathsf{p}_{1}^{\Sigma}}(\mathsf{B}))} &\Sigma(\Sigma(\mathsf{A}\times \mathsf{A},\mathsf{R}' ),\mathbf{Col}_{\mathsf{p}_{1}\circ\mathsf{p}_{1}^{\Sigma}}(\mathsf{B}))\ar[dd]^{\mathsf{p}_{1}^{\Sigma}}\ar[u]_-{\Sigma(\mathsf{p}_{1}\circ \mathsf{p}_{1}^{\Sigma},\mathsf{B})}\\
\Sigma(\mathsf{A}\times \mathsf{A}\times \mathsf{A},\mathbf{Col}_{\langle \mathsf{p}_{1},\mathsf{p}_{2}\rangle}(\mathsf{R}')\times\mathbf{Col}_{\langle \mathsf{p}_{2},\mathsf{p}_{3}\rangle}(\mathsf{R}')\times\mathbf{Col}_{\langle \mathsf{p}_{1},\mathsf{p}_{3}\rangle}(\mathsf{R}') )\ar[d]^{\mathbf{I}(\mathsf{p}_{2})}&\\
\Sigma(\mathsf{A}\times \mathsf{A}\times \mathsf{A},\mathbf{Col}_{\langle \mathsf{p}_{2},\mathsf{p}_{3}\rangle}(\mathsf{R}') )\ar[r]^{\Sigma(\langle \mathsf{p}_{2},\mathsf{p}_{3}\rangle,\mathsf{R}')}	&\Sigma(\mathsf{A}\times \mathsf{A},\mathsf{R}') )\\
}$$}}

{\scriptsize \def\objectstyle{\scriptstyle}
\def\labelstyle{\scriptstyle}
{
$$\xymatrix@-1pc{
\Sigma(\Sigma(\mathsf{A}\times \mathsf{A},\mathsf{R}' ),\mathbf{Col}_{\mathsf{p}_{1}\circ\mathsf{p}_{1}^{\Sigma}}(\mathsf{B}))\ar[r]^-{\mathbf{I}(\sigma)}	&\Sigma(\Sigma(\mathsf{A}\times \mathsf{A},\mathsf{R}' ),\mathbf{Col}_{\mathsf{p}_{2}\circ\mathsf{p}_{1}^{\Sigma}}(\mathsf{B}))\ar[d]^-{\Sigma(\mathsf{p}_{2}\circ \mathsf{p}_{1}^{\Sigma},\mathsf{B})}\\
\Sigma(\Sigma(\mathsf{A}\times \mathsf{A}\times \mathsf{A},\mathbf{Col}_{\langle \mathsf{p}_{1},\mathsf{p}_{3}\rangle}(\mathsf{R}') ),\mathbf{Col}_{\mathsf{p}_{1}\circ\mathsf{p}_{1}^{\Sigma}}(\mathsf{B}))\ar[u]^{\Sigma(\Sigma(\langle \mathsf{p}_{1},\mathsf{p}_{3}\rangle,\mathsf{R'}),\mathbf{Col}_{\mathsf{p}_{1}\circ \mathsf{p}_{1}^{\Sigma}}(\mathsf{B}))}	&\Sigma(\mathsf{A},\mathsf{B})\\
\Sigma(\Sigma(\mathsf{A}\times \mathsf{A}\times \mathsf{A},\mathbf{Col}_{\langle \mathsf{p}_{1},\mathsf{p}_{2}\rangle}(\mathsf{R}')\times\mathbf{Col}_{\langle \mathsf{p}_{2},\mathsf{p}_{3}\rangle}(\mathsf{R}')\times\mathbf{Col}_{\langle \mathsf{p}_{1},\mathsf{p}_{3}\rangle}(\mathsf{R}') ),\mathbf{Col}_{\mathsf{p}_{1}\circ\mathsf{p}_{1}^{\Sigma}}(\mathsf{B}))\ar[r]^-{\tau}\ar[u]^-{\Sigma(\mathbf{I}(\mathsf{p}_{3}),\mathbf{Col}_{\mathsf{p}_{1}\circ \mathsf{p}_{1}^{\Sigma}}(\mathsf{B}))}\ar[dd]_-{\tau'}&\Sigma(\mathsf{A},\mathsf{B}\times \mathsf{B})\ar[u]_-{\mathbf{I}(\mathsf{p}_{1})}\ar[d]^{\mathbf{I}(\mathsf{p}_{2})}\\
					&\Sigma(\mathsf{A},\mathsf{B})\\
\Sigma(\Sigma(\mathsf{A}\times \mathsf{A},\mathsf{R}' ),\mathbf{Col}_{\mathsf{p}_{1}\circ\mathsf{p}_{1}^{\Sigma}}(\mathsf{B}))\ar[r]^{\mathbf{I}(\sigma)}					&\Sigma(\Sigma(\mathsf{A}\times \mathsf{A},\mathsf{R}' ),\mathbf{Col}_{\mathsf{p}_{2}\circ\mathsf{p}_{1}^{\Sigma}}(\mathsf{B}))\ar[u]_-{\Sigma(\mathsf{p}_{2}\circ \mathsf{p}_{1}^{\Sigma})}\\
}$$
}}

then
$\top\sqsubseteq \overline{\mathbf{Prop}^{r}}_{\tau}(\mathsf{S})$.
\end{enumerate}
\end{definition}
\begin{definition}Let $(\mathsf{A},\mathsf{R})$ be an object of $\peff$ and let
$\mathsf{R}'$ a representative of the equivalence class  $\mathsf{R}$.
We define {\em a family of collections over $(\mathsf{A},\mathsf{R})$}
as 
\begin{enumerate}
\item $\mathsf{B}$ is an object of $\mathbf{Col}^{r}(\mathsf{A})$;
\item $\mathsf{S}$ is a $\overline{\mathbf{Prop}^r}$-equivalence relation depending  on $\mathsf{B}$:
\item $\sigma$ is an action of  $\mathsf{R}'$ over $\mathsf{B}$ with respect to $\mathsf{S}$.
\end{enumerate}
\end{definition}
Now we are going to define
an arrow between families of collections over $(\mathsf{A},\mathsf{R})$ from $(\mathsf{B},\mathsf{S},\sigma)$ to $(\mathsf{C},\mathsf{T},\psi)$.
Informally, this arrow should be an equivalence class of arrows $\mathsf{f}$ from $\mathsf{B}$ to $\mathsf{C}$ in $\mathbf{Col}^{r}(\mathsf{A})$ respecting the following clauses:
\begin{enumerate}
\item if $\mathsf{S}_{a}(b,b')$, then $\mathsf{T}_{a}(\mathsf{f}_{a}(b),\mathsf{f}_{a}(b'))$
\item if $r\Vdash \mathsf{R}(a,a')$ and $b\,\varepsilon\, \mathsf{B}(a)$, then $\mathsf{T}_{a'}(\mathsf{f}_{a'}(\sigma_{a,a'}(b)),\psi_{a,a'}(\mathsf{f}_{a}(b)))$
\end{enumerate}
with respect to the equivalence for which $\mathsf{f}\simeq \mathsf{g}$ if and only if from $\mathsf{S}_{a}(b,b')$ it follows that $\mathsf{T}_{a}(\mathsf{f}_{a}(b),\mathsf{g}_{a}(b'))$.

This is precisely expressed by the following definition.

\begin{definition}
Let $(\mathsf{A},\mathsf{R})$ be an object of $\peff$ and let $\mathsf{R}=[\mathsf{R}']$.
A {\em morphism between families of collections over $(\mathsf{A},\mathsf{R})$} with respect to $\mathsf{R}'$ from $(\mathsf{B},\mathsf{S},\sigma)$ to $(\mathsf{C},\mathsf{T},\psi)$ is an equivalence class $[\mathsf{f}]_{\cong}$ of arrows $\mathsf{f}:\mathsf{B}\rightarrow \mathsf{C}$ in $\mathbf{Col}^{r}(\mathsf{A})$ such that 
\begin{enumerate}
\item $\mathsf{S}\sqsubseteq \overline{\mathbf{Prop}^{r}}_{\mathbf{I}(\mathsf{f}\times \mathsf{f})}(\mathsf{T})$
\item $\top\sqsubseteq \overline{\mathbf{Prop}^{r}}_{\beta}(\mathsf{T})$
where $\beta$ is the unique arrow in $\mathcal{C}_{r}$ (which exists thanks to lemma \ref{univprod}) making the following diagram commute
{\small \def\objectstyle{\scriptstyle}
\def\labelstyle{\scriptstyle}
{
$$\xymatrix@-1pc{
\Sigma(\mathsf{A},\mathsf{B})\ar[r]^-{\mathbf{I}(\mathsf{f})}	&\Sigma(\mathsf{A},\mathsf{C})\\
\Sigma(\Sigma(\mathsf{A}\times \mathsf{A},\mathsf{R'}),\mathbf{Col}_{\mathsf{p}_{2}\circ \mathsf{p}_{1}^{\Sigma}}(\mathsf{B}))\ar[u]_-{\Sigma(\mathsf{p}_{2}\circ \mathsf{p}_{1}^{\Sigma},\mathsf{B})}	&\\
\Sigma(\Sigma(\mathsf{A}\times \mathsf{A},\mathsf{R'}),\mathbf{Col}_{\mathsf{p}_{1}\circ \mathsf{p}_{1}^{\Sigma}}(\mathsf{B}))\ar[r]^-{\beta}\ar[u]_-{\mathbf{I}(\sigma)}\ar[d]^{\mathbf{I}(\mathbf{Col}_{\mathsf{p}_{1}\circ \mathsf{p}_{1}^{\Sigma}}(\mathsf{f}))}	&\Sigma(\mathsf{A},\mathsf{C}\times \mathsf{C})\ar[uu]_-{\mathbf{I}(\mathsf{p}_{1})}\ar[dd]^-{\mathbf{I}(\mathsf{p}_{2})}\\
\Sigma(\Sigma(\mathsf{A}\times \mathsf{A},\mathsf{R'}),\mathbf{Col}_{\mathsf{p}_{1}\circ \mathsf{p}_{1}^{\Sigma}}(\mathsf{C}))\ar[d]^-{\mathbf{I}(\psi)}		&\\
\Sigma(\Sigma(\mathsf{A}\times \mathsf{A},\mathsf{R'}),\mathbf{Col}_{\mathsf{p}_{2}\circ \mathsf{p}_{1}^{\Sigma}}(\mathsf{C}))\ar[r]^-{\Sigma(\mathsf{p}_{2}\circ \mathsf{p}_{1}^{\Sigma},\mathsf{C})}		&\Sigma(\mathsf{A},\mathsf{C})\\
}$$}}

\end{enumerate}

with respect to the equivalence relation defined by $\mathsf{f}\simeq \mathsf{g}$ if and only if $\mathsf{S}\sqsubseteq \overline{\mathbf{Prop}^{r}}_{\mathbf{I}(\mathsf{f}\times \mathsf{g})}(\mathsf{T})$.
\end{definition}
\begin{lemma} Let $(\mathsf{A},\mathsf{R})$ be an object of $\peff$ and let $\mathsf{R}=[\mathsf{R}']$. Families of collections over $(\mathsf{A},\mathsf{R})$ (with respect to $\mathsf{R}'$) form a category together with morphisms between them and compositions and identities inherited from $\mathbf{Col}^{r}(\mathsf{A})$.\end{lemma}

\begin{definition}
Let $(\mathsf{A},\mathsf{R})$ be an object of $\peff$ and let $\mathsf{R}=[\mathsf{R}']$. We call $\mcol(\mathsf{A},\mathsf{R}') $ the category which is described in the previous lemma. 

Moreover,
we call $\mset(\mathsf{A},\mathsf{R}')$ the full subcategory of $\mcol(\mathsf{A},\mathsf{R}') $ whose objects are families $(\mathsf{B},\mathsf{S},\sigma)$ of collections over $(\mathsf{A},\mathsf{R})$ with 
\begin{enumerate}
\item $\mathsf{B}$ in $\mathbf{Set}(\mathsf{A})$;
\item $\mathsf{S}$ in $\overline{\mathbf{Prop}^{r}_{s}}(\Sigma(\mathsf{A},\mathsf{B}\times \mathsf{B}))$. 
\end{enumerate}
\end{definition}
We can first notice the following property of $\mcol(\mathsf{A},\mathsf{R}') $.
\begin{remark}
Let $(\mathsf{A},\mathsf{R})$ be an object of $\peff$ and let $\mathsf{R}'$ and $\mathsf{R}''$ be representative of $\mathsf{R}$. Then $\mcol(\mathsf{A},\mathsf{R}') $ and $\mcol(\mathsf{A},\mathsf{R}'') $ are equivalent categories. This a consequence of the equivalence between $\mathsf{R}'$ and $\mathsf{R}''$ and point 2 in definition \ref{labb}.
\end{remark}

\begin{theorem}Let $(\mathsf{A},\mathsf{R})$ be an object of $\peff$ and let $\mathsf{R}=[\mathsf{R}']$. $\mcol(\mathsf{A},\mathsf{R}') $ and $\mset(\mathsf{A},\mathsf{R}')$ are cartesian closed list-arithmetic pretoposes.
\end{theorem}
\begin{proof}
\begin{enumerate}
\item A terminal object is given by $(\mathsf{1},\overline{\mathbf{Prop}^{r}}_{\mathsf{p}_{1}^{\Sigma}}(\top),\mathsf{id})$. This object is a well defined terminal object also in $\mset(\mathsf{A},\mathsf{R}')$.
\item A binary product for $(\mathsf{B},\mathsf{S},\sigma)$ and $(\mathsf{C},\mathsf{T},\psi)$ can be given by 
$$(\mathsf{B}\times \mathsf{C},\overline{\mathbf{Prop}^{r}}_{\mathbf{I}(\mathsf{p}_{1}\times \mathsf{p}_{1})}(\mathsf{S})\sqcap \overline{\mathbf{Prop}^{r}}_{\mathbf{I}(\mathsf{p}_{2}\times \mathsf{p}_{2})}(\mathsf{T}),\sigma\times \psi)$$
with projections inherited by those of $\mathsf{B}\times \mathsf{C}$. This construction restricts to $\mset(\mathsf{A},\mathsf{R}')$.

\item An equalizer for $[\mathsf{f}]_{\simeq},[\mathsf{g}]_{\simeq}: (\mathsf{B},\mathsf{S},\sigma)\rightarrow(\mathsf{C},\mathsf{T},\psi)$ is given by 
$$(\mathsf{E},\;\overline{\mathbf{Prop}^{r}}_{\mathbf{I}(\mathsf{e}^{\Sigma}_{1}\times \mathsf{e}^{\Sigma}_{1})}(\mathsf{S}),\;\sigma')$$
with 
$$\mathsf{E}:=\{x'|\,p_{1}(x')\,\varepsilon\, \mathsf{B}(x)\,\wedge\,p_{2}(x)\Vdash \mathsf{T}'(p(x,p(\{\mathbf{n}_{\mathsf{f}}\}(p_{1}(x')), \{\mathbf{n}_{\mathsf{g}}(p_{1}(x')))))\}$$
and with $\mathsf{T}=[\mathsf{T}']$ and $\mathsf{f}:=[\mathbf{n}_{\mathsf{f}}]_{\sim}$ and $\mathsf{g}:=[\mathbf{n}_{\mathsf{g}}]_{\sim}$ and $\mathsf{e}^{\Sigma}_{1}:=[\Lambda x.\Lambda x'.\{\mathbf{p}_{1}\}(x')]_{\sim}$ and $\sigma'$ is 
an arrow who coincides with $\sigma$ on the first component and acts using witnesses given by the first condition in the definition of morphism applied to $[\mathsf{f}]_{\simeq}$ an $[\mathsf{g}]_{\simeq}$ on the second component. 
One can easily show that this construction restricts to $\mset(\mathsf{A},\mathsf{R}')$.
\item An initial object is given by $(\mathsf{0},\overline{\mathbf{Prop}^{r}}_{\mathsf{p}_{1}^{\Sigma}}(\top),\mathsf{id})$.  This object is a well defined initial object also in $\mset(\mathsf{A},\mathsf{R}')$.
\item A binary coproduct for $(\mathsf{B},\mathsf{S},\sigma)$ and $(\mathsf{C},\mathsf{T},\psi)$ can be given by 
$$(\mathsf{B}+ \mathsf{C},\exists_{\mathbf{I}(\mathsf{j}_{\mathsf{B}}\times \mathsf{j}_{\mathsf{B}})) }(\mathsf{S})\sqcup \exists_{\mathbf{I}(\mathsf{j}_{\mathsf{C}}\times \mathsf{j}_{\mathsf{C}}))}(\mathsf{T}),\sigma+ \psi)$$
with injections inherited by those of $\mathsf{B}\times \mathsf{C}$.
One can easily notice that $\mathsf{j}_{\mathsf{B}}\times \mathsf{j}_{\mathsf{B}}$ and $\mathsf{j}_{\mathsf{C}}\times \mathsf{j}_{\mathsf{C}}$ are isomorphic to representable arrows (see \ref{repress}) when acting on sets, which implies that this construction restricts to $\mset(\mathsf{A},\mathsf{R}')$.
\item Stability and disjointness of coproducts follow from direct verification.
\item A list object for $(\mathsf{B},\mathsf{S},\sigma)$ is given by 
$$(\mathsf{List}(\mathsf{B}), \overline{\mathsf{Prop}^{r}}_{\mathsf{lh}\times \mathsf{lh}}(\exists_{\Delta}(\top))\sqcap \forall_{\mathbf{I}(\mathsf{p}_{1})}(\mathsf{Prop}_{\mathbf{I}(\mathsf{cons}\times \mathsf{cons})}(\exists_{\mathbf{I}(\mathsf{j}_{\mathsf{A}}\times \mathsf{j}_{\mathsf{A}})}(\mathsf{S})\sqcup \exists_{\mathbf{I}(\mathsf{j}_{\mathsf{1}}\times \mathsf{j}_{\mathsf{1}})}(\top) )), \mathsf{List}(\sigma))$$
where $\mathsf{lh}:\mathsf{List}(A)\rightarrow \mathsf{N}$ is the length arrow, $\mathsf{comp}:\mathsf{List}(A)\times \mathsf{N}\rightarrow \mathsf{A}+1$ is the component arrow in $\mathcal{C}_{r}$ and 
{\small $$\mathsf{List}(\sigma):\mathsf{List}(\mathbf{Col}^{r}_{\mathsf{p}_{1}\circ \mathsf{p}_{1}^{\Sigma}}(\mathsf{B}))=\mathbf{Col}^{r}_{\mathsf{p}_{1}\circ \mathsf{p}_{1}^{\Sigma}}(\mathsf{List}(\mathsf{B}))\rightarrow \mathbf{Col}^{r}_{\mathsf{p}_{2}\circ \mathsf{p}_{1}^{\Sigma}}(\mathsf{List}(\mathsf{B}))=\mathsf{List}(\mathbf{Col}^{r}_{\mathsf{p}_{2}\circ \mathsf{p}_{1}^{\Sigma}}(\mathsf{B}))$$ }
is the unique arrow making the following diagram in $\mathbf{Col}^{r}\Sigma(\mathsf{A}\times \mathsf{A},\mathsf{R}'))$ commute.
{\small \def\objectstyle{\scriptstyle}
\def\labelstyle{\scriptstyle}
{
$$\xymatrix@-1pc{
\mathsf{1}	\ar[r]^-{\epsilon}\ar[rd]_-{\epsilon}	&\mathsf{List}(\mathbf{Col}^{r}_{\mathsf{p}_{1}\circ \mathsf{p}_{1}^{\Sigma}}(\mathsf{B}))\ar[d]^-{\mathsf{List}(\sigma)}&\mathsf{List}(\mathbf{Col}^{r}_{\mathsf{p}_{1}\circ \mathsf{p}_{1}^{\Sigma}}(\mathsf{B}))\times \mathbf{Col}^{r}_{\mathsf{p}_{1}\circ \mathsf{p}_{1}^{\Sigma}}(\mathsf{B})\ar[l]^-{\mathsf{cons}}\ar[d]^-{\mathsf{List}(\sigma)\times \sigma}	\\
			&\mathsf{List}(\mathbf{Col}^{r}_{\mathsf{p}_{2}\circ \mathsf{p}_{1}^{\Sigma}}(\mathsf{B}))	&\mathsf{List}(\mathbf{Col}^{r}_{\mathsf{p}_{2}\circ \mathsf{p}_{1}^{\Sigma}}(\mathsf{B}))\times \mathbf{Col}^{r}_{\mathsf{p}_{2}\circ \mathsf{p}_{1}^{\Sigma}}(\mathsf{B})\ar[l]^-{\mathsf{cons}}\\
}$$}}
This construction restricts to $\mset(\mathsf{A},\mathsf{R}')$ as $\mathsf{j}_{\mathsf{A}}\times \mathsf{j}_{\mathsf{A}}$ and $\mathsf{j}_{\mathsf{1}}\times \mathsf{j}_{\mathsf{1}}$ are isomorphic to representable arrows.
\item Every arrow admits a pullback stable mono-regular epi factorization: suppose 
$[\mathsf{f}]_{\simeq}: (\mathsf{B},\mathsf{S},\sigma)\rightarrow(\mathsf{C},\mathsf{T},\psi)$ is a morphism in  $\mcol(\mathsf{A},\mathsf{R}')$, then the following object $\mathsf{Im}([\mathsf{f}]_{\simeq})$ defines an image for $[\mathsf{f}]_{\simeq}$
$$(\mathsf{I},\;,\overline{\mathbf{Prop}^{r}}_{\mathbf{I}(\mathsf{e}^{\Sigma}_{1}\times \mathsf{e}^{\Sigma}_{1})}(\mathsf{S}),\sigma'')$$
with $\mathsf{I}$ defined as 
$$\{x'|\,p^{3}_{1}(x')\,\varepsilon\, \mathsf{B}(x)\,\wedge\,p^{3}_{2}(x')\,\varepsilon\,\mathsf{C}(x)\,\wedge\, p^{3}_{3}(x')
\Vdash \mathsf{T}'(p(x,p(\{\mathbf{n}\}(x,p_{1}(x')), p_{2}(x')) ))\}$$
and with $\mathsf{T}=[\mathsf{T}']$, $\mathbf{n}$ a numeral such that $\mathsf{f}:=[\mathbf{n}]_{\sim}$ and $\mathsf{e}^{\Sigma}_{1}:=[\Lambda x.\Lambda x'.\{\mathbf{p}^{3}_{1}\}(x')]_{\sim}\footnote{The functions $p^{3}_{1}$,$p^{3}_{2}$, $p^{3}_{3}$ are projections relative to a primitive recursive encoding of ternary products and $\mathbf{p}^{3}_{1}$ is a numeral representing $p_{1}^{3}$.}$ and $\sigma''$ is an arrow who coincides with $\sigma$ on the first component, with $\psi$ on the second component and acts using witnesses given by the first condition in the definition of morphism applied to $[\mathsf{f}]_{\simeq}$ on the second component. The arrows in the factorization are the obvious ones.
{\small \def\objectstyle{\scriptstyle}
\def\labelstyle{\scriptstyle}
{
$$\xymatrix@-1pc{
(\mathsf{B},\mathsf{S},\sigma)\ar[rr]^{[\mathsf{f}]_{\simeq}}\ar@{>>}[rd]	&									&(\mathsf{C},\mathsf{T},\psi)\\
													&\mathsf{Im}([\mathsf{f}]_{\simeq})\ar@{^{(}->}[ru]		&\\
}$$}}
One can easily show that this construction restricts to $\mset(\mathsf{A},\mathsf{R}')$.

\item Effective coequalizers of equivalence relations
exist  by definition of \\$\mcol(\mathsf{A},\mathsf{R}')$ and the same holds for $\mset(\mathsf{A},\mathsf{R}')$ thanks to the properties for elementary quotient completions shown in \cite{qu12}.
\end{enumerate}

\end{proof}

Now we are ready to show that any slice category  $\peff/(\mathsf{A},\mathsf{R})$
is essentially equivalent (up to a choice of representatives)
to the category of families of collections over $(\mathsf{A},\mathsf{R})$:

\begin{theorem} Let $(\mathsf{A},\mathsf{R})$ be an object of $\peff$ and let $\mathsf{R}=[\mathsf{R}']$. If $(\mathsf{B},\mathsf{S},\sigma)$ is an object of $\mcol((\mathsf{A},\mathsf{R}')) $, then $$\mathbf{K}(\mathsf{B},\mathsf{S},\sigma):=[\mathsf{p}_{1}^{\Sigma}]_{\simeq}:(\Sigma(\mathsf{A},\mathsf{B}), \exists_{\mathsf{p}_{1}^{\Sigma}}(\overline{\mathbf{Prop}}_{\eta}(\mathsf{S})))\rightarrow (\mathsf{A},\mathsf{R})$$
where 
$\eta$ is the unique arrow (which exists thanks to lemma \ref{univprod}) making the following diagram commute in $\mathcal{C}_{r}$:

{\small \def\objectstyle{\scriptstyle}
\def\labelstyle{\scriptstyle}
{
$$\xymatrix@-1pc{
\Sigma(\Sigma(\mathsf{A},\mathsf{B})\times \Sigma(\mathsf{A},\mathsf{B}),\mathbf{Col}^{r}_{\mathsf{p}_{1}^{\Sigma}\times \mathsf{p_{1}^{\Sigma}}}(\mathsf{R}'))\ar[d]^{\gamma}\ar[r]^-{\mathsf{p}_{2}\circ \mathsf{p}_{1}^{\Sigma}})\ar[rd]^{\mathsf{\eta}}	&\Sigma(\mathsf{A},\mathsf{B})\\		
\Sigma(\Sigma(\mathsf{A}\times \mathsf{A},\mathsf{R}'),\mathbf{Col}_{\mathsf{p}_{1}\circ\mathsf{p}_{1}^{\Sigma}}(\mathsf{B}))\ar[d]^{\mathbf{I}(\sigma)}		&\Sigma(\mathsf{A},\mathsf{B}\times \mathsf{B})\ar[u]_-{\mathbf{I}(\mathsf{p}_{1})}\ar[d]^-{\mathbf{I}(\mathsf{p}_{2})}\\
\Sigma(\Sigma(\mathsf{A}\times \mathsf{A},\mathsf{R}'),\mathbf{Col}_{\mathsf{p}_{1}\circ\mathsf{p}_{1}^{\Sigma}}(\mathsf{B}))\ar[r]^-{\Sigma(\mathbf{p}_{2}\circ \mathsf{p}_{1}^{\Sigma},\mathsf{B})}&\Sigma(\mathsf{A},\mathsf{B})\\
}$$}}
where $\gamma$ is the unique arrow (which exists thanks to lemma \ref{univsig}) such that the following diagram commute
{\small \def\objectstyle{\scriptstyle}
\def\labelstyle{\scriptstyle}
{
$$\xymatrix@-1pc{
\Sigma(\mathsf{A},\mathsf{B})&\Sigma(\Sigma(\mathsf{A}\times \mathsf{A},\mathsf{R}'),\mathbf{Col}_{\mathsf{p}_{1}\circ\mathsf{p}_{1}^{\Sigma}}(\mathsf{B}))\ar[rd]^-{\mathsf{p}_{1}^{\Sigma}}\ar[l]_-{\Sigma(\mathsf{p}_{1}\circ \mathsf{p}_{1}^{\Sigma},\mathsf{B})}\\
&\Sigma(\Sigma(\mathsf{A},\mathsf{B})\times \Sigma(\mathsf{A},\mathsf{B}),\mathbf{Col}^{r}_{\mathsf{p}_{1}^{\Sigma}\times \mathsf{p_{1}^{\Sigma}}}(\mathsf{R}'))\ar[r]_-{\Sigma(\mathsf{p}_{1}^{\Sigma}\times \mathsf{p}_{1}^{\Sigma},\mathsf{R}')}\ar[u]^-{\gamma}\ar[lu]^-{\mathsf{p}_{1}\circ \mathsf{p}_{1}^{\Sigma}}		&\Sigma(\mathsf{A}\times \mathsf{A},\mathsf{R}')	\\
}$$}}
is a well defined object of $\peff/(\mathsf{A},\mathsf{R})$.

Moreover if $[\mathsf{f}]_{\simeq}$ is an arrow in $\mcol(\mathsf{A},\mathsf{R}')$ from $(\mathsf{B},\mathsf{S},\sigma)$ to $(\mathsf{C},\mathsf{T},\psi)$, then $\mathbf{K}([\mathsf{f}]_{\simeq}):=[\mathbf{I}(\mathsf{f})]_{}$ is a well defined arrow from $\mathbf{K}(\mathsf{B},\mathsf{S},\sigma)$ to $\mathbf{K}(\mathsf{C},\mathsf{T}, \psi)$ in $\peff/(\mathsf{A},\mathsf{R})$.

Finally,
$\mathbf{K}$ is a well defined full and faithful functor from $\mcol(\mathsf{A},\mathsf{R}')$ to $\peff/(\mathsf{A},\mathsf{R})$ which preserves finite limits and  is surjective on objects.
\end{theorem}
\begin{proof} The proof is a direct verification and
is completely analogous to that in \cite{m09} (between slice categories
of the elementary quotient completions and the category of extensional dependent collections). 
\end{proof}

\subsubsection{The notion of family of sets}

\begin{definition}Let $\mathsf{b}$ be an arrow of $\peff$ with codomain $(\mathsf{A},\mathsf{R})$. An  arrow $\mathsf{b}$ is {\em family of sets depending on}  $(\mathsf{A},\mathsf{R})$ or  is a \emph{small map} (in the spirit of algebraic set theory) if and only if there exists an $\mathsf{R}'$ with $\mathsf{R}=[\mathsf{R}']$ and an object $(\mathsf{B},\mathsf{S},\sigma)$ of $\mcol(\mathsf{A},\mathsf{R}')$ such that
\begin{enumerate}
\item $\mathsf{b}$ is isomorphic to $\mathbf{K}(\mathsf{B},\mathsf{S},\sigma)$ in $\peff/(\mathsf{A},\mathsf{R})$;
\item $\mathsf{B}$ is an object of $\mathbf{Set}^{r}(\mathsf{A})$
\item $\mathsf{S}$ is an object of $\overline{\mathbf{Prop}^{r}_{s}}(\Sigma(\mathsf{A}\times \mathsf{A},\mathsf{R}'))$.
\end{enumerate}
Small maps of $\peff$ form a full subcategory of $\peff^{\rightarrow}$ which we denote with $\peff^{\rightarrow}_{small}$.
\end{definition}
\begin{lemma}
The restriction of the codomain fibration $\mathbf{cod}:\peff^{\rightarrow}\rightarrow \peff$ to $\peff^{\rightarrow}_{small}$ is a fibration.
\end{lemma}
\begin{proof} 
The class of small maps  is pullback stable in $\peff$ as one can work with codings in the subcategory of $\mcol(\mathsf{A},\mathsf{R}')$ of those objects which give rise to small maps in $\peff$ and use the fact that $\mathbf{K}$ preserves finite limits.
\end{proof}

\begin{remark} Notice that if one consider the full subcategory $\peff_{s}$ of $\peff$ whose objects are those objects $(\mathsf{A},\mathsf{R})$ such that the unique arrow from them to the terminal object $\mathbf{1}$ in $\peff$ is small, then $\mprops$ restricted to $\peff_{s}$ is equivalent to $\mathbf{Sub}_{\peff_{s}}$. 
In fact one can easily see that an object of $\peff_{s}$ is an object of $\peff$ which is isomorphic to an object $(\mathsf{A},\mathsf{R})$ such that
\begin{enumerate}
\item $\mathsf{A}$ has the form $\{x|\,x\,\overline{\varepsilon }\,\mathbf{n}\}$ where $\tar\vdash \mathbf{n}\,\varepsilon\,\mathsf{U}_{\mathsf{S}}$
\item $\mathsf{R}$ is an object of $\overline{\mathbf{Prop}^{r}_{s}}(\mathsf{A}\times \mathsf{A})$.
\end{enumerate}

One can notice that $\peff_{s}$ is equivalent to the exact completion on a lex category of the full subcategory $\mathcal{S}_{r}$ of $\mathcal{C}_{r}$ whose objects are those objects of the form $\{x|\,x\,\overline{\varepsilon }\,\mathbf{n}\}$ where $\tar\vdash \mathbf{n}\,\varepsilon\,\mathsf{U}_{\mathsf{S}}$. In fact the subobject doctrine on $\mathcal{S}_{r}$ is equivalent to the doctrine $\overline{\mathbf{Prop}_{s}^{r}}$ restricted to $\mathcal{S}_{r}$.

Moreover $\overline{\mathbf{Sub}}_{\peff_{s}}$ is equivalent to the restriction of $\mprops$ to $\peff_{s}$.

If $[\mathsf{j}]_{\simeq}:(\mathsf{B},\mathsf{S})\rightarrow (\mathsf{A},\mathsf{R})$ is a monomorphism in $\peff_{s}$, then $\mathsf{p}_{1}:\mathsf{A}\times \mathsf{B}\rightarrow \mathsf{A}$ is representable (see \ref{repress}) and in particular $\exists_{\mathsf{p}_{1}}((\overline{\mathbf{Prop}_{s}})_{\mathsf{id}\times \mathsf{j}}(\mathsf{R}))$
is isomorphic to an element of $\overline{\mathbf{Prop}_{s}}(\mathsf{A})$ which is an element of $\mprops(\mathsf{A},\mathsf{R})$ too.

Vice versa, if $\mathsf{P}$ is an element of $\mprops(\mathsf{A},\mathsf{R})$, then 
$$[\mathsf{p}_{1}^{\Sigma}]_{\simeq}:(\Sigma(\mathsf{A},\mathsf{P}),(\overline{\mathbf{Prop}_{s}})_{\mathsf{p}_{1}^{\Sigma}\times \mathsf{p}_{1}^{\Sigma}}(\mathsf{R}))\rightarrow (\mathsf{A},\mathsf{R})$$
can easily be shown to be a monomorphism in $\peff_{s}$.

\end{remark}

Working with codings one can also prove the following lemma.
\begin{lemma} If $(\mathsf{A},\mathsf{R})$ is an object of $\peff_{s}$ and $[\mathsf{f}]_{\simeq}$ is a small map with codomain $(\mathsf{A},\mathsf{R})$ in $\peff$, then the dependent product of $[\mathsf{f}]_{\simeq}$ over $(\mathsf{A},\mathsf{R})$ is an object of $\peff_{s}$.
\end{lemma}

\subsection{The Predicative Effective p-Topos}
Now we are ready to give the definition of our predicative variant of Hyland's Effective Topos:
\begin{definition}
We call {\em Predicative Effective p-Topos} the $5$-uple
$$(\peff,\mathbf{cod},\mathbf{cod}|_{\peff^{\rightarrow}_{small}},\mprop,\mprops)$$
\end{definition}
The name {\it p-topos} stands for a {\it predicative generalization of an elementary topos}.

If we denote with $\mathsf{Gr}(\mathbf{C})$ the Grothendieck construction of a fibration assigned to a pseudofunctor $\mathbf{C}:\mathcal{C}^{op}\rightarrow \mathbf{Cat}$ (see \cite{jacobbook}) we can define functors 
\begin{enumerate}
\item $\Delta:\mathcal{C}_{r}\rightarrow \peff$
\item $\Delta_{c}:\mathsf{Gr}(\mathbf{Col}^{r})\rightarrow\peff^{\rightarrow}$
\item $\Delta_{s}:\mathsf{Gr}(\mathbf{Set}^{r})\rightarrow\peff^{\rightarrow}_{small}$
\item $\Delta_{p}:\mathsf{Gr}(\overline{\mathbf{Prop}^{r}})\rightarrow\mathsf{Gr}(\mprop)$
\item $\Delta_{ps}:\mathsf{Gr}(\overline{\mathbf{Prop}_{s}^{r}})\rightarrow\mathsf{Gr}(\mprops)$
\end{enumerate}
as follows:
\begin{enumerate}
\item $\Delta(\mathsf{A}):=(\mathsf{A},\exists_{\Delta_{\mathsf{A}}}(\top_{\mathsf{A}}))$ and $\Delta(\mathsf{A})(\mathsf{f}):=[\mathsf{f}]_{\simeq}$
\item $\Delta_{c}(\mathsf{A},\mathsf{B}):=[\mathsf{p}_{1}^{\Sigma}]_{\simeq}: (\Sigma(\mathsf{A},\mathsf{B}),\exists_{\mathbf{I}(\Delta_{\Sigma(\mathsf{A},\mathsf{B})})}(\top_{\Sigma(\mathsf{A},\mathsf{B})}))\rightarrow(\mathsf{A},\exists_{\Delta_{\mathsf{A}}}(\top_{\mathsf{A}}))$
\item[] If $(\mathsf{f},\mathsf{g}):(\mathsf{A},\mathsf{B})\rightarrow (\mathsf{A}',\mathsf{B}')$ in $\mathsf{Gr}(\mathsf{Col}^{r})$, then $\Delta_{c}((\mathsf{f},\mathsf{g})):=[\Sigma(\mathsf{f},\mathsf{B}')\circ \mathbf{I}_{\mathsf{A}(\mathsf{g})}]_{\simeq}$ 
\item $\Delta_{c}$ restricts from $\mathsf{Gr}(\mathbf{Set}^{r})$ to $\peff^{\rightarrow}_{small}$ giving rise to $\Delta_{s}$
\item $\Delta_{p}(\mathsf{A},\mathsf{P}):=((\mathsf{A},\exists_{\Delta_{\mathsf{A}}}(\top_{\mathsf{A}})),\mathsf{P})$
\item[] $\Delta_{p}(\mathsf{f},*):=([\mathsf{f}]_{\simeq},*)$
\item $\Delta_{p}$ restricts from $\mathsf{Gr}(\overline{\mathbf{Prop}_{s}^{r}})$ to $\mathsf{Gr}(\mprops)$ giving rise to $\Delta_{ps}$
\end{enumerate}

The following commutative diagram shows how  the  Effective Kleene \mf-tripos embeds into the Predicative Effective p-Topos 
{\small \def\objectstyle{\scriptstyle}
\def\labelstyle{\scriptstyle}
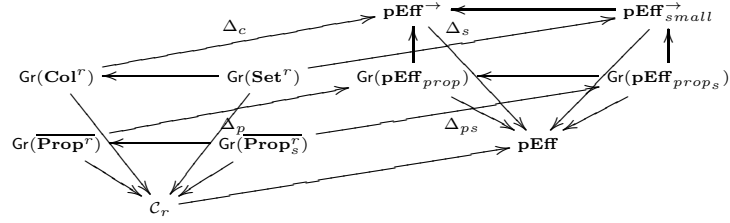
{
$$\xymatrix@-1pc{											&					&														&\peff^{\rightarrow}\ar[ddr]	&	&\peff^{\rightarrow}_{small}\ar[ll]\ar[ldd]\\
\mathsf{Gr}(\mathbf{Col}^{r})\ar[ddr]\ar[rrru]^-{\Delta_{c}}			&					&\mathsf{Gr}(\mathbf{Set}^{r})\ar[ll]\ar[ddl]\ar[rrru]^-{\Delta_{s}}						&\mathsf{Gr}(\mprop)\ar[rd]\ar[u]			&	&\mathsf{Gr}(\mprops)\ar[ll]\ar[ld]\ar[u]\\
\mathsf{Gr}(\overline{\mathbf{Prop}^{r}})\ar[dr]\ar[rrru]_-{\Delta_{p}}	&					&\mathsf{Gr}(\overline{\mathbf{Prop}_{s}^{r}})\ar[ll]\ar[dl]\ar[rrru]_-{\Delta_{ps}}		& &\peff		&\\
											&\mathcal{C}_{r}\ar[rrru]		&														& &			&\\
}$$}}

\section{Embedding of $\peff$ in  Hyland's Effective Topos}
Here we are going to show that the construction of $\peff$ performed on
the subcategory of recursive functions of $\eff$ gives rise
to a full subcategory of the Effective Topos $\eff$ whose embedding
preserves the list-arithmetic locally cartesian closed pretopos structure.

Before proceeding we recall that
the Effective Topos (see \cite{eff}) can be presented in three ways:
\begin{enumerate}
\item as the result of the tripos-to-topos (see \cite{tripos}) construction from a tripos $$\mathbf{P}:\mathbf{Set}^{op}\rightarrow \mathbf{preHeyt}$$
\item as the exact on regular completion $(\mathbf{Asm})_{ex/reg}$ of the category $\mathbf{Asm}$ of assemblies (see \cite{CFS});
\item as the exact on lex completion $(\mathbf{pAsm})_{ex/lex}$ of the category $\mathbf{pAsm}$ of partitioned assemblies (see \cite{MRC}).
\end{enumerate}
Let us recall the definition of these categories.
\begin{definition}
The category $\mathbf{Asm}$ is the category whose objects are pairs $(A,P)$ where $A$ is a set and $P:A\rightarrow \mathcal{P}(\mathbb{N})$ is a function such that for every $a\in A$, $P(a)\neq \emptyset$ and whose arrows from such an object $(A,P)$ to another $(B,Q)$ is a function $f:A\rightarrow B$ such that $P\sqsubseteq_{A}\mathbf{p}_{f}(Q)$.

The category $\mathbf{pAsm}$ of partitioned assemblies is the full subcategory of $\mathbf{Asm}$ whose objects $(A,P)$ satisfy the following additional property: for every $a\in A$, $P(a)$ is a singleton, namely  $\#P(a)=1$.

The category $\mathbf{Mod}$ of modest sets is the full subcategory of $\mathbf{Asm}$ whose objects $(A,P)$ satisfy the following additional property: for every $a,b\in A$, if it holds that $P(a)\cap P(b)\neq \emptyset$, then $a=b$.
\end{definition}

We now consider the category of recursive definitions and we recall a well-known characterization (see also \cite{Discrete}, \cite{Car}).
\begin{definition}
The category $\mathbf{Rec}$ is defined as follows:
\begin{enumerate}
\item the objects of $\mathbf{Rec}$ are subsets $A\subseteq\mathbb{N}$;
\item an arrow from $A$ to $B$ is a function $f:A\rightarrow B$ such that there exists a (partial) recursive function $\overline{f}:\mathbb{N}\rightharpoonup \mathbb{N}$ such that $A\subseteq dom(\overline{f})$ and $\overline{f}|_{A}=f$;
\item composition is given by composition of functions and identities are defined as identity functions.
\end{enumerate}
\end{definition}

\begin{lemma}
The category $\mathbf{Rec}$ is equivalent to the intersection of the categories $\mathbf{pAsm}$ and $\mathbf{Mod}$.
\end{lemma}
\begin{proof}
First notice that an assembly $(A,P)$ is partitioned and modest at the same time if and only if there exists an injective function $j:A\rightarrow \mathbb{N}$ such that $P(a)=\{j(a)\}$. These assemblies form a full subcategory of $\mathbf{Asm}$ which is equivalent to $\mathbf{Rec}$. In fact one can consider the functor 
$$\mathbf{Sgl}:\mathbf{Rec}\rightarrow \mathbf{Asm}$$
sending each object $A$ in $\mathbf{Rec}$ to the assembly $(A, a\mapsto \{a\})$ and every function $f$ to itself as an arrow between assemblies.
The functor $\mathbf{Sgl}$ factorizes both through the category of partioned assemblies $\mathbf{pAsm}$ (via a functor $\mathbf{Sgl}_{p}$) and via that of modest sets $\mathbf{Mod}$ and make $\mathbf{Rec}$ a full subcategory of $\mathbf{Asm}$ (every map in $\mathbf{Asm}$ from $\mathbf{Sgl}(A)$ to $\mathbf{Sgl}(B)$ is the restriction of a partial recursive function).
$$\xymatrix{
\mathbf{pAsm}\ar[r]								&\mathbf{Asm}\\
\mathbf{Rec}\ar[ru]^{\mathbf{Sgl}}\ar[u]^{\mathbf{Sgl}_{p}}\ar[r]	&\mathbf{Mod}\ar[u]\\
}$$
This means in particular that there exists a functor $\mathbf{Sgl}':\mathbf{Rec}\rightarrow \mathbf{pAsm}\cap\mathbf{Mod}$.
In order to obtain an equivalence one can just consider the functor 
$$\mathbf{Frg}:\mathbf{pAsm}\cap \mathbf{Mod}\rightarrow \mathbf{Rec}$$ defined as follows:
\begin{enumerate} 
\item an object $(A,P)$ is sent to $\cup_{a\in A}P(a)$;
\item an arrow $f:(A,P)\rightarrow (B,Q)$ is sent to the function $x\mapsto \{r\}(x)$ where $r$ is a natural number for which for all $a\in A$ and $x\in P(a)$, $\{r\}(x)$ is defined and is in $Q(f(a))$.
\end{enumerate}  
gives rise to the equivalence we were looking for.
\end{proof}


One can now prove the following remarkable property:
\begin{theorem}
\label{sub}
The functor of weak subobjects over $\mathbf{Rec}$ is equivalent
to that of weak subobjects over partitioned assemblies and also to that of subobjects over assemblies, both
composed with the corresponding embeddings of $\mathbf{Rec}$ into them:
$$\mathbf{wSub}_{\mathbf{Rec}}\cong\mathbf{wSub}_{\mathbf{pAsm}}\circ \mathbf{Sgl}_{p}\cong\mathbf{Sub}_{\mathbf{Asm}}\circ \mathbf{Sgl}$$
\end{theorem}
\begin{proof}
The equivalence between $\mathbf{wSub}_{\mathbf{Rec}}$ and $\mathbf{wSub}_{\mathbf{pAsm}}\circ \mathbf{Sgl}_{p}$ is obtained via the natural transformations $\mu:\mathbf{wSub}_{\mathbf{Rec}}\rightarrow\mathbf{wSub}_{\mathbf{pAsm}}\circ \mathbf{Sgl}_{p}$ and \\$\nu:\mathbf{wSub}_{\mathbf{pAsm}}\circ \mathbf{Sgl}_{p}\rightarrow \mathbf{wSub}_{\mathbf{Rec}}$ defined as follows, for every object $A\subseteq \mathbb{N}$ of $\mathbf{Rec}$
\begin{enumerate}
\item for every $f:B\rightarrow A$ in $\mathbf{Rec}$, $\mu_{A}([f]):=[\mathbf{Sgl}_{p}(f)]$;
\item for every $f:(B,P)\rightarrow \mathbf{Sgl}_{p}(A)$ in $\mathbf{pAsm}$, $\nu_{A}([f]):=[f']$, where\\ $f':\bigcup_{b\in B}P(b)\rightarrow A$ is the recursive function (whose existence and uniqueness follows from the definition of $\mathbf{pAsm}$) such that for every $b\in B$ and every $n\in P(b)$, $f'(n)=f(b)$.
\end{enumerate}

The equivalence between $\mathbf{wSub}_{\mathbf{Rec}}$ and $\mathbf{Sub}_{\mathbf{Asm}}\circ \mathbf{Sgl}$ is obtained via the natural transformations $\mu':\mathbf{wSub}_{\mathbf{Rec}}\rightarrow\mathbf{Sub}_{\mathbf{Asm}}\circ \mathbf{Sgl}$ and \\$\nu':\mathbf{wSub}_{\mathbf{Asm}}\circ \mathbf{Sgl}\rightarrow \mathbf{wSub}_{\mathbf{Rec}}$ defined as follows, for every object $A\subseteq \mathbb{N}$ of $\mathbf{Rec}$:
\begin{enumerate}
\item for every $f:B\rightarrow A$ in $\mathbf{Rec}$, $\mu'_{A}([f]):=[j]$ where $j$ is the embedding of $(f(B),f^{-1})$ into $\mathbf{Sgl}(A)$;
\item for every mono $j:(B,P)\rightarrow \mathbf{Sgl}(A)$ in $\mathbf{Asm}$, $\nu'_{A}([j]):=[j']$, where\\ $j':\bigcup_{b\in B}P(b)\rightarrow A$ is the recursive function such that for every $b\in B$ and every $n\in P(b)$, $j'(n)=j(b)$; the existence of $j'$ is guaranteed by the fact that $j$ is mono in $\mathbf{Asm}$ and hence injective as a set theoretical function.
\end{enumerate}
\end{proof}

Now we mimick the construction of $\peff$ over the category $\mathbf{Rec}$
by replacing the doctrine of realized propositions with the doctrine of weak subobjects
in $\mathbf{Rec}$:
\begin{definition}
Let $\mathcal{Q}_{\mathbf{wSub}_{\mathbf{Rec}}}$ be the elementary quotient completion in \cite{qu12} applied to the doctrine $\mathbf{wSub}_{\mathbf{Rec}}$.
\end{definition}

Note that the category just defined is actually an exact completion
as observed in \cite{qu12,MR16}:
\begin{theorem}
The category $\mathcal{Q}_{\mathbf{wSub}_{\mathbf{Rec}}}$ is equivalent to the exact on lex completion of the category of recursive functions $\mathbf{Rec}$.
\end{theorem}

From theorem~\ref{sub} it immediately follows the following corollary.
\begin{corollary}\label{corcor}
$\mathcal{Q}_{\mathbf{wSub}_{\mathbf{Rec}}}$ is a full subcategory of the effective topos:
$$\mathcal{Q}_{\mathbf{wSub}_{\mathbf{Rec}}}\cong (\mathbf{Rec})_{ex/lex}\hookrightarrow (\mathbf{pAsm})_{ex/lex}\cong (\mathbf{Asm})_{ex/reg}\cong \eff$$\end{corollary}

\begin{remark}
Notice that the category $(\mathbf{Rec})_{ex/lex}$ is clearly equivalent to the full subcategory of $\mathbf{Eff}$ whose objects are quotients in $\mathbf{Eff}$ of objects of the form $\mathbf{Sgl}(A)$ for some object $A$ of $\mathbf{Rec}$.
\end{remark}

The theory $\tar$ has a \emph{standard} model in set theory: one can in fact consider the set of natural numbers with the interpretation of Peano arithmetic and interpret the fixpoint formulas in $\tar$  using transfinite induction (till the first uncountable ordinal $\omega_{1}$). 

\begin{theorem}
The standard interpretation of $\tar$ in $\mathsf{ZFC}$ gives rise to a functor 
$$\mathbf{Int}:\mathcal{C}_{r}\rightarrow \mathbf{Rec}$$
sending each realized collection $\mathsf{A}$ of $\tar$ to the subset of $\mathbb{N}$ given by the interpretation of the formula $x\,\varepsilon\, \mathsf{A}$ and sending each arrow $[\mathbf{n}]_{\approx}$ to the recursive function encoded by the corresponding natural number $n$.
The interpretation also gives rise to a natural transformation preserving connectives and quantifiers
$$\eta:\overline{\mathbf{Prop}^{r}}\rightarrow \mathbf{wSub}_{\mathbf{Rec}}\circ \mathbf{Int}$$
which is defined analogously using theorem \ref{ws}.
In particular this allows one to define a functor $\mathbf{J}:\peff\rightarrow (\mathbf{Rec})_{ex/lex}$
and then also the functor 
$$\mathbf{I}:\peff\rightarrow \eff$$
obtained by composing $\mathbf{J}$ with the embedding in corollary \ref{corcor}.
The functor $\mathbf{I}$ preserves finite limits, exponentials, lists, finite coproducts and quotients. 
\end{theorem}
\begin{proof}
The first part is immediate and follows from how the above interpretation is defined. The functor $\mathbf{I}$ preserves finite limits, exponentials, lists, finite coproducts and quotients, because $\mathbf{J}$ and the embedding from $(\mathbf{Rec})_{ex/lex}$ to $\eff$ preserve them from direct verifications.

\end{proof}

\begin{remark}
One could also compare the construction of $\peff^{\rightarrow}_{small}$ to the construction of an effective predicative category of small maps $(\eff, \overline{S})$ from a category of assemblies with small maps $(\mathcal{A}sm_{\mathcal{E}},\mathcal{S}_{\mathcal{E}} )$ with respect to a base predicative category with small maps $(\mathcal{E},\mathcal{S})$ in \cite{VDB}. 
However the construction in \cite{VDB} is a predicative rendering of the exact on regular completion from which one can obtain the Effective Topos from the category of its assemblies, while our approach is a strictly predicative rendering of an exact on lex completion preformed on a full subcategory of the category of partitioned assembies.
Morever the properties of the class of small maps obtained by considering $\mset$ are much weaker than those validated by the small maps in \cite{VDB}. In fact the class of small maps in \cite{VDB} gives rise to an internal model of $\mathsf{CZF}$ and this is not possible in our case, as we know that the proof-theoretical strength of $\tar$ is strictly weaker than that of $\mathsf{CZF}$.
\end{remark}




 

\subsection*{Acknowledgments}The authors acknowledge Steve Awodey, Martin Hyland, Pino Rosolini and Thomas Streicher for very helpful discussions and suggestions and their colleagues Francesco Ciraulo and Giovanni Sambin for supporting this line of research.

\bibliographystyle{plain}		
\bibliography{bibliopsp}

\end{document}